%% file: Weak_Arithmetic.tex
\documentclass[11pt]{article}
\usepackage{geometry}
\usepackage{bbm}
\usepackage{amsmath,amssymb}
\usepackage[all]{xy}
\usepackage{mathrsfs}
\usepackage[pdftex,bookmarks=true]{hyperref}

\usepackage{graphicx}
\usepackage{amsthm}
\usepackage{helvet}         
\usepackage{courier}        
\usepackage{type1cm} 

\usepackage{epigraph}
\usepackage[nottoc]{tocbibind}
\usepackage{makeidx}

\usepackage{verbatim}
\makeatletter
\newdimen\rh@wd
\newdimen\rh@hta
\newdimen\rh@htb
\newbox\rh@box
\def\rh@measure#1{\setbox\rh@box=\hbox{$#1$}\rh@wd=\wd\rh@box \rh@hta=\ht\rh@box}

\def\widecheck#1{\rh@measure{#1}%
  \setbox\rh@box=\hbox{$\widehat{\vrule height \rh@hta width\z@ \kern\rh@wd}$}%
  \rh@htb=\ht\rh@box \advance\rh@htb\rh@hta \advance\rh@htb\p@
  \ooalign{$\vrule height \ht\rh@box width\z@ #1$\cr
           \raise\rh@htb\hbox{\scalebox{1}[-1]{\box\rh@box}}\cr}}
\makeatother

\theoremstyle{plain}
\newtheorem{Theo}{Theorem}[subsection]

\theoremstyle{plain}
\newtheorem{Lem}[Theo]{Lemma}

\theoremstyle{plain}
\newtheorem{Prop}[Theo]{Proposition}

\theoremstyle{plain}
\newtheorem{Def}[Theo]{Definition}

\theoremstyle{plain}
\newtheorem{Cor}[Theo]{Corollary}

\theoremstyle{remark}
\newtheorem{Rque}[Theo]{Remark}

\DeclareMathOperator{\Div}{div}

\DeclareMathOperator{\Pic}{Pic}

\DeclareMathOperator{\ord}{ord}

\DeclareMathOperator{\End}{End}
\DeclareMathOperator{\CH}{CH}

\DeclareMathOperator{\CHw}{\widecheck{CH}}
\DeclareMathOperator{\Kw}{\widecheck{K}}
\DeclareMathOperator{\Gw}{\widecheck{G}}
\DeclareMathOperator{\bra}{\langle}
\DeclareMathOperator{\ket}{\rangle}
\DeclareMathOperator{\ch}{ch}

\DeclareMathOperator{\cht}{\wt{ch}}

\DeclareMathOperator{\Td}{Td}
\DeclareMathOperator{\c1}{\widehat{c_1}}

\DeclareMathOperator{\codim}{codim}
\DeclareMathOperator{\Proj}{Proj}

\DeclareMathOperator{\Spec}{Spec}

\DeclareMathOperator{\cone}{cone}
\DeclareMathOperator{\DIM}{DIM}
\DeclareMathOperator{\SECT}{SECT}
\DeclareMathOperator{\FGL}{FGL}

\DeclareMathOperator{\bc}{bc}

\DeclareMathOperator{\Id}{id}
\DeclareMathOperator{\im}{im}
\DeclareMathOperator{\Hom}{Hom}

\DeclareMathOperator{\Trs}{tr_s}
\DeclareMathOperator{\cob}{\widecheck{\Omega}}

\DeclareMathOperator{\V}{V^b}
\DeclareMathOperator{\Vh}{\ovl{V}^b}

\newcommand{\na}{_{n.a}}
\newcommand{\Kh}{\widehat{K}}
\newcommand{\Gh}{\widehat{G}}
\newcommand{\ds}[1]{\xrightarrow{#1}}

\newcommand{\cobna}{\Omega_{n.a}}

\newcommand{\mcl}{\mathcal}
\newcommand{\mbb}{\mathbb}

\newcommand{\mfr}{\mathfrak}

\newcommand{\ovl}{\overline}
\newcommand{\wh}{\widehat}
\newcommand{\wt}{\widetilde}
\newcommand{\fii}{\varphi}

\newcommand{\KA}{\mcl{KA}}

\newcommand{\M}{\ovl{M}}
\newcommand{\E}{\ovl{E}}
\renewcommand{\H}{\wh{H}}
\newcommand{\N}{\ovl{N}}
\newcommand{\X}{\ovl{X}}
\newcommand{\Z}{\ovl{Z}}
\newcommand{\Y}{\ovl{Y}}

\renewcommand{\t}{\mathbf{t}}
\renewcommand{\S}{\ovl{S}}
\newcommand{\g}{\mfr{g}}
\newcommand{\h}{\mfr h}
\renewcommand{\O}{\mcl{O}}
\renewcommand{\P}{\mbb{P}}
\newcommand{\Dt}{\wt D^{\bullet,\bullet}_{\mbb R}}
\newcommand{\At}{\wt A^{\bullet,\bullet}_{\mbb R}}
\newcommand{\Att}{\wt A^{\bullet,\bullet}_{[\t]}}
\newcommand{\DL}{\wt D^{\bullet,\bullet}_{\Lh}}
\newcommand{\Dtp}[1]{\wt D^{#1,#1}_{\mbb R}}
\newcommand{\DLp}[1]{\wt D_{\Lh, #1}}
\newcommand{\Dtt}{\wt D^{\bullet,\bullet}_{[\t]}}
\renewcommand{\d}{\partial}
\newcommand{\db}{\ovl{\partial}}
\newcommand{\Lh}{\wh{\mbb L}}
\renewcommand{\L}{\ovl{L}}
\renewcommand{\div}{div}

\title{Weak Arithmetic Cobordism.}
\author{Aur\'{e}lien Rodriguez}
\makeindex

\begin{document}
\maketitle
\abstract{In the early 2000's, Levine and Morel built an algebraic cobordism theory, extending to the case of arbitrary algebraic varieties over any field the construction and properties of the complex cobordism ring studied by Milnor and Quillen. We will show how we can refine their construction to build a weak arithmetic cobordism group in the context of Arakelov geometry. The general strategy is to define the notion of homological theory of arithmetic type, encapsulating the common properties of arithmetic $K$-theory and arithmetic Chow groups, and to build a universal such theory.
}
\newpage

\input{Intro.tex}
\tableofcontents
\newpage
\section{Some metric notions}
In this section we recall some results obtained by Burg\'os, Freixas and Litcanu in a series of papers, \cite{BL, BFL, BFL1, BFL2}, which will enable us, among other things, to have a good notion of hermitian sheaf i.e a coherent sheaf on a proper, smooth, complex variety, equipped with a hermitian structure generalizing the case of locally free sheaves.
\input{Part1.tex}

\newpage
\section{Weak Arithmetic Theories}
\input{WAT.tex}

\newpage

\section{Weak Arithmetic Cobordism}

\input{WAC.tex}

\bibliographystyle{alpha-fr}
\bibliography{Biblio}

\end{document}

%% file: Intro.tex
\subsection*{Introduction}

The general purpose of Arakelov geometry is to extend classical algebraic geometry over a field, and more precisely intersection theory, to objects over $\Spec \mbb Z$, with the idea to use these new tools to extract arithmetic information. As is well-known, a good Arakelov theory for arithmetic surfaces is sufficient to prove the Mordell conjecture. Satisfying analogs of Chow groups, and $K$-theory groups have been defined and studied by Gillet-Soul\'{e}, Burgos, Faltings and others.\par
In the geometric context Chow theory and $K$-theory are the prototype of what Levine and Morel call homology theories for algebraic varieties, and these two authors have incorporated an essential new theory to the picture, namely \emph{Algebraic Cobordism}. It was a natural question to ask what the analog of that theory in the Arakelov world would be and this paper is a first step towards an answer to that question.\par
The general approach that we will follow is what we may call the functorial approach to cobordism. Namely we want to define a good notion of homology theory (in the arithmetic case), such that arithmetic Chow theory and arithmetic $K$-theory would be examples of such theories, and then we want to construct a universal such theory.\par
Let's see how we can achieve this in more details.\\

In the first section, we introduce some specializations of certain notions defined by Burgos, Freixas and Litcanu, mostly coming from \cite{BFL}. The notion of metrized sheaf, and secondary forms associated to it was already existent in the literature for instance in \cite{GSinv}, although we give a slightly different version of it, using the language of \cite{BFL}, notably the notion of meager complex and of quasi-isometry. The reader familiar with \cite{BFL} won't find anything new, although some of our proofs are different. This language will make it easy for us to introduce the different notions of weak arithmetic theories.\par
The second section contains the review of what we call \emph{weak arithmetic theories}. In order to mimic the functorial construction of the cobordism group of Levine and Morel, we need to have good functorial properties for our arithmetic objects. Therefore we introduce a weak version of arithmetic Chow groups $\CHw(\X)$ for an arithmetic variety $\X$, that is an algebraic variety together with a K\"ahler metric on its tangent bundle invariant by complex conjugation. Those groups were introduced by Zha, Burgos and Moriwaki independently, we prove that these groups are the prototype of what we call an \emph{oriented Borel Moore functor of arithmetic type}.\par
We then review the theory of Bott-Chern singular currents, and of the Analytic torsion forms both essentially do to Bismut and his collaborators on the one hand, and to Burgos, Freixas and Litacnu on the second hand, we make heavy use of the language defined in \cite{BFL} which makes the analogy between those two objects clear. We then introduce the notion of weak arithmetic $\Kw$-theory and prove that it is also an oriented Borel Moore functor of arithmetic type.\par
The parallel between Chow and $K$-groups show the particular place occupied by the Todd form. It appears that both those theories have a Todd form, but the Todd form in the case of arithmetic Chow theory is just 1, therefore it disappears from the classical presentation of the theory and the usual Todd form appears to be a specificity of $\Kh$-theory.\par 
In the third section we proceed to construct a universal Borel Moore functor of arithmetic type. For this, we will need a universal Todd form, for various reasons we define a universal inverse Todd form which we denote $\g$, we also introduce secondary forms associated to it.\par
This $\g$ class will enable us to construct a \emph{universal Bott-Chern singular current} for the immersion of a smooth divisor. The crucial observation is that in the case of Chow theory we have the following relation\footnote{under a technical meager condition} relating the first Chern class and the direct image via the immersion of a divisor  $$i_*(1_{\Z})=\c1(L)(1_{\X})+a(\log\|s\|^2)$$ whereas in $\Kw$ theory this relations become $$i_*(1_{\Z})=\c1(L)(1_{\X})+a(\log\|s\|^2 \Td(\L)^{-1})$$
It is therefore natural to replace the $\Td^{-1}$ form by the most general form $\g(\L)$.\par
The formal group law giving the action of $\c1(\L\otimes \M)$ in function of $\c1(\L)$ and $\c1(\M)$ imposes relations between the coefficient of $\g$ and those of $F_\mbb L$ the universal group law on the Lazard ring. We show that this enables to relate $\g$ to the universal logarithmic class defined by Hirzebruch. In other words, the formal group law imposes what the Todd form should be and \emph{vice versa}. This sheds lights on various constructions of classical Arakelov theory and especially explains why it is possible to define covariant arithmetic Chow groups on the category of algebraic varieties but that it is only possible to define covariant arithmetic $\Kw$ groups on arithmetic varieties, the difference being explained by the triviality of the Todd form in Chow theory but not in $\Kw$-theory.\par
We then proceed to a technical remark about projective Borel-Moore functor, essentially destined to prove that there are no surprises in passing from the quasi-projective to the projective case for Borel-Moore functors of geometric type.\par
We can now prove the fundamental exact sequence
$$\DL(X)\ds a \cob(X) \ds \zeta \Omega(X)\to 0$$
which fulfills the goal of building arithmetic cobordism as an extension of the geometric cobordism by the space of currents and which lies at the heart of our theory.\par
This exact sequence should be thought as the analog of the localization exact sequence in classical algebraic cobordism, as to extend usual geometric arguments, involving this localization, to Arakelov geometry, one usually uses this exact sequence for arithmetic Chow or $K$-groups. In an upcoming paper we will show and we can extend this exact sequence on the left, by relating it to "higher cobordism" groups, $\text{MGL}^{2n-1,n}$ from motivic homotopy theory.\par
We then prove that an analog of the star product lies in the groups $\cob$ that gives back the star-product of Gillet-Soul\'{e} when mapped to arithmetic Chow theory, and we prove a universal anomaly formula that, here again, explains the differences between arithmetic Chow and $K$-theory (the anomaly term being 0 in Chow theory).\par
This enables us to compute the structure of $\cob(k)$. It can be given the structure of a commutative ring because, over a point, strong object and weak objects should coincide. In doing so we prove some kind of universal Hirzebruch-Riemann-Roch formula for the $\g$ class, which is a reflexion of the canoncial isomorphisms $$MU_{2\bullet}\simeq \mbb L_{\bullet}\simeq \Omega_\bullet(k)$$
this formula is the key fact that ensures that the groups $\cob(X)$ have a natural $\cob(k)$-module structure. The explicit description that we then give of $\cob(k)$ seems to fit perfectly in the general framework of Arakelov theory.\par
Finally we prove the existence of different arrows form $\cob$ to $\CHw$ and $\Kw$ and make explicit the notion of Borel-Moore functor of arithmetic type, rouding up or prospect of integrating in Arakelov geometry the general situation of cobordism with respect to other oriented (co)homology theories.

\subsection*{Acknowledgments}
I'd like to thank my thesis advisor Vincent Maillot, for his innumerable comments, suggestions and helpful insights. I'm also greatly thankful to C.Soul\'{e} for the very fruitful discussions that we had together, and M.Levine who helped me greatly in clarifying the presentation of this paper.
\subsection*{Notations and conventions}
Throughout all the paper $k$ will be a number field. If $X$ is a complex manifold we set $A^{p,p}_{\mbb R}(X)$ to be the set of smooth real forms over $X$ of $(p,p)$ type satisfying $F_\infty^*(w)=(-1)^pw$, the notation $D^{p,p}_{\mbb R}(X)$ will represent the space of real currents of $(p,p)$-type, satisfying $F_\infty^*(\eta)=(-1)^p\eta$, and $\wt{A}^{p,p}_{\mbb R}(X)$ will be $A^{p,p}_{\mbb R}(X)/(\im \d+\im \db)$, in the same way $D^{p,p}_{\mbb R}(X)/(\im \d+\im \db)$ is to be denoted $\Dtp p(X)$.\par
When $X$ is a complex quasi-projective variety, we will use the same notations to denote the corresponding objects over $X(\mbb C)$ seen as a complex manifold consisting of the disjoint union of the complex points of $X_{\sigma(\mbb C)}=X\times_{\sigma, k}\Spec \mbb C$ where $\sigma$ runs through the embeddings of $k$ in the complex numbers.\par
The suggestion to differentiate weak objects and strong objects by capping them with a "check" for the former and a "hat" for the latter had been made to me by C.Soul\'{e}.

%% file: Part1.tex
\subsection{Resolutions}
We review here a specialization of the general theory of metrized structures on derived categories (see \cite{BFL2}); as we only need to metrize sheaves, and not complexes of sheaves we present the general theory in this restrictive context.\par
In this section $X$ will denote a projective complex manifold (the complex points of a smooth projective variety over $\mbb C$). On such a variety, any coherent sheaf admits a finite resolution by vector bundles \cite{SGA6}.
\begin{Def} Let $\mcl F$ be a coherent sheaf over $X$, a metric structure (or sometimes just metric) on $\mcl F$, will consist in the datum of a (finite) resolution of $\mcl F$ by algebraic vector bundles, endowed with hermitian metrics.
\end{Def}
We will need some facts about complexes of vector bundles on a smooth manifold, that we recall here. Let us state some conventions, following \cite{BFL2}, we'll denote $\V(X)$ (resp. $\Vh(X)$) the category of bounded complexes of vector bundles (resp. hermitian vector bundles) over $X$; such a complex will be written homologically $$0\to E_n\ds{d_{n}}...\to E_1\ds{d_1}E_0\ds{d_0}0$$  
If we have a complex $E_\bullet$ resolving a coherent sheaf $\mcl F$, we will label the resolution $$\cdots\to E_1\to E_0 \to \mcl F$$
with $\mcl F$ in degree $-1$. We may also interpret such a resolution as a quasi-isomorphism between the complex $E_\bullet$ and the complex consisting of the single sheaf $\mcl F$ placed in degree 0. 
\begin{Prop}\label{toit}
Let $E_\bullet$, $F_\bullet$ and $G_\bullet$ be three complexes of vector bundles over $X$ such that we have a diagram of quasi-isomorphisms
$$\xymatrix{
     F_\bullet \ar[rd]_f & & E_\bullet \ar[ld]^g\\
     &G_\bullet& }$$

Then there exists a complex of vector bundles $H_\bullet$ such that we have a diagram, commuting up to homotopy $$\xymatrix{
   &H_\bullet \ar[rd] \ar[ld] &\\
     F_\bullet \ar[rd] & & E_\bullet \ar[ld]\\
     &G_\bullet& }$$
     where the top arrows are quasi-isomorphisms.
\end{Prop}
\begin{proof}
We have a map from $E_\bullet$ to $\cone(f)$ given by sending $x$ to $(0,g(x))$, we set $H=\cone(E,\cone(f))[-1]$, we have $H_n=E_n\oplus F_n\oplus G_{n-1}$ and we get diagram $$\xymatrix{
   &H_\bullet \ar[rd] \ar[ld] &\\
     F_\bullet \ar[rd] & & E_\bullet \ar[ld]\\
     &G_\bullet& }$$ as well as a homotopy $h:H_n\to G_{n-1}$ that makes the diagram commutes up to homotopy.
\end{proof}

We have the following lemma.
\begin{Lem}
Let $E_\bullet \to \mcl F$ be a locally free resolution of a coherent sheaf and $\mcl G\to \mcl F$ a morphism of sheaves, we can find a locally free resolution $H_\bullet$ of $\mcl G$ such that we have a commutative diagram $$\xymatrix{
    H_\bullet \ar[r]\ar[d] &\mcl G\ar[d]\\
      E_\bullet \ar[r] & \mcl F}$$
      Moreover if $\mcl G\to \mcl F$ is onto, then we can choose the $H_i\to E_i$ to be onto too, in this case we say that the top resolution dominates the bottom one.
\end{Lem}
\begin{proof}
We build the resolution $H_\bullet$ by induction. Set $\pi$ to be the morphism from $\mcl G$ to $\mcl F$, and let $K$ be the kernel of the map $\mcl G\oplus E_0  \xrightarrow{\pi-d_0} \mcl F$, as $K$ is a coherent sheaf, it is a quotient of some locally free sheaf, say $H_0$. We thus get a commutative diagram $$\xymatrix{H_0 \ar[r]\ar[d] & \mcl G\ar[r]\ar[d]&0\\E_0 \ar[r] & \mcl F\ar[r]&0}$$
On the other hand if $\pi$ is surjective, then by construction $\pi_0 :H_0\to E_0$ is too.\par
Now, assume that we have built the bundles $H_i$, and the morphisms $\pi_i$ for $i=0...n$.\par 
We have a surjection from $E_{n+1}$ to the kernel of $d_n$, however the kernel of $H_n\to H_{n-1}$ may not map surjectively onto $\ker(d_n)$, still we may replace $H_n$ by $H_n'=H_n\oplus F$ where $F$ is a locally free sheaf equipped with a surjection $\epsilon$ onto $\ker(d_n)$, we extend the map $H_{n}\to H_{n-1}$ to a map $H'_{n}\to H_{n-1}$, by sending $F$ to $0$.\par
Of course $H'_n$ is still mapped surjectively onto the kernel of $H_{n-1}\to H_{n-2}$, and the kernel of this map is $\ker(H_n\to H_{n-1})\oplus F$ which maps surjectively onto $\ker(d_n)$,  hence by re-iterating the procedure described in the previous paragraph we build a commutative diagram
$$\xymatrix{H_{n+1} \ar[r]\ar[d] & \mcl \ker(H'_n\to H_{n-1})\ar[r]\ar[d]&0\\E_{n+1} \ar[r] & \mcl \ker(d_{n-1})\ar[r]&0}$$
which enable us to construct $H_{n+1}$, and the morphism $\pi_{n+1}$. The arrow from $H_{n+1}$ to $E_{n+1}$ being surjective as the one from $\ker(H'_n)$ to $\ker(d_n)$ is, by construction\par
After a finite number of steps, we're left with the following commutative diagram
$$\xymatrix{0\ar[d]\ar[r]&\ker(H_{p+1}\to H_p)\ar[r]\ar[d]&H_{\bullet} \ar[r]\ar[d] & \mcl G\ar[r]\ar[d]&0\\0\ar[r]&0\ar[r]&E_\bullet \ar[r] & \mcl F\ar[r]&0}$$
As $\ker(H_{p+1}\to H_p)$ is a coherent sheaf, it will be enough to replace it by a resolution by locally free sheaves to prove the proposition.
\end{proof}

\subsection{Acyclic calculus of Burgos, Freixas and Litcanu}
Let $X$ be a complex algebraic variety (the $\mbb C$-valued points of an algebraic variety to be precise), we review here the theory of acyclic calculus developed in \cite{BFL2}.

Recall that in the general formalism of derived categories, a basic observation is that a resolution of a coherent sheaf should be defined up to a "roof", we mimic this situation in the metric case, the analog of quasi-isomorphisms will be called quasi-isometries (which is not exactly the terminology employed in \cite{BFL2}).\par
In order to define them we first define the notion of meager complex which is introduced in \cite{BFL}, this is the analog (and a refinement) of the notion of acyclic complex is the context of hermitian sheaves.\par
\begin{Def}
The class of \emph{meager complexes}, denoted $\mcl M(X)$ is the smallest class of complexes of hermitian bundles over $X$ satisfying, 
\begin{enumerate}
	\item Every ortho-split complex is meager.
	\item Every complex isometric to $\ovl F_\bullet\oplus \ovl F_\bullet[1]$, for $F_\bullet$ an acyclic complex endowed with any metric, i.e the cone of the zero map from an acyclic complex to itself is meager.
	\item The cone of the identity of a complex $\ovl F_\bullet$ is meager.
	\item For every morphism of complex $f:\ovl{E}_\bullet\to \ovl{F}_\bullet$, if any two of the following complexes $\ovl{E}_\bullet$, $\ovl{F}_\bullet$, ${\cone} (f)$ are meager, so is the third.
	\item Every shift of a meager complex, is meager.
\end{enumerate}
\end{Def}
Here, $\cone(f)$ is equipped with the orthognal metric induced by the metric on both summands.
\begin{Rque}
We will call a family of complexes of hermitian vector bundles satisfying the above conditions, a hermitian admissible class. It is easy to see that the intersection of a family of hermitian admissible classes is still an hermitian admissible class, thus the class of meager complexes is the intersection of all hermitian admissible classes.
\end{Rque}
\begin{Prop}
Let $\ovl{E}_\bullet$ be a meager complex, then $E_\bullet$ is acyclic.
\end{Prop}
\begin{proof}
Set $\mcl{ACL}(X)$ for the class of all complexes of hermitian bundles, $\ovl{E}_\bullet$ such that the underlying complex, $E_\bullet$ is acyclic. It will be sufficient to prove that $\mcl{ACL}(X)$ is an hermitian admissible class.\\
It is clear that ortho-split complexes are acyclic, just as clear is the fact that the cone of the zero map of an acyclic complex and of the identity map of any complex, is acyclic. Now let's consider $f:E_\bullet \to F_\bullet$ a morphism of complexes, we have a long exact sequence in cohomology $$...\to H_{i-1}(\cone(f))\to H_i(E_\bullet) \to H_i(F_\bullet)\to H_i (\cone(f)) \to ...$$
that ensures that if any two of the three complexes, $E_\bullet, F_\bullet, \cone(f)$ are acyclic, the the third is too.\\
Of course the shift of any acyclic complex is still acyclic.
\end{proof}
In the same way that a quasi-isomorphism has an acyclic cone, we define an equivalence relation between metric resolutions by declaring equivalent two resolutions differing by a meager cone.
\begin{Def}
A morphism $f:\ovl E_\bullet \to \ovl F_\bullet$ is said to be tight iff $ \cone (f)$ is meager.
\end{Def}
The preceding proposition admits the following translation
\begin{Cor}
A tight morphism between two complexes of hermitian vector bundles is a quasi-isomorphism.
\end{Cor}
Let us state
\begin{Def}
We will say that two complexes of hermitian vector bundles, $\ovl{E}_\bullet$ and $\ovl{F}_\bullet$, are quasi-isometric\footnote{In \cite{BFL2} the term used is tightly related.} iff there exists a complex of hermitian bundles $\ovl H_\bullet$ such that we have a diagram $$\xymatrix{
   &\ovl H_\bullet \ar[rd] \ar[ld] &\\
     \ovl F_\bullet  & & \ovl E_\bullet}$$
     where the two arrows are tight morphisms.
\end{Def}
We have the following characterization of the quasi-isometry relation 
\begin{Prop}\label{MaigreTight}
Two complexes of hermitian vector bundles, $\E_\bullet$ and $\ovl F_\bullet$ are quasi-isometric iff we can find a complex of hermitian vector bundles $\ovl{H}_\bullet$ such that we have a diagram$$\xymatrix{
   &\ovl{H}_\bullet \ar[rd]^{f} \ar[ld]_{g} &\\
     \ovl{E}_\bullet  & & \ovl{F}_\bullet   }$$
      with $g$ a quasi-isomorphism and such that complex $\ovl{\cone}(f)\oplus \ovl \cone (g) [1]$ is meager.
\end{Prop}
\begin{proof}This is \cite[Lemma 2.20]{BFL2}
\end{proof}

Moreover we have
\begin{Prop}\label{ToitTight}
Any diagram of tight morphisms of the form $$\xymatrix{
     \ovl E_\bullet \ar[rd] & & \ovl F_\bullet \ar[ld]\\
     &\ovl G_\bullet& }$$
		can be completed into a diagram 
		$$\xymatrix{
   &\ovl H_\bullet \ar[rd] \ar[ld] &\\
     \ovl E_\bullet \ar[rd] & & \ovl F_\bullet \ar[ld]\\
     &\ovl G_\bullet & } $$
		which is comutative up to homotopy and where all the arrows are tight.
\end{Prop}
\begin{proof}
This is \cite[Lemma 2.21]{BFL2}
\end{proof}
It is important to note that \begin{Prop}
The quasi-isometry is an equivalence relation.
\end{Prop}
\begin{proof}
The only part that is not obvious is the transitivity of this relation.\\
Let us consider $\ovl E^i_\bullet$, for $i=1,2,3$ three complexes of hermitian vector bundles, we assume that $\ovl E^1_\bullet$ and $\ovl E^2_\bullet$ are quasi-isometric and that $\ovl E^2_\bullet$ and $\ovl E^3_\bullet$ are also quasi-isometric. We thus have a diagram $$\xymatrix{
   \ovl H_\bullet\ar[d]\ar[rd] &&\ovl H'_\bullet \ar[ld]\ar[d]\\
   \ovl E^1_\bullet &\ovl E^2_\bullet & \ovl E^3_\bullet\\}
   $$
   We can complete this diagram into a diagram which is commutative up to homotopy $$\xymatrix{
   &\ovl G_\bullet \ar[rd] \ar[ld] &\\
     \ovl H_\bullet \ar[rd] & & \ovl H'_\bullet \ar[ld]\\
     &\ovl E^2_\bullet & } $$
     where all the arrows are tight thanks to \ref{ToitTight}, as the composition of tight morphisms is tight, we deduce the result.

\end{proof}
Let us set $\Vh(X)/\mcl M(X)$ to be the class of hermitian vector bundle modulo the quasi-isometry relation and $\KA(X)$ be the subset of $\Vh(X)/\mcl M(X)$ corresponding to the image of complexes of hermitian vector bundles such that the underlying complex is acyclic. \par
One can endow $\Vh(X)/\mcl M(X)$ with a structure of monoid using the orthogonal sum as the addition, the image of a complex $\E_\bullet$ in $\Vh(X)/\mcl M(X)$ will be denoted $[\E_\bullet]$.\par
This object inherits several properties summing up some diagram constructions, and that make proofs much less cumbersome that we list in the following proposition
\begin{Prop}\label{AcyCal}In $\Vh(X)/\mcl M(X)$ we have
\begin{enumerate}
	\item A complex $[\ovl E_\bullet]$ is invertible iff it is acyclic and then its inverse is given by the shift $[\ovl E_\bullet[1]]$.
	\item For every arrow $\ovl E_\bullet\to \ovl F_\bullet$, if $\ovl E_\bullet$ is acyclic (resp. $\ovl F_\bullet$ acyclic) then $$[ \cone (E, F)_\bullet]=[\ovl F_\bullet]-[\ovl E_\bullet]$$
	$$(\text{ resp. } [ \cone (E, F)_\bullet]=[\ovl E_\bullet[1]]+[\ovl F_\bullet])$$
	
	\item For every diagram $$\xymatrix{
   &\ovl G_\bullet \ar[rd] \ar[ld] &\\
     \ovl H_\bullet \ar[rd] & & \ovl H'_\bullet \ar[ld]\\
     &\ovl E_\bullet & } $$ which is commutative up to homotopy, we have $$[\cone( \cone (G, H),\cone(H',E))]=[\cone( \cone (G, H'), \cone(H,E))]$$

     \item If $f:\ovl E_\bullet \to\ovl F_\bullet; g:\ovl F_\bullet\to \ovl G_\bullet$ are two morphism between metrized complexes then we have $$[\cone(  \cone (g\circ f), \cone(g))]=[\cone(f)[1]]$$
     $$[\cone(  \cone (f), \cone(g\circ f))]=[\cone(g)]$$
     Moreover if $g$ or $ f$ is a quasi-isomorphism (resp. if $g\circ f$ is a quasi-isomorphism) then $$[( \cone (g\circ f)]=[ \cone(g))]+[\cone(f)]$$
    
    $$\text{ (resp. }[  \cone (g\circ f)]+[ \cone(f)[1]]=[\cone(g)])$$
\end{enumerate}
\end{Prop}
\begin{proof}
This is \cite[Theorem 2.27]{BFL2}.
\end{proof}

\subsection{Metric resolutions}
In their article \cite{BFL}, Burgos, Freixas et Litcanu, define a notion of equivalence for hermitian structure on the derived category of coherent sheaves, here we will simply restrict their definition to the case of a single coherent sheaf over a projective complex variety $X$.\par
\begin{Def}
We say that two hermitian structures $\E_\bullet \to \mcl F$ and $\ovl F_\bullet \to \mcl F$ on a coherent sheaf are quasi-isometric if there exists a complex of hermitian vector bundles $\ovl H_\bullet$ and a commutative diagram of resolutions $$\xymatrix{
   &\ovl{H}_\bullet \ar[rd]^{f} \ar[ld]_{g} &\\
     \ovl{E}_\bullet \ar[rd] & & \ovl{F}_\bullet \ar[ld]\\
     &\mcl F& }$$
		such that $f$ and $g$ are tight morphisms.
\end{Def}
Notice that as $f$ and $g$ are tight, they're quasi-isomorphisms and therefore $H_\bullet$ is a resolution of $\mcl F$.\par
We will need the following lemma
\begin{Lem}\label{PtiLem}
Assume that we have a diagram of complex of hermitian vector bundles $$\xymatrix{
   \ovl{E}_\bullet \ar[r]^{f}\ar[d]^g  & \ovl{E_\bullet}'\ar[d]^{g'}\\
     \ovl{F}_\bullet \ar[r]^{f'} & \ovl{F'}_\bullet }$$
		that commutes up to a homotopy, say $h$.\par
		Then $h$ induces two morphisms of complex $$\psi:\cone(f)\to \cone(f')$$
		and $$\fii: \cone(-g)\to \cone(g')$$ and we have a natural isometry $$\cone(\psi)\simeq \cone(\fii)$$
\end{Lem}
\begin{proof}
This is \cite[Lemma 2.3]{BFL2}
\end{proof}
\begin{Prop}
Let $\mcl F$ be a coherent sheaf on a complex algebraic variety and let ${\E}_\bullet\to\mcl F$ and ${\ovl G}_\bullet \to \mcl F$ be two metric resolutions of $\mcl F$, the following conditions are equivalent
\begin{itemize}
\item[i)] The two metric structures $\ovl{E}_\bullet\to\mcl F$ and $\ovl{G}_\bullet \to \mcl F$  are quasi-isometric.
\item[ii)] There exists $\ovl H_\bullet \to \mcl F$ a metric resolution such that we have a diagram of resolutions $$\xymatrix{
   &\ovl{H}_\bullet \ar[rd]^{f} \ar[ld]_{g} &\\
     \ovl{E}_\bullet \ar[rd] & & \ovl{F}_\bullet \ar[ld]\\
     &\mcl F& }$$
      with $g$ a quasi-isomorphism and such that complex $\ovl{\cone}(f)\oplus \ovl \cone (g) [1]$ is meager.
\item[iii)]For any metric resolution $ \ovl{H}_\bullet \to \mcl F$ such that we have a diagram of resolutions $$\xymatrix{
   &\ovl{H}_\bullet \ar[rd]^{f} \ar[ld]_{g} &\\
     \ovl{E}_\bullet \ar[rd] & & \ovl{F}_\bullet \ar[ld]\\
     &\mcl F& }$$
     where $g$ is a quasi-isomorphism, the complex $\ovl{\cone}(f)\oplus \ovl \cone (g) [1]$ is meager.
\end{itemize}
\end{Prop}
\begin{proof}The fact that $i)$ implies $ii)$ is obvious as the orthogonal sum of two meager complexes is meager.\par
Let's prove that $ii)$ implies $iii)$, so assume that there exists a metric resolution $\ovl H$ of $\mcl F$  like the one in the proposition.\par
Now let us consider $\ovl H'_\bullet$ any other metric resolution giving a diagram which is just like the one in the proposition. We can, using proposition \ref{toit}, find a complex of vector bundles, say $G_\bullet$ such that we have a commutative diagram up to homotopy $$\xymatrix{
   & G_\bullet \ar[rd]^\beta\ar[ld]_\alpha & \\
   H_\bullet \ar[rrd]_\delta\ar[d]^g & &H'_\bullet\ar[lld]_{f}\ar[d]_{\delta'} &\\
     E_\bullet & &F_\bullet  }$$
		with $\alpha$ and $\beta$ being quasi-isomorphisms. Let us endow $G_\bullet$ with any metric.\par
     
     By the previous lemma, we have $\cone(\cone(\alpha),\cone(\delta))$ isometric to $\cone(\cone(\beta),\cone(f))$ and $\cone(\cone(\alpha),\cone(\delta'))$ isometric to $\cone(\cone(\beta),\cone(g))$\par 
		Therefore, in $\Vh(X)/\mcl M(X)$, using that $\alpha$ and $\beta$, as well as $g$ and $\delta$ are quasi-isomorphisms and \ref{AcyCal} we have
		$$-[\cone(\alpha)]+[\cone(\delta)]=-[\cone(\beta)]+[\cone(g)]$$
		$$-[\cone(\alpha)]+[\cone(\delta')]=-[\cone(\beta)]+[\cone(f)]$$		
		by subtracting the first equation to the second (which is possible because all the cones appearting in the first equation are acyclic), we get $$[\cone(\delta')]-[\cone(\delta)]=[\cone(f)]-[\cone(g)]=[\cone(g)[1]]+[\cone(f)]=0$$
		and we are done.\par
For the implication $iii)\Rightarrow i)$ it results from \ref{MaigreTight}
\end{proof}
 
The following proposition is already proved because two resolutions define the same metric structure if the complexes of hermitian bundles obtained by truncating the $\mcl F$ are quasi-isometric.
\begin{Prop}
The relation of quasi-isometry is an equivalence relation over the set of metrized locally free finite resolutions of a given sheaf $\mcl F$
\end{Prop}
\begin{Rque}
The group $\KA(X)$ is identified with the metric structures on the zero sheaf over $X$, it has been named the group of universal secondary characteristic classes, or the group of acyclic $K$-theory.\end{Rque}

\subsection{Secondary classes}
In this section we review the notion of secondary characteristic classes for hermitian sheaves. Those classes are concrete realization of universal classes built up in \cite{BFL}. In this whole section, $X$ will design a complex algebraic variety, which we may assume to be projective and smooth over $\mbb C$.\\

\begin{Theo}\label{Bott-Chern} Let $\ovl {\mcl F_\bullet}$ be an acyclic complex of hermitian coherent sheaves, that is hermitian sheaves equipped with a metric structure. There exists a unique way of attaching to every such complex a Bott-Chern secondary characteristic form, denoted $\wt \ch(\ovl{\mcl F}_\bullet)\in \wt A^{\bullet,\bullet}(X)$, satisfying the following conditions.
\begin{enumerate}
	\item (Compatibility with Bott-Chern forms) If $\mcl E: 0\to \ovl E_1\to \ovl E_2\to \ovl E_3\to 0$ is an exact sequence of hermitian vector bundles, then $$\wt \ch(\mcl E)=\wt\ch^{BC}(\mcl E)$$ where $\wt\ch^{BC}$ is the Bott-Chern form associated to the exact sequence (see \cite{GS1})
	\item (Normalization) If $\ovl E_\bullet \to \mcl F$ is the metric resolution defining the hermitian structure over  $\ovl{\mcl F}$, then $\wt \ch(\E_\bullet \to \ovl{\mcl F})=0$
	\item (Devissage) If we have a complex of acyclic coherent metrized sheaves $\ovl{\mcl F}_\bullet$ that can be split up into exact sequences $${\mcl E_i}, 1\leq i\leq n-1: 0\to \mcl G_i\to \mcl F_i \to \mcl G_{i-1} \to 0$$ with $\mcl G_{-1}=\mcl F_0$ et $\mcl G_{n-1}=\mcl F_n$. We have $$\wt\ch(\ovl{\mcl F}_\bullet )+\sum_{i\geq 1}(-1)^{i}\wt\ch(\ovl{\mcl E_i})=0$$ for every choice of metric structure on the sheaves $\mcl G_i$ for $1\leq i\leq n-2$
	
	\item (Exactness) If we have a commutative diagram with exact rows and columns
	$$\xymatrix{
	& 0\ar[d]& 0\ar[d]&0\ar[d]&\\
   0 \ar[r]&\mcl F_{1,1} \ar[r]\ar[d]&\mcl F_{1,2} \ar[r]\ar[d]&\mcl F_{1,3} \ar[r]\ar[d]&0\\
     0 \ar[r]&\mcl F_{2,1} \ar[r]\ar[d]&\mcl F_{2,2} \ar[r]\ar[d]&\mcl F_{2,3} \ar[r]\ar[d]&0\\
     0 \ar[r]&\mcl F_{3,1} \ar[r]\ar[d]&\mcl F_{3,2} \ar[r]\ar[d]&\mcl F_{3,3} \ar[r]\ar[d]&0\\
     & 0& 0&0& }$$
     then we have the following equality $\wt \ch(\ovl{\mcl L_1})-\wt \ch(\ovl{\mcl L_2})+\wt \ch(\ovl{\mcl L_3})=\wt \ch(\ovl{\mcl C_1})-\wt \ch(\ovl{\mcl C_2})+\wt \ch(\ovl{\mcl C_3})$ where $\mcl C_i$ (resp. $\mcl L_i$) designs the i-th exact column (resp. the i-th exact row).
\end{enumerate}
\end{Theo}
\begin{proof}
This is proved in \cite{Zha}.
\end{proof}
Let us simply note that these Bott-Chern classes are in fact defined for hermitian sheaves up to quasi-isometry in the sense of Burgos, Freixas, Litcanu. Notice that, now that we have at our disposition the notion of hermitian structure for a sheaf it is easy to prove the analog of \ref{PtiLem} where the complexes of vector bundles are replaced with complexes of hermitian sheaves.
\begin{Lem}\label{NoDepend}
Assume that we have a short exact sequence of the form $0\to (\mcl F, h_1)\to (\mcl F, h_2)\to 0$ where the hermitian structures on both copies of $\mcl F$ are quasi-isometric, then its secondary Bott-Chern form $\wt\ch (\mcl F,h^1,h^2)$ vanishes.

\end{Lem}
\begin{proof}
Set $\ovl E^1_\bullet$ and $\E^2_\bullet$ two metric structures on $\mcl F$, that are assumed to be quasi-isometric. Then there exists a metric resolution, say $\ovl H_\bullet$, and a diagram of resolutions  $$\xymatrix{
   &\ovl H_\bullet \ar[rd]^g \ar[ld]_f &\\
     \ovl E^1_\bullet \ar[rd] & & \ovl E^2_\bullet \ar[ld]\\
     &\mcl F& }$$
     such that $\ovl{\cone}(f)\oplus \ovl \cone (g) [1]$ is meager, it will be sufficient to prove that for every meager complex $\ovl M_\bullet$, we have $\wt\ch(\ovl M_\bullet)=0$. Indeed, if such a result is satisfied, we have \begin{eqnarray*}0&=&\wt\ch(\ovl{\cone}(f)\oplus \ovl \cone (g) [1])\\
		&=&\wt\ch(\ovl E^1_\bullet\to (\mcl F,h^1))-\wt\ch(\ovl E^2_\bullet\to (\mcl F,h^1))\\
		&=&-\wt\ch(\ovl E^1_\bullet\to (\mcl F,h^2))+\wt\ch(\ovl E^2_\bullet\to (\mcl F,h^2))\end{eqnarray*}
     which will imply the result by the normalization condition. This also proves that in the general case, if $\ovl E_\bullet\to \mcl F$ is a metric structure on $\mcl F$ and $\ovl E'_\bullet$ another metric structure on $\mcl F$, then as expected $$\wt\ch(\ovl E_\bullet\to \mcl (F,h'))=\wt\ch (\mcl F,h',h)$$\par
     So let us first prove that a meager complex has a vanishing secondary class.\par
     Let us first consider $0\to \ovl E\to \ovl F\to \ovl G\to 0$ a short exact sequence of hermitian bundles, that is orhto-split, then using the compatibility condition with traditional Bott-Chern classes we have $\wt\ch(0\to \ovl E\to \ovl F\to \ovl G\to 0)=0$, so let us set $\mcl{CSN}(X)$ the class of acyclic complex of hermitian bundles that have vanishing secondary classes.\par
     This class certainly contains the cone of the identity map, and of the zero map of acyclic complexes, but also ortho-split complexes, and according to the preceding remark, it is also true that if any two of the three complexes $\ovl{E}_\bullet$, $\ovl{F}_\bullet$, $\ovl{\cone} (f)$ (where, of course, $f$ is an arrow from $E_\bullet$ to $F_\bullet$) have zero secondary classes, then so does the third, finally a shifting of an acyclic complex only changes the sign of the secondary classes of the complex in question; hence $\mcl{CSN}(X)$ is an admissible class, and as such, contains the class of meager complexes, which make the proof complete.
\end{proof}
\begin{Cor}
Secondary characteristic classes, only depend on the quasi-isometry class of the metric structure on the sheaves and not on the particular choice of a resolution within this quasi-isometry class.
\end{Cor}
\begin{proof}
We readily see that it is enough to prove that for every exact sequence $0\to \ovl{\mcl F_1}\to \ovl{\mcl F_2} \to \ovl{\mcl F_3}\to 0 $ the associated secondary class does not depend on the quasi-isometry class of the hermitian sheaves, the general result will follow by devissage.\par
Let us consider exact sequence $0\to \ovl{\mcl F'_1}\to \ovl{\mcl F'_2} \to \ovl{\mcl F'_3}\to 0$ where the sheaves are the same, but where the hermitian structures on the $\mcl F'_i$'s are quasi-isometric to the ones on the $\mcl F_i$'s, then we certainly have a commutative diagram with exact rows and columns. $$\xymatrix{0\ar[r] &\ovl{\mcl F_1}\ar[r]\ar[d]& \ovl{\mcl F_2} \ar[r]\ar[d]& \ovl{\mcl F_3}\ar[r]\ar[d]& 0 \\0\ar[r] &\ovl{\mcl F'_1}\ar[r]& \ovl{\mcl F'_2} \ar[r]& \ovl{\mcl F'_3}\ar[r]& 0 }$$
So, using exactness, and the previous lemma, we get $$\wt\ch(0\to \ovl{\mcl F_1}\to \ovl{\mcl F_2} \to \ovl{\mcl F_3}\to 0)=\wt\ch(0\to \ovl{\mcl F'_1}\to \ovl{\mcl F'_2} \to \ovl{\mcl F'_3}\to 0)$$
\end{proof}

\begin{Theo}\label{PropBottChern}
Let $\mcl F:0\to \ovl{\mcl F_1}\to \ovl{\mcl F_2} \to \ovl{\mcl F_3}\to 0$ a short exact sequence of coherent sheaves, and let $\ovl E$ be a hermitian vector bundle.
We have
\begin{enumerate}
	\item $dd^c\wt\ch(\mcl F)=\ch(\ovl{\mcl F_2})-\ch(\ovl{\mcl F_1})-\ch(\ovl{\mcl F_3}) $
  \item $\wt\ch(\mcl F\otimes \ovl E)=\wt\ch(\mcl F).\ch(\ovl E)$
\end{enumerate}

\end{Theo}
\begin{proof}
The first formula is immediate, it results from the fact that the formula is known to hold for complex of bundles, and from the the definitions we have given for secondary forms.\par
The second one is easy too, it follows from the fact that tensoring with a vector bundle is an exact functor and thus preserves dominating resolutions, and, there again, from the fact that the result is known to hold for complex of bundles for classical Bott-Chern forms.
\end{proof}
From the first formula one can deduce the following result, which is the one that interests us.
\begin{Cor}
The Chern form associated to a metrized sheaf only depends on the quasi-isometry class of the metric structure on the sheaf.
\end{Cor}

%% file: WAT.tex
In this section, we detail the main features of arithmetic Chow and $K$-theory that we want to preserve in arithmetic cobordism. Unfortunately the general features of the groups $\wh K_0(X)$ and $\wh \CH(X)$ are very different, not only with each other, but also from their geometric counterparts $K_0(X)$ and $\CH(X)$, we thus need weaker objects, who behave, at least on the functorial level much more closely with each other and with their geometric counterparts. Such objects had been introduced and studied before, by Burgos, Gillet-Soul\'{e}, Moriwaki, Zha...

\subsection{Weak Arithmetic Chow Groups}
Let $X$ be an algebraic projective smooth variety over a number field $k$. Let $V$ be a subvariety of $X$ of dimension $d+1$, and let $f$ be any rational function on $V$, recall that $\Div(f)$ is the cycle on $X$ defined as $$\Div(f)=\sum_{\text{irreducible }W\subset V; \codim_V(W)=1} \ord_W(f)[W]$$
We also set $\log|f|^2$ to be the current over $X$ defined in the following manner; let $\omega$ be any real smooth compactly supported form over $X$, of type $(d+1,d+1)$, we set $$\bra \log|f|^2, \omega\ket=\int_{V^{\text{ns}}}\log|f|^2\omega$$
where $V^{\text{ns}}$ denotes the open subset of $V(\mbb C)$ consisting of smooth points. As the singular locus of $V$ is of codimension at least $1$ in $V$, and as $\log|f|^2$ is a locally integrable function over $V^{ns}$, this is a well defined current over $X$, notice that $\log|f|^2$ is of type $(d_X-d_V, d_X-d_V)$. We could also define $\log|f|^2$ using the resolution of singularities of $V$. \\

We define the weak arithmetic Chow group in the following manner
\begin{Def}
We call the arithmetic weak Chow group of $X$ and we denote by $\CHw(X)$ the group $\wh{Z(X)}/\wh{\text{Rat}(X)}$ where 
\begin{itemize}
	\item The group $\wh Z(X)$ is the direct sum of the free abelian groups built on symbol $[Z]$ for every, $Z$, closed irreducible subset of $X$ and the group $\Dt(X)$.
	\item The subgroup $\wh{\text{Rat}}(X)$ is the subgroup of $\wh Z(X)$ generated by $([\div(f)],-\log|f|^2)$ for every $f\in k(V)^*$ for every subvariety $V$ of $X$.
\end{itemize}
\end{Def}
\begin{Rque}
We have a natural grading over $\CHw(X)$, where the homogenous piece of degree $d$ is given by $$\wh Z_d(X)=\bigoplus_{\dim Z=d} \mbb Z[V] \oplus \Dtp{d_X-1-d}(X)$$
For any $d+1$-dimensional subvariety $V$, $\div(f)$ is of degree $d$, and $\log|f|^2$ being of type $(d_X-(d+1), d_X-(d+1))$ is of degree $d$, thus $\wh{\text{Rat}}(X)$ is a homogenous subgroup of $\wh Z_\bullet(X)$, and $\CHw(X)$ inherits the grading.
\end{Rque}
From now on, we will simplify notations a bit, by writing $[Z,g]$, instead of $([Z],g)$, of course we have two natural maps $a: \Dt(X)\to \CHw(X)$ sending $g$ to $[0,g]$ and $\zeta: \CHw(X)\to \CH(X)$ sending $[Z,g]$ to $[Z]$.\par
Let us briefly examine the different operations that we can define on such groups.
\begin{Def}
Let $\pi: \ovl X\to \ovl Y$ be a projective morphism between arithmetic varieties, we define $$\pi_*[Z,g]=[\pi_*[Z],\pi_*g]$$
where $\pi_*[Z]$ is the push forward of geometric cycles defined in \cite{Fulton}; and $\pi_*(g)$ is the push forward of currents.
\end{Def}
It is a well known fact (see for instance \cite[Theorem 3.6.1]{GSA}) that this push forward is well defined and gives a functorial map $$\CHw(X)\ds{\pi_*}\CHw(Y)$$
this map is degree preserving.\\
In the same manner we can define a pull back-operation.
\begin{Def}(Pull-Back)
Let $f: \ovl X'\to \ovl X$ be a smooth equidimensional morphism between arithmetic varieties, we define $$f^*[Z,g]=[f^*[Z],f^*g]$$
where $f^*[Z]$ is the cycle associated to the equidimensional scheme $X'\times_X Z$, see \cite{Fulton}; and $f^*(g)$ is the pull-back of currents.
\end{Def}
The proof that this map is well defined on the level of the $\CHw$, and is functorial can be found in \cite[Theorem 3.6.1]{GSA}
\begin{Rque}
Here, the morphism $f^*$, for $f$ equidimensional of relative dimension $d$ increases degree by $d$, the relative dimension.
\end{Rque}
We can also define a first Chern class operator, but to do so let us fist notice that if $f\in k(V)^*$ is a rational function defined on a subvariety $V$ of $X$, then for every closed subvariety $Z$ of $X$ generically transverse to $V$, we can restrict $\log|f|^2$ to a (locally integrable) function defined on $V\cap Z$, that defines a current on $X$ by integration along the smooth locus of $V\cap Z$ (with the appropriate coefficient for each irreducible component of $V\cap Z$, namely its geometric multiplicity), we will denote such current as $\delta_Z\wedge \log|f|^2$, notice that we have a projection formula $$i_*i^*(\log|f|^2)=\delta_Z\wedge \log|f|^2=i_*i^*(1_Z)\wedge \log|f|^2$$
For every (regular) closed immersion $Z\ds{i}X$ it is a well known fact that we can find for every closed subvariety $V$ of $X$, another variety, say $W$, rationally equivalent to $V$ and transverse to $i$ so that we can extend that procedure to arbitrary arithmetic cycles on $X$, in fact we can also extend this definition to rational sections of hermitian bundles, as locally such a section can be represented as a rational function via a holomorphic trivialization, for details see \cite[1.3]{GSA}.\par
We can now define a \emph{First Chern class operator}
\begin{Def}(First Chern class operator)\\
Let $\ovl L\in \wh{\Pic}(X)$ be a hermitian line bundle over $\ovl X$, we define $\c1(\L)$ as an endomorphism of $\CHw(X)$  by the following formula $$\c1(\ovl L)[Z,g]=[\Div(s).[Z], c_1(\L)\wedge g-\log\|s\|^2\wedge \delta_Z]$$
where $s$ is any rational section of $L$ over $Z$, and $c_1(\L)$ is the curvature of the bundle $L$, which can be defined locally as $(-2i\pi)^{-1}\partial\ovl{\partial}\log\|s\|^2$ for any holomorphic local section of $L$.
\end{Def}
We need to check that this first Chern class operator is well-defined, but this again follows from \cite{GSA} or in \cite[paragraph before Theorem 2.4]{FaltBook}.
\begin{Rque}
Notice that $\c1(\L)$ decreases degrees by $1$, and that on $\Dt(X)$, $\c1(\L)$ only acts as $g\mapsto g\wedge c_1(\L)$.
\end{Rque}
Let's sum up the fundamental properties of these operations in the following proposition
\begin{Prop}(Borel-Moore properties)\label{BMforCH}\\
Let $X, Y, Y',S$ and $S'$, be smooth projective varieties and let $\pi: X\to Y$ and $\pi': Y\to Y'$ be projective morphisms and $f:S\to X$ and $f': S'\to S$ be smooth equidimensional morphism. We also fix $\M$ (resp. $\L$ and $\L'$), a (resp. two) hermitian bundle on $Y$ (resp. $X$), we have 
\begin{enumerate}
\item (Functoriality of the push forward) $(\pi' \circ\pi)_*=\pi'_*\pi_*$
\item (Functoriality of the pull back) $(f\circ f')^*=f'^*f^*$
\item (Naturality of the 1st Chern class) $f^* \circ \c1(\L)=\c1(f^*\L)\circ f^*$.
\item (Projection Formula) $\pi_*\circ \c1(\pi^*\M)=\c1(\M)\circ \pi_*$
\item (Commutativity of the 1st Chern Classes) $\c1(\L)\circ \c1(\L')= \c1(\L')\circ \c1(\L)$
\item (Grading) The degree of $\pi_*$ is $0$, the degree of $f^*$ is $d$, the degree of $\c1(\L)$ is $-1$.
\end{enumerate}
\end{Prop}
\begin{proof}In each case we can evaluate the veracity of these statements on cycles of the form $[Z,0]$ and $[0,g]$
\begin{enumerate}
	\item[1, 2.] This is \cite[Theorem 3.6.1]{GSA} for smooth pull-backs; notice that the push-forward is clearly functorial in the case where one the morphisms is a closed immersion, and for the composition of two smooth morphisms it is done in \cite[Theorem 3.6.1]{GSA}.\par
	If $f$ is a general projective morphism, that we may factor as $p\circ i$, then $p_*i_*$ does not depend on the choice of that factorization and thus with obvious notations $f'_*f_*=p'_*i'_*p_*i_*=p'_*(i'\circ p\circ i)_*=p'_*(p''\circ k)_*=p'_* p''_* k_*=(p'\circ p''\circ k)_*=(f'\circ f)_* $
	\item[3] For cycles $[Z,0]$ both sides are equal to $$[\div(f^*s).f^*Z, \log(\|f^*s\|^2)\wedge \delta_{f^*Z})$$ and on classes of the form $[0,g]$ it follows from the naturality of the Chern form for currents.
	\item[4] This results from the geometric projection formula which is true on the level of cycles and from the naturality of the Chern form for forms which implies the formula for currents by duality. Alternatively, one may argue that the formula is obvious if $\pi$ is a closed immersion, and is proved in \cite[(7) Thm p.158]{GSA} when $\pi$ is a smooth morphism as the action of the first Chern class, is simply the intersection with $[\div(s), \log\|s\|^2]$ where we choose $s$ a rational section of $L$ that is generically transverse to $Z$.
	\item[5] This is \cite[Thm 2.4]{FaltBook} for cycles $[Z,0]$ and a special case of \cite[Cor 2.2.9]{GSA} for currents. Another proof can be given using \ref{Transverse}.
	\item[6] This results from the definitions.
\end{enumerate}
\end{proof}
This properties give the functor $X \mapsto \CHw(X)$ the properties of a Borel-Moore functor. The following ones illustrate the "arithmetic" nature of this functor.
\begin{Prop}(Arithmetic Type of $\CHw$)\label{ATforCH}\\
Let $X$ be a projective smooth variety over $k$ of dimension $d$, we have
\begin{enumerate}
\item For any hermitian line bundles over $X$, $\L_1,...,\L_{d+2}$, we have $$\c1(\L_1)\circ...\circ \c1(\L_{d+2})=0$$ as an endomorphism of $\CHw(X)$.
\item Let $\L$ be a hermitian line bundle over $X$, with $s$ a global section of $L$ that is transverse to the zero section. Let $Z$ be the zero scheme of such a section, and $i:Z\to X$ the corresponding immersion. We have $$i_*(1_Z)=\c1(\L)(1_X)+a(\log\|s\|^2)$$
\item Given two hermitian bundles $\L$ and $M$ over $X$ we have $$\c1(\L\otimes \M)=\c1(L)+\c1(M)$$
\end{enumerate}

\end{Prop}
\begin{proof}
\begin{enumerate}
\item This results simply from the fact that we have a decomposition of abelian group $$\CHw(X)=\CHw_d(X)\oplus...\oplus \CHw_0(X)\oplus \CHw_{-1}(X)$$ and from the fact that the first Chern class operator is of degree $-1$.
\item Keeping the notation of the proposition we have $i_*(1_Z)=[Z,0]$, and $\c1(\L)(1_X)=[\Div(s), -\log\|s\|^2]=[Z,0]-a(\log\|s\|^2)$, and the result follows.
\item We have $\div(s\otimes t)=\div(s)+\div(t)$ as cycles, and by the very definition of the tensor product metric we have $\log\|s\otimes t\|^2=\log[\|s\|^2\|t\|^2]$, which implies the result.
\end{enumerate}
\end{proof}
To complete our description let us note that the weak arithmetic Chow groups are an extension of classical geometric Chow groups by the space of real currents modulo $\im \d +\im \db$.
\begin{Prop}
We have an exact sequence $$\Dt(X)\ds a \CHw(X) \ds \zeta \CH(X)\to 0$$
that breaks up into $$\Dtp{d-1-p}(X)\ds a \CHw_p(X) \ds \zeta \CH_p(X)\to 0$$
\end{Prop}
\begin{proof}
Let $\alpha=\sum n_i[Z_i,g_i]$ be a weak arithmetic cycle. The fact that $\sum n_{i} [Z_i]$ is trivial in $\CH(X)$ is equivalent to the existence of subvarieties $V_j$ of $X$ and $f_j\in k(V_j)$ rational functions over $V_j, $such that $\sum [Z_i]=\sum \Div(f_j)$ as cycles, we thus have $\alpha=\sum a(g_i)+\sum \div(f_j)=\sum a(g_i)+\sum a(\log|f_j|^2)$ which is evidently in the image of $a$.\par
The fact that the first exact sequence implies the others is simply a reformulation of the fact that the maps $a$ and $\zeta$ preserve the grading.
\end{proof}
\begin{Rque}The reader will compare this exact sequence to the one found in \cite{GSA}, and see that we have just replaced the space of real smooth forms by the space of general real currents, which has the advantage of having much better functoriality properties. This is why we have replaced the notion of arithmetic cycle presented in \cite{GSA} using a green current, by the notion of weak arithmetic cycle.
\end{Rque}

\subsection{Higher Analytic Torsion of Bismut-K\"ohler}
Recall the definition of arithmetic $\Kh$-theory given by Gillet and Soul\'{e} in \cite{GS1}
\begin{Def}Set $\Kh_0(\X)$ to be the free abelian group $\bigoplus \mbb Z[\ovl{E}]\times \At(X)$ where $\E$ is an isometry class of hermitian vector bundle over $X$, subject to the following relations: for every exact sequence $\mcl E: 0\to \ovl{E''}\to \ovl{E}\to \ovl{E'}\to 0 $, $$[\ovl{E},0]=[\ovl{E''},0]+[\ovl{E'},0]+[0,\cht(\mcl E)]$$
\end{Def}
One of the most profound problem in Arakelov theory is to define a direct image for such groups and to compute it, firstly we need to fix a metric on $\pi_* E$, unfortunately \emph{a priori} this is only a sheaf, and not a vector bundle, so Gillet and Soul\'{e} chose to examine a particular situation of utmost interest.\par
Consider a holomorphic proper submersion between complex manifolds, $\pi: M\to B$. Let $g$ be a hermitian metric on the holomorphic relative tangent bundle to $\pi$, denoted $T_{M/B}$, and let $J$ be the complex structure on the underlying real bundle to $T_{M/B}$ and $H_{M/B}$ the choice of a horizontal bundle i.e a smooth subbundle of $TM$ such that we have $TM=T_{M/B}\oplus H_{M/B}$.
\begin{Def} (K\"ahler Fibration)\\
 We say that this data defines a K\"ahler Fibration is there exists a smooth $(1,1)$-real form, say $\omega$ over $M$ such that\begin{enumerate}
	\item The form $\omega$ is closed
	\item The real bundles ${H_{M/B}}_{\mbb R}$ and ${T_{M/B}}_{\mbb R}$ are orthogonal with respect to $\omega$
	\item We have $\omega(X,Y)=g(X,JY)$ for $X$ and $Y$ vertical real vector fields.
\end{enumerate}
\end{Def}

Let us take $\E$ a hermitian bundle, we will make the two following assumptions
\begin{enumerate}
\item[(A1)] Assume that $\pi$ is a (proper) smooth submersion that is equipped with a structure of K\"ahler fibration.
\item[(A2)] Assume that $E$ is $\pi_*$-acyclic, meaning that $R^q\pi_*E=0$ as soon as $q>0$.
\end{enumerate}
In this case, the upper-semi-continuity theorem ensures that $\pi_*E$ is a vector bundle over $Y$, and for each (closed) point of $y$ we have $j_y^*\pi_*E=H^0(X_y, E_{|X_y})$. Now, using the K\"ahler fibration structure on $\pi$, we get a smooth family of metrics over the relative tangent spaces to $\pi$, that give a K\"ahler structure to the fiber $X_y$. Now, we can identify the space $H^0(X_y, E_{|X_y})$ to the subspace of $A^0(X_y, E)$ of smooth forms with coefficients in $E$, that are holomorphic, i.e killed by $\db$.
Now, on $A^0(X_y, E)$ we have a natural hermitian form given by the K\"ahler metric, defined by $$\bra s,t \ket=\int_{X_y} h^E(s(x),t(x)) \omega$$
where $\omega$ is the volume form defined by the K\"ahler metric on $X_y$.\par
We have thus defined a (punctual) metric on each fiber of the vector bundle $\pi_*E$, which we will call \emph{the $L^2$-metric associated to the K\"ahler fibration}.
\begin{Theo}
Assuming the previous conditions, on $X, Y, \pi$ and $\ovl E$, the $L^2$ metric is smooth and thus define a hermitian vector bundle structure on $\pi_*E$.
\end{Theo}
\begin{proof}
See \cite[p.278]{BGV} or \cite[p.2176]{RGS}
\end{proof}
We will denote $\ovl{\pi_*E}^{L^2}$ the vector bundle $\pi_*E$ equipped with its $L^2$-metric subordinated to the choice of metrics on both $X$ and $Y$. The problem now is to choose a form say $\Xi$ such that $$\pi_*[\E,0]=[\ovl{\pi_*E}^{L^2}, \Xi]$$
A first step to understand what this $\Xi$ should be, is to investigate the (cohomological) Riemann-Roch formula that says
 $$\ch(\ovl{\pi_*E}^{L^2})-\pi_*[\ch( \ovl E)\Td(\ovl{T_{X/Y}})]$$ must be the $\d \db$ of some smooth form over $Y$. Our form $\Xi$ should be one of those forms, indeed we have a generalized cycle map $$\omega: \wh K_0(X)\to A^{\bullet,\bullet}(X,\mbb R)$$ characterized by $\omega([\E,0])=\ch(\ovl E)$ and $\omega([0,\alpha])=dd^c\alpha$
, moreover this map is not compatible with pushforwards of forms in a naive sense, indeed we have by Riemann-Roch again $$\omega([\ovl{\pi_*E}^{L^2},0])=\pi_*[\omega(\ovl E)\Td(\ovl{T_{X/Y}})]$$ we thus see that our form $\Xi$ should \emph{double transgress} the Riemann-Roch formula i.e satisfy the equation $$\ch(\ovl{\pi_*E}^{L^2})-\pi_*[\ch( \ovl E)\Td(\ovl{T_{X/Y}})]=dd^c\Xi$$
such a form is \emph{a priori} only determined up to $\im \d+\im \db$, and of course there are many possible choices.\par
Bismut and K\"ohler were able to give a satisfying choice for $\Xi$, (see \cite[Def 1.7, Def 1.8, the paragraph before Thm 3.4 and Def 3.7]{BismutKohler} for the definition of the different terms, which we will not need)
\begin{Theo}(Bismut, K\"ohler)\\
Let $\E$ be a hermitian bundle and $\pi: X\to Y$ a holomorphic submersion endowed with a K\"ahler fibration structure, set $T(\ovl{T_{X/Y}}, \E)=\zeta_E'(0)$ where 
$$\zeta_E(s)=\frac{1}{\Gamma(s)}\int_0^\infty \frac{u^s}{u} \left[\fii\Trs( N_ue^{-B_u^2})-\fii \Trs( N_Ve^{-\nabla_\pi^2})\right]du$$ then $$\pi_*[\ch( \ovl E)\Td(\ovl{T_{X/Y}})]-\ch(\ovl{\pi_*E}^{L^2})=dd^cT(\ovl{T_{X/Y}}, \E)$$
The form $T(\ovl{T_{X/Y}}, \E)$ is called the \emph{Higher Analytic Torsion form} associated to $\E$ and $\ovl{T_{X/Y}}$
\end{Theo}
\begin{proof} This is \cite[Theorem 0.2]{BismutKohler}\end{proof}

\begin{Rque}\label{ConvT}
In the previous theorem the higher analytic torsion form $T(\ovl{T_{X/Y}}, \E)$ is associated to a K\"ahler fibration structure on $\pi$, note however that when $T_X$ and $T_Y$ are equipped with K\"ahler metrics, we have a natural structure of K\"ahler fibration on $\pi$ (see \cite[Thm 1.5]{BGS22}), in this case, we will then denote $T(\ovl T_X, \ovl T_Y,\E)$ instead of $T(\ovl{T_{X/Y}}, \ovl E)$ to mean the higher analytic torsion form associated to the K\"ahler fibration structure induced by the K\"ahler metrics over $T_X$ and $T_Y$
\end{Rque}
 We list here the fundamental properties of this higher analytic torsion that we may need.
\begin{Prop}\label{PropAT}
Let $\pi:X\to Y$ be a smooth submersion equipped with a K\"ahler fibration structure, let $\E$ be a $\pi_*$-acyclic vector bundle, and let's endow $\pi_* E$ with its $L^2$-metric. The \emph{analytic torsion} associated to this data, $T(\ovl{T_{X/Y}}, \E)$ is a smooth form in $\At(Y)$ that satisfy 
\begin{enumerate}
\item (Naturality) Let $g:Y'\to Y$ be projective morphism, then $X_{Y'}\to Y'$ is a K\"ahler fibration, and we have $$T(\ovl{g^*{T_{X/Y}}}, g^*\E)=g^*T(\ovl{T_{X/Y}}, \E)$$
\item (Additivity) For every pair $\E_1, \E_2$ of hermitian vector bundles on $X$, we have $$T(\ovl{T_{X/Y}}, \E_1\oplus^\perp \E_2)=T(\ovl{T_{X/Y}}, \E_1)+T(\ovl{T_{X/Y}},\E_2)$$
\item (Compatibility with the projection formula) For $F$ a hermitian vector bundle on $Y$, we have $$T(\ovl{T_{X/Y}}, \E\otimes \pi^*\ovl F)=T(\ovl{T_{X/Y}}, \E)\otimes \ch(\ovl F)$$ 
\item (Transitivity) If $\pi:X \to Y$ and $\pi': Y \to Z$ are two K\"ahler fibration structures and $E$ is a bundle $\pi_*$-acyclic, such that $\pi_*E$ is also $\pi'_*$-acyclic 
we have the following relation between the different analytic torsions \begin{eqnarray*}T(\ovl{T_{X/Z}}, \E)&=&T(\ovl{T_{Y/Z}}, \ovl{\pi_*E}^{L^2})+\pi'_*(T(\ovl{T_{X/Y}}, \E)\Td(\ovl{T_{Y/Z}}))\\
& &+\cht(\ovl{\pi'\circ \pi_*E}^{L^2},\ovl{\pi'_* (\ovl \pi_*E^{L^2})}^{L^2})\\
& &+\pi'_*\pi_*(\ch(\E)\wt{Td}(\mcl E)\Td(\ovl T_{X/Y})\Td^{-1}(\ovl T_X))\end{eqnarray*}
where $\mcl E$ is the exact sequence  $$\mcl E:0\to \ovl{T_{X/Y}} \to \ovl{T_{X/Z}} \to \pi^*\ovl{T_{Y/Z}}\to 0$$ 
\end{enumerate}
\end{Prop}
\begin{proof}
This is \cite[Cor 8.10, Cor 8.11]{BFL}.
\end{proof}

\subsection{Generalized Analytic Torsion of Burgos, Freixas, Litcanu}
To investigate the situation for a general projective morphism, Burgos, Freixas and Litcanu have split the problem into two different ones. First one wants to construct direct images for projective spaces $\P^r_Y \to Y$ and for closed immersions, and ask for a compatibility condition that would ensure a general functoriality property.\par
That's why Burgos, Freixas and Litcanu defined 
\begin{Def}(Generalized Theory of Analytic torsion for submersions)\\
A theory of generalized analytic torsion forms for submersions is an assignment of a smooth real form, $T(\ovl{T_X}, \ovl{T_Y}, \E)$ in $\At(Y)$ to every smooth submersion $X \ds{\pi} Y$ and every hermitian bundle $\E$ over $X$, with $T_X$ and $T_Y$ equipped with a K\"ahler metric, and $\E$ being $\pi_*$-acyclic, satisfying

$$\pi_*[\ch( \ovl E)\Td(\ovl{T_{X/Y}})]-\ch(\ovl{\pi_*E}^{L^2})=dd^cT(\ovl T_X, \ovl T_Y, \ovl E)$$
We say that a theory of generalized analytic torsion forms for submersions is well behaved, if it satisfies the following properties
\begin{enumerate}
\item (Naturality) Let $g:Y'\to Y$ be projective morphism, then $X'=X_{Y'}\ds{g'} Y'$ is a also a smooth submersion, and for any choice of metrics over $T_{X'}$ and $T_{Y'}$ such that we have an isometry $\ovl{T_{X'/Y'}} \simeq g'^*\ovl{T_{X/Y}}$ we have $$T(\ovl{{T_{X'}}}, \ovl{T_{Y'}}, g'^*\E)=g^*T(\ovl{T_X}, \ovl{T_Y}, \E)$$
\item (Additivity) For every pair $\E_1, \E_2$ of hermitian vector bundles on $X$, we have $$T(\ovl{T_{X}}, \ovl{T_{Y}}, \E_1\oplus^\perp \E_2)=T(\ovl{T_{X}}, \ovl{T_{Y}}, \E_1)+T(\ovl{T_{X}}, \ovl{T_{Y}},\E_2)$$
\item (Compatibility with the projection formula) For $F$ a hermitian vector bundle on $Y$, we have $$T(\ovl{T_X}, \ovl{T_Y}, \E\otimes \pi^*\ovl F)=T(\ovl{T_X}, \ovl{T_Y}, \E)\otimes \ch(\ovl F)$$ 
\end{enumerate}
\end{Def}
\begin{Rque}
The theory of analytic torsion defined by Bismut and K\"ohler, is an example of such a well-behaved theory, we will denote it $T^{BK}$.
\end{Rque}

Let's now turn to the case of a closed immersion, so assume now, that we've been given $i$ a (regular) closed immersion between projective smooth complex varieties. The geometric situation is somewhat more complicated given that in general $i_*E$ will not be a vector bundle on $Y$, but we can just arbitrarily choose a hermitian structure on $i_*E$, given by a resolution of that sheaf on $Y$, in the sens of the first section, let us choose such a resolution $\E_\bullet\to i_*E\to 0$.\par
We thus have a current, represented by a smooth form $\ch(\ovl{i_*E})$ which is equal to $\sum (-1)^i \ch(\ovl E_i)$, we want to compare it to the current $i_*[\ch(\ovl E)\Td(\ovl N)^{-1}]$ for a choice of metric over the normal bundle to $i$.\par
Notice here that $\ch(\ovl E)\Td(\ovl N)^{-1}$ is a well-defined smooth form on $X$, but is only a current on $Y$, when we push it forward through $i_*$, however as the cohomology of currents coincide with that of forms (\cite{Griffith-Harris}), we do know that there exists a current $\Xi \in \Dt(Y)$ depending \emph{a priori} on the choice of the metric on $N$, and the metric structure on $i_*E$ such that $$ i_*[\ch(\ovl E)\Td(\ovl N)^{-1}]-\sum (-1)^i \ch(\ovl E_i)=\d\db \Xi$$
we can, here again, try to give an explicit formula, let us examine the case of a smooth effective divisor.\par
Assume that we've been given a hermitian line bundle $\L$ over $Y$, such that $X$ is the zero locus of a section $s$ of $L$, transverse to the zero section. We have the following exact sequence $$0\to L^{\vee}\to \mcl O_Y\to i_*\mcl O_X\to 0$$
that gives a resolution of $i_*\mcl O_X$ over $Y$, if we equip $\mcl O_X$ with the trivial metric, and $L^\vee$ with the dual metric, we get a hermitian structure on $i_*\mcl O_X$. Moreover as $i^*L$ is naturally isomorphic to $N_{X/Y}$ we also have a natural metric on the normal bundle to $i$. For this data, a current $\Xi$ solving the Grothendieck-Riemann-Roch equation is easy to compute.
\begin{Prop}
Let $\L$ be a hermitian vector bundle over a smooth complex variety, say $Y$, and let $X$ be the zero scheme of a transverse section. We denote by $j$ the corresponding regular immersion. We have $$\ch(\ovl{j_*{O}_X})=\ch([\ovl O_Y]-[\L^\vee])=j_*(\ch(\ovl  O_X)\Td(\ovl{j^*L})^{-1})-dd^c(\Td(\L)^{-1}\wedge \log\|s\|^2 )$$
\end{Prop}
\begin{proof}
First let us compute 
\begin{eqnarray*}
\ch(\ovl{j_* O_X})&=&\ch([\ovl  O_Y]-[\L^\vee])\\
 &=& e^{c_1(\ovl  O_Y)}-e^{c_1(\L^\vee)}\\
 &=& 1-e^{-c_1(\L)}
\end{eqnarray*}
Let's now compute 
 
\begin{eqnarray*}
j_*(\ch(\ovl  O_X)\Td(\ovl j^*L)^{-1})&=&j_*(\Td[j^*\L]^{-1})\\
 &=& \Td(\L)^{-1}j_*(1_X)\\
 &=& \Td(\L)^{-1}\wedge \delta_X\\
&=&\Td(\L)^{-1}\wedge(dd^c\log\|s\|^2+c_1(\L))\\
&=& dd^c(\Td(\L)^{-1}\wedge \log\|s\|^2)+c_1(\L)\wedge \Td(\L)^{-1}\\
&=&dd^c(\Td(\L)^{-1}\wedge \log\|s\|^2)+(1-e^{-c_1(\L)})
\end{eqnarray*}
where we have used the projection formula, the definition of $\delta_X$, the Poincar\'e-Lelong formula, the fact that the Todd form is closed, and finally the definition of the Todd form. The proposition follows.
\end{proof}

To generalize this phenomenon, we need the following definition.
\begin{Def}(Singular Bott-Chern Current)\\
Let $i: Z\to X$ be a (regular) immersion between smooth projective complex varieties and $\ovl E$ a hermitian bundle over $Z$, we assume that we've been given a hermitian structure on $i_*E$ and on $N=N_{Z/X}$.\par A singular Bott-Chern current for this data, which we will denote $\bc(\N, \ovl{i_*E})$ is current defined up to $\im \d+\im \db$ satisfying the following differential equation $$\ch(\ovl{i_*E})=\sum_{i\geq 0} (-1)^i\ch(\E_i)=i_*(\ch \E \Td(\N)^{-1})-dd^c(\bc(\N, \ovl{i_*E}))$$
\end{Def}
\begin{Rque}
Notice that we have chosen to compare $i_*[\ch(\ovl E)\Td(\ovl N)^{-1}]$ with $\sum (-1)^i \ch(\ovl E_i)$ but there is another choice, just as natural, namely to compare $\sum (-1)^i \ch(\ovl E_i)$ with $i_*[\ch(\ovl E)\Td(i^*\ovl T_X)^{-1}\Td(\ovl{T_Z})]$ as both classes are mapped to the same cohomology class.\par
We could of course give the same definition replacing $\bc(\N, \ovl{i_*E})$ by $\bc(\ovl T_Z, \ovl T_X, \ovl{i_*E})$ which would satisfy the equation $$\ch(\ovl{i_*E})=\sum_{i\geq 0} (-1)^i\ch(\E_i)=i_*[\ch(\ovl E)\Td(i^*\ovl T_X)^{-1}\Td(\ovl{T_Z})]-dd^c(\bc(\ovl T_Z, \ovl T_X, \ovl{i_*E}))$$
If we have a singular Bott-Chern current for one of these two choices, it is easy to find a singular Bott-Chern current for the other by the following formula $$\bc(\ovl T_Z, \ovl T_X, \ovl{i_*E})=\bc(\N, \ovl{i_*E})+i_*[\ch(\E)\wt\Td^{-1}(\mcl E)\Td(\ovl T_Z)]$$
where $\mcl E$ is the exact sequence $$0\to T_X\to i^*T_Y\to N\to 0$$
Of course if that exact sequence were to be meager for the different metrics chosen, then the two Bott-Chern singular currents would agree.\par
In the situations that will be of prime interest to us, we will have metrics on $X$ and $Y$ instead of $N$, so we will use the singular Bott-Chern determined by the tangent metrics rather than the normal one, nevertheless in the literature, the formulae for singular Bott-Chern currents are usually given for a choice of metric on the normal bundle, that's why we chose to follow this convention in the remainder of this section.\par
We hope that the reader will have no problem in making the occasional switch between properties for $\bc(\N, \ovl{i_*E})$ and $\bc(\ovl T_Z, \ovl T_X, \ovl{i_*E})$
\end{Rque}
The previous proposition gives us an obvious candidate for a singular Bott-Chern current in the case of the immersion of a divisor.\par
Notice that, the fact that a singular Bott-Chern current always exists is a trivial consequence of the Grothendieck-Riemann-Roch theorem (and the $\d\db$-lemma for currents), but notice also that there are many possible choices \emph{a priori} for $\bc(\N, \ovl{i_*E})$, all differing by a cohomology class.\par
The general case of a (regular) immersion is much more complicated but Bismut-Gillet-Soul\'e in \cite{BGS1} gave an explicit formula for a singular Bott-Chern current. Just like for the case of the analytic torsion a little bit of technology (which we will not explicitly describe) is needed to be able to state the result.\par
We can state the following formula (see \cite{BGS1} for the definition of the terms, we won't need those in the following) 
\begin{Theo}( Bismut-Gillet-Soul\'{e})\\
Let $j: X\to Y$ be an immersion between compact complex manifold and $\ovl E$ a hermitian vector bundle on $X$, let us endow $N_{X/Y}$ with a hermitian metric and $i_*E$ with a hermitian structure $\E_i\to i_*E$ satisfying the condition (A) of Bismut, then $\zeta_E'(0)$ is a singular Bott-Chern Current for this data, where $$\zeta_E(s)=\frac{1}{\Gamma(s)}\int_0^\infty \frac{u^s}{u} \left[\Trs( Ne^{-A_u^2})-i_*\int_{X}\Trs( Ne^{-B^2})\right]du \label{BKC}$$
This current agrees with $-\log\|s\|^2\Td(\L)^{-1}$ in the case of the immersion of a smooth divisor.
\end{Theo}
\begin{proof}
See \cite[Theorem 1.9, Theorem 3.17]{BGS1}
\end{proof}
Here again, there are a priori many possible choices for a singular Bott-Chern current, and we wish to determine uniquely a choice for it that agrees with our explicit choice for divisors. Fortunately the classification of the theories of singular Bott-Chern currents has been accomplished by Burgos and Litcanu in \cite{BL}.\par
We now turn to a brief description of their theory, to do so we need to describe a paradigmatic situation in which we will state the analog of the properties of \ref{PropAT} for a singular Bott-Chern current.\par
Let $i: X\to Y$ and $j:Y\to Z$, be regular immersions of complex varieties, we have the following exact sequence \[ \ 0\to N_{X/Y}\to N_{X/Z}\to j^*N_{Y/Z}\to 0  \label{eq:NNN}\tag{$\dagger$} \]
and assume that we have chose a hermitian structure on $i_*E$ given by a complex $\E_{\bullet}$, as $j_*$ is exact we have an exact sequence \[0\to j_*E_n\to...\to j_*E_1\to j_*i_*E \to 0 \]
if we equip all the $E_i$'s with hermitian structures, say $\ovl{E_{i,\bullet}}$ we get a global resolution of $j_*i_*E$ given by the total complex of the double complex $E_{\bullet, \bullet}$. This is the hermitian structure that $j_*i_*E$ will be equipped with.

\begin{Def}\label{PropBCC}
A \emph{theory of singular Bott-Chern classes} is an assignment of a current $\bc(\ovl N, \E, \ovl{i_*E})$ to each immersion $i: X\to Y$ between smooth projective complex varieties and a hermitian bundle $\E$ over $X$, equipped with a hermitian structure on both $N_{X/Y}$ and $i_*E$, satisfying $$\ch(\ovl{i_*E})=i_*(\ch \E \Td(\N)^{-1})-dd^c\bc(\ovl N, \E, \ovl{i_*E})$$
A theory of singular Bott-Chern Classes is said to be
\begin{enumerate}
\item natural:  if given $g: Y'\to Y$ a morphism transverse to $i$ (e.g smooth), recall that the transversality condition implies $g^*N_{X/Y}\simeq N_{X'/Y'}$, we have $$\bc(\ovl{g^*{N_{X/Y}}}, g^*\E, g^*\ovl{i_*E})=g^*\bc(\ovl N_{X/Y}, \E, \ovl{i_*E})$$
\item additive: if $$\bc(\ovl N, \E_1\oplus^\perp \E_2, \ovl{i_*E_1}\oplus^\perp\ovl{i_*E_2})=\bc(\ovl N, \E_1, \ovl{i_*E_1})+\bc(\ovl N, \E_2, \ovl{i_*E_2})$$
\item compatible with the projection formula: if $$\bc(\ovl N, \E\otimes i^*\ovl F, \ovl{i_*E}\otimes \ovl F)=\bc(\ovl N, \E_1, \ovl{i_*E_1})\ch(\ovl F)$$ where $\ovl F$ is a hermitian vector bundle on $Y$
\item transitive: if it is additive and if for every composition of closed immersion $i: X\to Y$, and $j:Y\to Z$, and for every choice of metrics on the normal bundles, we have \begin{eqnarray*}\bc(\ovl N_{X/Z}, \E, \ovl{j_*i_*E})&=&\sum_{r \geq 0}(-1)^r \bc(\ovl N_{Y/Z}, \E_r, \ovl{j_*E_r})\\
& &+j_*[\bc(\ovl N_{X/Y}, E, \ovl{i_*E})\Td(\ovl N_{Y/Z})^{-1}]+j_*i_*[\ch(\E)\wt{\Td}^{-1}(\dagger)]\end{eqnarray*}
\end{enumerate}
\end{Def}

\begin{Rque}
If the different conditions in the previous definition are only satisfied for a particular class of metric structure we will say that the corresponding theory of singular Bott-Chern is the corresponding adjective with respect to that particular choice of metrics.\par
If a theory of Bott-Chern singular currents satisfies all of the assumptions above, we will say that it is well-behaved. 
\end{Rque}
We have the following proposition 
\begin{Prop}
The theory of singular Bott-Chern Classes defined by Bismut and given for a metric satisfying condition $(A)$ by formula \ref{BKC} is well-behaved.
\end{Prop}
\begin{proof}
This is \cite[Prop 9.28]{BL}
\end{proof}
Therefore the singular Bott-Chern current constructed by Bismut-Gillet-Soul\'{e} (that we will denote $\bc^{BGS}$) is an example of a well-behaved theory of singular Bott-Chern current, but it is far from being the only one. Indeed we have
\begin{Theo}
For any choice of a real additive genus $S$ there exists a unique theory of well behaved Bott-Chern singular currents satisfying $$\bc(\E, \ovl{i_*E}, \N)=\bc^{BGS}(\E, \ovl{i_*E}, \N)+i_*[\ch(E)\Td(N)S(N)]$$
\end{Theo}
\begin{proof}
This is \cite[7.14]{BFL}
\end{proof}
\begin{Rque}
If $\Lambda$ is a ring, a genus over $\Lambda$ is simply a power series over $\Lambda$. We will say that a genus $g$ is multiplicative (resp. additive) if we extend it as a power series given by $$g(T_1,...,T_n)=g(T_1)...g(T_n) (\text{resp. } g(T_1)+...+g(T_n))$$

We can associate to such a genus over $\Lambda$ a characteristic form with coefficients in $\Lambda$, which will be also called a genus. 
\end{Rque}
Up to this point we have considered two kinds of secondary objects, the (higher) analytic torsion form for K\"ahler fibrations $\pi: X\to Y$, and the singular Bott-Chern classes for immersions $i: Y\to Z$, anticipating just a bit on the following section, we will see that these two secondary objects help us define a direct image in arithmetic weak $\wh{K}$-theory, each construction will assure functoriality of this direct image with respect to the kind of morphism it is defined with, i.e we will have $(ij)_*=i_*j_*$ for composition of closed immersions, and we will have $(\pi\pi')_*=\pi_*\pi'_*$ for composition of K\"ahler fibrations.\par
In order to have a general functoriality property for arbitrary projective morphism, one needs to impose a compatibility condition between analytic torsion, and singular Bott-Chern classes. Burgos, Freixas, and Litcanu have studied that question in \cite{BFL} and it turns out that it can be done as soon as we have compatibility for them in a mild situation.\par
Let's consider the following diagram
$$\xymatrix{
  \mbb P^n \ar[r]^{\Delta} \ar[rd]^{id} &\mbb P^n\times \mbb P^n\ar[r]^{p_1}\ar[d]^{p_2}& \mbb P^n\ar[d]\\
     & \mbb P^n\ar[r]& \Spec k}$$
     
     If we want to achieve functoriality in $\wh{K}$-theory, the least we can ask is that $p_{2,*}\Delta_*=\Id$, let us write down explicit equations for this condition to be true for the trivial bundle over $\mbb P^n$.\par
     We have an explicit resolution of $\Delta_*\mcl O_{P_n}$ given by the Koszul complex
     $$0\to \bigwedge^r (p_2^*Q\otimes p_1^*\mcl O(1)) \to ...\to \bigwedge^2(p_2^*Q\otimes p_1^*\mcl O(1))\to (p_2^*Q\otimes p_1^*\mcl O(1)) \to \mcl O_{\mbb P^n\times \mbb P^n}\to \Delta_*\mcl O_X\to 0$$
     where $Q$ is the universal subbundle on $\mbb P^n$.
     Now if we choose a trivial metric on the trivial bundle of rank $n+1$ on $\mbb P^n$ we get a Fubini-Study metric on $\mcl O(1)$, and on $T_{\mbb P^n}$ and also on $T_{\mbb P^n\times \mbb P^n}$, moreover, the universal exact sequence $$0\to Q\to q^*\mcl O_{\mbb P^n}^{n+1}\to \mcl O(1)$$
     enables us to equip $Q$ with a metric too.\par
Therefore we have a metric structure on $\Delta_* \mcl O_{\mbb P^n}$, now let us define $p_{2*} \ovl{\Delta_* \mcl O_{\mbb P^n}}$ as $\sum (-1)^i\ovl{p_{2*}\bigwedge^r (p_2^*\ovl Q\otimes p_1^*\ovl{\mcl O(1)})}^{L^2}$ where we use the structure of K\"ahler fibration defined by the Fubini Study metrics over $\mbb P^n$ and $\mbb P^n\times \mbb P^n$, of course as the normal bundle to the diagonal immersion is naturally isomorphic to the tangent bundle of $\P^n$, we also have a metric on it.\par
We need to compare this class in $\wh{K}$-theory, with the class of $\mcl O_{\mbb P^n}$ with trivial metric; in order to have compatibility of the two kinds of secondary objects that have been added, the difference between them must be zero, and this justifies the following definition extracted form \cite[Def 6.2]{BFL}
\begin{Def}
We say that a well behaved theory of Bott-Chern singular class is compatible with a well behaved theory of analytic torsion if the following identity holds for the situation previously defined \begin{eqnarray*}0&=&\sum (-1)^r T(\mbb P^n\times \P^n, \P^n, \ovl{p_{2*}\bigwedge^rp_2^*\ovl Q\otimes p_1^*\ovl{\mcl O(1)})}^{L^2})\\
 & & +p_{2 *}[\bc(\ovl N_{\P^n/\P^n\times \P^n}, \ovl{i_*\mcl O_{\P^n}}). \Td(\ovl{T_{\P^n\times \P^n/\P^n}})]\end{eqnarray*}
\end{Def}
The following proposition is due to Burgos, Freixas and Litcanu
\begin{Theo}
For any choice of well-behaved theory of singular Bott-Chern currents for closed immersions, there exists a well-behaved theory of higher analytic torsion classes compatible with it.\end{Theo}
\begin{proof}
This is \cite[Thm 7.7]{BFL}
\end{proof}
\begin{Rque}
In \cite{BFL} they use a different normalization of the singular Bott-Chern current of Bismut-Gillet-Soul\'e, due to the fact that they work with the Deligne complex, instead of general currents, therefore they have to multiply the singular Bott-Chern current by $-\frac{1}{2}$ in order to have compatibility.
\end{Rque}
Recall the the $R$-genus of Bismut-Gillet-Soul\'e is the additive characteristic class determined by the following equation $$R(L)=\sum_{m \text{ odd}}(2\zeta(-m)+\zeta'(-m)(1+\frac{1}{2}+...+\frac{1}{m}))\frac{c_1(L)^m}{m!}$$
\begin{Theo}\label{Rgenus}(Bismut; Burgos-Freixas-Litcanu)\\
The theory of analytic torsion for K\"ahler fibrations associated to the singular Bott-Chern current $\bc^{BGS}$ is given by $$T(\ovl{T_X}, \ovl{T_Y}, \E)=T^{BK}(\ovl{T_X}, \ovl{T_Y}, \E)-\int_{X/Y}\ch(E)\Td(T_{X/Y})R(T_{X/Y})$$
where $R$ is the $R$-genus of Bismut-Gillet-Soul\'{e}.
\end{Theo}
\begin{proof}
This is the conjunction of \cite[Thm 7.14]{BFL} and \cite[Thm 0.1 and 0.2]{Bismut1}
\end{proof}
From now on we will only work with the theory of singular Bott-Chern current of Bismut-Gillet-Soul\'{e}, whose Bott-Chern current associated to the immersion of a divisor is given by $-\log\|s\|^2\Td(\L)^{-1}$, which we will simply denote $\bc$ and with the analytic torsion for submersions compatible to it, which we will simply denote $T$.

\subsection{Weak $\Gw$ and $\Kw$-theory}
We're now able to define a weak arithmetic analog of $K$-theory, and of $G$-theory. Recall that $G_0(X)\simeq K_0(X)$ for regular schemes, and \emph{a fortiori} for smooth projective varieties over a number field. During this part, we will fix $\X$ an arithmetic variety.
\begin{Def}(Arithmetic variety)\\
Let $X$ be a smooth algebraic variety over a field $k$, such that its (holomorphic) tangent bundle is equipped with a hermitian K\"ahler metric $\ovl{T_X}$, that is invariant under complex conjugation.
\end{Def}
\begin{Rque} We will often denote $d_X$ the dimension of $X$ as a complex manifold, moreover, note that on the algebraic variety $\Spec k$ there is only one metric. We will thus simply denote $\Spec k$ for the arithmetic variety $\ovl{\Spec k}$ equipped with that metric.
\end{Rque}
\begin{Rque}
It is convenient for our purpose to have finite coproducts in our category of arithmetic varieties, if $\wh{H}$ is a hermitian Borel-Moore functor (see \ref{HBM}) defined only for arithmetic varieties, we can extend it to the category of (finite) disjoint union of arithmetic varieties by setting $$\wh H(\coprod_{i=1}^n \X_i)=\bigoplus_{i=1}^n \wh H(\X_i)$$
in that case the operations defined are additively-extended. This convention will be followed in the rest of the paper, in particular we will only define hermitian Borel-Moore functors on arithmetic varieties, the extension to a finite disjoint of arithmetic varieties being implicit.
\end{Rque}
\begin{Def}(Weak $\Gw$-theory)\\
We set  $\Gw_0(\X)$ to be the free abelian group $\bigoplus \mbb Z[\ovl{\mcl F}]\times \Dt(X)$ where $\ovl{\mcl F}$ is a quasi-isometry class of hermitian coherent sheaves over $X$, subject to the following relations: for every exact sequence $\mcl E: 0\to \ovl{\mcl F''}\to \ovl{\mcl F}\to \ovl{\mcl F'}\to 0 $, $$[\ovl{\mcl F},0]=[\ovl{\mcl F''},0]+[\ovl{\mcl F'},0]+[0,\cht(\mcl E)]$$
\end{Def}
Similarly we define 
\begin{Def}(Weak $\Kw$-theory)
We set  $\Kw_0(\X)$ to be the free abelian group $\bigoplus \mbb Z[\ovl{E}]\times \Dt(X)$ where $\E$ is an isometry class of hermitian vector bundle over $X$, subject to the following relations: for every exact sequence $\mcl E: 0\to \ovl{E''}\to \ovl{E}\to \ovl{E'}\to 0 $, $$[\ovl{E},0]=[\ovl{E''},0]+[\ovl{E'},0]+[0,\cht(\mcl E)]$$
\end{Def}
\begin{Rque}
As in the case of weak arithmetic Chow theory, we have maps $$\zeta: \Gw_0(\X)\to G_0(X); \  \zeta: \Kw_0(\X)\to K_0(X)$$ and maps $$a: \Dt(X)\to \Gw_0(\X) ; \  a: \Dt(X)\to \Kw_0(\X) $$
this last map will not be functorial with respect to push-forwards.
\end{Rque}
For the case of smooth projective varieties it makes no difference to work with either $\Kw$ or $\Gw$-theory, as the next proposition shows
\begin{Theo}\label{PoincareIso}
Let $X$ be a smooth variety, then we have a natural isomorphism $$\Gw_0(\X) \ds \sim \Kw_0(\X)$$
\end{Theo}
\begin{proof}We have an obvious map form $\Kw_0(\X)$ to $\Gw_0(\X)$ that maps a hermitian bundle $\E$ to the same bundle equipped with the hermitian structure $0\to \E \ds \Id E \to 0$, and that maps $[0,g]$ to itself. Let us construct a map from $\Gw_0(\X)$ to $\Kw_0(\X)$, let $\E_\bullet \to \mcl F$ be a hermitian coherent sheaf, we map $\ovl {\mcl F}$ to $\sum (-1)^i [\E_i]$.\par
This map is well defined, if $\E'_\bullet \to \mcl F$ is another hermitian structure on $\mcl F$ quasi-isometric to the first one, then we have a diagram of resolutions $$\xymatrix{
   &\ovl H_\bullet \ar[rd] \ar[ld] &\\
     \ovl E_\bullet \ar[rd] & & \ovl E'_\bullet \ar[ld]\\
     &\mcl F& }$$
		where the two top arrows are tight.\par		Now, by the very definition of $\Kw_0(\X)$, if $\ovl{M}_{\bullet}$ is a meager complex of vector bundles then $\sum (-1)^i\ovl M_i=0$, therefore the complex $\cone(\ovl H_\bullet, \E_\bullet)[1]\oplus \cone(\ovl H_\bullet, \ovl E'_\bullet)$ (which is a hermitian structure on the zero sheaf over $X$) being meager gives us the following identity in $\Kw_0(\X)$, $$ \sum (-1)^i\ovl H_i -\sum (-1)^i\ovl E_i -\sum (-1)^i\ovl H_i+\sum (-1)^i\ovl E'_i=0$$
		which ensures that our map depends only on the quasi-isometry class of the chosen metric structure.\par
		Now if $$\mcl E: 0\to \ovl{\mcl F''}\to \ovl{\mcl F}\to \ovl{\mcl F'}\to 0 $$
		is an exact sequence of hermitian sheaves, where the hermitian structure on $\mcl F^\star$ is given by a resolution $\ovl F_\bullet^\star$, we have to prove that \[ \sum (-1)^i\ovl F''_i+\sum (-1)^i\ovl F'_i-\sum (-1)^i\ovl F_i+\cht(\mcl E)=0 \tag{$\star$} \label{eq:etoile}\]
		Let us take another choice of resolutions of $\mcl F$ and $\mcl F''$, dominating the previous ones, such that we have a diagram with exact rows and columns $$\xymatrix{
     0 \ar[r]&  G''_{\bullet} \ar[r]\ar[d]& G_{\bullet} \ar[r]\ar[d]& F'_{\bullet} \ar[r]\ar[d]&0\\
     0 \ar[r]&{\mcl F''} \ar[r]\ar[d]&{\mcl F} \ar[r]\ar[d]&{\mcl F'} \ar[r]\ar[d]&0\\
     & 0& 0&0& }$$
		let us endow $G''_{\bullet}$ and $G_{\bullet}$ with arbitrary metrics.
		Now, by devissage and normalization in \ref{Bott-Chern}, we have \begin{eqnarray*}\cht(0\to \ovl G''_{\bullet}\to \ovl G_{\bullet} \to \ovl F'_{\bullet}\to 0 )-\wt\ch(\ovl G''_{\bullet}\to \ovl{\mcl{F}}'')+\wt\ch( \ovl G_{\bullet}\to  \ovl{\mcl F}))=\cht(\mcl E)\end{eqnarray*}				
on the other hand we have in $\Kw_0(\X)$, $$\wt\ch( \ovl G_{\bullet}\to \ovl{\mcl{F}})=\sum (-1)^i\ovl F_i-\sum (-1)^i\ovl G_i$$
and similarly $$\wt\ch(\ovl{G''_{\bullet}}\to \ovl{\mcl{F}}'')=\sum (-1)^i\ovl{ F}_i''-\sum (-1)^i\ovl G''_i$$
and of course $$\cht(0\to \ovl G''_{\bullet}\to \ovl G_{\bullet} \to \ovl F'_{\bullet}\to 0 )=\sum(-1)^i \ovl G_i-\sum(-1)^i \ovl G''_i-\sum(-1)^i \ovl F'_i$$
Putting all this together yields \eqref{eq:etoile}, and our map is well defined.\par
This obviously gives a left inverse to the natural map from $\Kw_0(\X)$ to $\Gw_0(\X)$, it suffices thus to prove the surjectivity of this map, but this results immediately from the definition.
\end{proof}
\begin{Rque}
One may wonder where the smoothness hypothesis intervene in the previous proof. It does not. By our very definition we have restricted ourselves to sheaves that admit finite locally free resolutions in $\Gw_0$ but the smoothness hypothesis implies that every coherent sheaf admits such resolutions, and this in turn will imply the surjectivity of the forgetful arrow $\zeta:\Gw_0(\X) \to G_0(X)$ which will be important for us, when given an arbitrary coherent sheaf, we want to equip it with a metric and view it as an element of the $\Kw_0$ as we will often do.
\end{Rque}
\begin{Rque}
In the same spirit we could have defined an intermediate group where the vector bundles are equipped with hermitian structures given by resolutions instead of "classical metrics", we leave it to the reader to check that this group would have also been isomorphic to the $\Kw_0$ we defined.
\end{Rque}

Let's define a {first Chern class} operator and a pull-back operation.
\begin{Def}(Pull-Back)\\
Let $f: \ovl X'\to \ovl X$ be a smooth equidimensional morphism between arithmetic varieties, we define $f^*:\Kw_0(\X) \to \Kw_0(\X')$ by the following formula $$f^*[\E,g]=[f^*\E,f^*g]$$
where $f^*\E$ is the pull-back of $E$ equipped with the hermitian metric that renders isometric the isomorphism $f^*E_x\simeq E_{f(x)}$; and $f^*(g)$ is the pull back of currents.
\end{Def}
The fact that this operation is well defined follows from the naturality of the secondary Bott-Chern classes.
\begin{Def}(First Chern class operator)\\
Let $\ovl L\in \wh{\Pic}(X)$ be a hermitian line bundle over $\ovl X$, we define $$\c1(\ovl L)[\E,g]=[\E,g]-[\E\otimes \L^\vee, g\wedge \ch(\L^\vee)]$$
we will call this operator \emph{the first Chern class operator}.
\end{Def}
We have to check that this operator is well defined, namely that it sends a class (coming from an exact sequence $\mcl E: 0\to \ovl{E''}\to \ovl{E}\to \ovl{E'}\to 0 $) of the form $$\ovl{E}-[\ovl{E''}+\ovl{E'}+\cht(\mcl E)]$$ to zero, but this follows from the second point of \ref{PropBottChern}, as the following sequence is exact $$\mcl E\otimes \L^\vee: 0\to \ovl{E''\otimes L^\vee}\to \ovl{E\otimes L^\vee}\to \ovl{E'\otimes L^\vee}\to 0 $$ therefore $\ovl{E\otimes L^\vee}=\ovl{E''\otimes L^\vee}+\ovl{E'\otimes L^\vee}+\cht(\mcl E\otimes L^\vee)=\ovl{E''\otimes L^\vee}+\ovl{E'\otimes L^\vee}+\cht(\mcl E)\ch(\L^\vee)$
\begin{Rque}
Let's precise the action of $\c1(\L)$ on classes of the form $[\E,0]$ and $[0,g]=a(g)$, we see that 
\begin{itemize}
	\item $\c1(\L)[\E,0]=[\E,0]-[\E\otimes \L^\vee,0]$
	\item $\c1(\L)a(g)=a(g\wedge(1-\ch(\L^\vee)))=a(g\wedge c_1(\L)\Td(\L)^{-1})$
\end{itemize}
\end{Rque}

\begin{Rque}
The attentive reader will have noticed that up to this point, nothing depends on the K\"ahler metric we have chosen on $X$. This dependency will play a part in the definition of the push-forward that we will give now.
\end{Rque}
Let us prove now that, as expected, arithmetic $\Kw$-theory is an extension of classical $K$-theory by the space of currents. 
\begin{Prop}
We have an exact sequence $$\Dt(X)\ds a \Kw_0(\X) \ds \zeta K_0(X)\to 0$$
\end{Prop}
\begin{proof}
Let $\alpha=\sum_i n_i[E_i,g_i]$ be any element in the kernel of $\zeta$. We thus have $\sum_i n_i [E_i]=\sum_j m_j ([F_j]-[F'_j]-[F''_j])$ for some exact sequences $\mcl F_j:0\to F'_j\to F_j\to F''_j \to 0$, more precisely each $E_i$ is isomorphic to an $F_i$. Therefore in $\Kw_0(\X)$ we have $$\alpha=\sum_j m_j ([\ovl F_j]-[\ovl F'_j]-[\ovl F''_j])+a(g)=\sum_j a(\cht(\ovl{\mcl F_j}))+a(g)$$ and we are done.
\end{proof}
As we said earlier, it is non trivial to define a push forward in arithmetic $\Kw$-theory, but we now have the tools to make such a construction, in order to do this let us state a lemma due to Quillen.
\begin{Lem}
Let $f: X\to Y$ be a projective morphism between separated noetherian schemes, then $K_0(X)$ in generated as a group by $f_*$-acyclic vector bundles.
\end{Lem}
\begin{proof} This is \cite[p.41, paragraph 2.7]{Quillen}
\end{proof}
For $f$ a projective morphism between arithmetic varieties, $\X$ and $\Y$, we set $\Kw_0^f(X)$ to be the free abelian group built on symbols $[\E]$ where $E$ is a $f_*$-acyclic vector bundle, times $\Dt(X)$, modulo $[\ovl{E}]=[\ovl{E''}]+[\ovl{E'}]+\cht(\mcl E)$ for every exact sequence $\mcl E: 0\to \ovl{E''}\to \ovl{E}\to \ovl{E'}\to 0 $ where the bundles are $f_*$-acyclic. We can reformulate Quillen's lemma as 
\begin{Lem}
The natural map $$\Kw_0^f(X) \ds c \Kw_0(X)$$ is an isomorphism.
\end{Lem}
In view of the previous lemma, it will be sufficient to construct a direct image for $f_*$-acyclic vector bundles: we will give below a definition for a map $f_*: \Kw_0^f(X) \to \Kw_0(Y)$, the morphism $f_*:\Kw_0(X) \to \Kw_0(Y)$ will simply be $f_*\circ c^{-1}$. \par
We will examine separately the case of a smooth submersion and that of an immersion.
\begin{Def}(Direct image for a submersion)\\
Let $\pi:\X \to \Y$ be a smooth submersion, the K\"ahler metrics on $\X$ and $\Y$ induce a structure of K\"ahler fibration on $\pi$, if $\E$ is a $\pi_*$-acyclic hermitian bundle on $X$ we set $$\pi_*[\E,g]=\left[\ovl{\pi_*E}^{L^2}, T(\ovl{T_{X/Y}}, \E)+\int_{X/Y}g\wedge \Td(T_{X/Y})\right]$$
\end{Def}
We need to check that this definition makes sense, namely that if $$\mcl E: 0\to \ovl{E''}\to \ovl{E}\to \ovl{E'}\to 0 $$
is an exact sequence of hermitian vector bundle that are $\pi_*$-acyclic, then as the exact sequence $\pi_*\mcl E$ remains acyclic, $$\pi_*[\ovl{E},0]=\pi_*[\ovl{E''}+\ovl{E'},\cht(\mcl E)]$$
This is achieved by the following anomaly formula 
\begin{Prop}(Anomaly Formula for the analytic torsion)\\
Let $\mcl E: 0\to \ovl{E''}\to \ovl{E}\to \ovl{E'}\to 0 $ be an exact sequence of $\pi_*$-acyclic vector bundles where $\pi$ is a K\"ahler fibration from $X$ to $Y$, we have $$T(\ovl{T_{X/Y}},\E)-T(\ovl{T_{X/Y}},\E')-T(\ovl{T_{X/Y}},\E'')-\cht(\pi_* \mcl E)=\int_{X/Y}\cht(\mcl E)\Td(T_{X/Y})$$
\end{Prop}
\begin{proof}
This is the equation (47) of \cite[p. 46]{ZGS} and \ref{Rgenus}
\end{proof}
We thus obtain a well defined direct image for smooth submersions from $X$ to $Y$. Let us now turn to the case of an immersion.
\begin{Def}(Direct image for an immersion)\\
Let $i:\X \to \Y$ be a (regular) immersion between arithmetic varieties, for any hermitian vector bundle $\E$, we equip $i_*E$ with any hermitian structure, we set $$i_*[\E,g]=\left[\ovl{i_*E}, \bc(\ovl{T_{X}}, \ovl{T_{Y}}, \E, \ovl{i_*E}')+i_*[g\wedge \Td(i^*\ovl{T_{X}})^{-1}\Td(\ovl{T_{Z}})]\right]$$
\end{Def}
Here again we need to check that everything is well defined, we have to check that $$[\ovl{i_*E}, \bc(\ovl{T_{X}}, \ovl{T_{Y}}, \E, \ovl{i_*E}')+i_*[g \Td(i^*\ovl{T_{X}})^{-1}\Td(\ovl{T_{Z}})]$$ does not depend on the hermitian structure chosen on $i_*E$, and that for an exact sequence $\mcl E: 0\to \ovl{E''}\to \ovl{E}\to \ovl{E'}\to 0 $, we have $$i_*[\ovl{E},0]=i_*[\ovl{E''}+\ovl{E'},\cht(\mcl E)]$$
We have the following anomaly formulae that ensure us that this is the case.
\begin{Prop}(Anomaly Formulae for the singular Bott-Chern current)\\
Let $\mcl E: 0\to \ovl{E''}\to \ovl{E}\to \ovl{E'}\to 0 $ be an exact sequence vector bundles and $i$ be a closed immersion from $\X$ to $\Y$, let us chose hermitian structures on $i_*E$, $i_*E'$, and $i_*E''$, we have an exact sequence $i_*\mcl E: 0\to \ovl{i_*E''}\to \ovl{i_*E}\to \ovl{i_*E'}\to 0 $
\begin{enumerate} 
\item If $\ovl{i_*E}$ and $\ovl{i_*E}'$ denote two different hermitian structures on ${i_*E}$, we have $$\bc(\ovl{T_{X}}, \ovl{T_{Y}}, \E, \ovl{i_*E})-\bc(\ovl{T_{X}}, \ovl{T_{Y}}, \E, \ovl{i_*E}')=\cht(0\to \ovl{i_*E}'\to \ovl{i_*E}\to 0)$$
\item Furthermore we have \begin{eqnarray*}& &\bc(\ovl{T_{X}}, \ovl{T_{Y}},\E, \ovl{i_*E})-\bc(\ovl{T_{X}}, \ovl{T_{Y}},\E',\ovl{i_*E'})-\bc(\ovl{T_{X}}, \ovl{T_{Y}},\E'', \ovl{i_*E''})\\
&=&i_*[\cht(\mcl E)\Td(i^*\ovl{T_{X}})^{-1}\Td(\ovl{T_{Z}})]+\cht(i_* \mcl E)\end{eqnarray*}
\end{enumerate}
\end{Prop}
\begin{proof}
The first formula is a particular case of the second one for a very short exact sequence. The second formula is \cite[Thm 2.9]{BGS1}
\end{proof}
This completes the definition of the direct image for a closed immersion, to have a fully fledged definition we need the following proposition 
\begin{Theo}(Direct image in $\Kw$-theory)\label{PFK}\\
Let $f:\X \to \Y$ be a projective morphism between two arithmetic varieties, and let $$\xymatrix{   &\P^r_Y\ar[rd]^p& \\
     X \ar[ru]^i\ar[rd]^j & & Y\\
     &\P^\ell_Y \ar[ru]^q& }$$
		be two decompositions of $f$ into an immersion followed by a smooth morphism\footnote{In fact we would replace the projective spaces over $Y$ by any variety smooth over $Y$ equipped with a K\"ahler metric.} (where the projective spaces are endowed with the Fubini-Study metric and $\P^\bullet_Y$ with the product metric). Then $p_*i_*=q_*j_*$ and this morphism only depends on the hermitian metrics on $X$ and $Y$.
\end{Theo}
\begin{proof} The proof of this result can essentially be found in the litterature, for instance in \cite[Theo 10.7]{BFL} albeit using a slightly different language, or \cite[Prop 5.8]{BFLGGR}
\end{proof}
Therefore the following definition makes sense
\begin{Def}
Let $f: \X \to \Y$ be a projective morphism between arithmetic varieties, we set $f_*=p_*i_*$ for any choice\footnote{we can choose $\ell=0$ if $f$ is an immersion} of factorization of $f$ into $\X\ds i \ovl{\P_Y^\ell} \ds p \Y$
\end{Def}

We see now that arithmetic $\Kw$-theory satisfies the same properties as arithmetic weak Chow theory except for the fact that the latter is graded whereas the former is not.
\begin{Prop}(Borel-Moore properties)\label{BMforK}\\
Let $X, Y, Y',S$ and $S'$, be smooth projective varieties and let $\pi: X\to Y$ and $\pi': Y\to Y'$ be projective morphisms and $f:S\to X$ and $f': S'\to S$ be smooth equidimensional morphism. We also fix $\M$ (resp. $\L$ and $\L'$), a (resp. two) hermitian bundle on $Y$ (resp. $X$), we have 
\begin{enumerate}
\item (Functoriality of the push forward) $(\pi' \circ\pi)_*=\pi'_*\pi_*$
\item (Functoriality of the pull back) $(f\circ f')^*=f'^*f^*$
\item (Naturality of the 1st Chern class) $f^* \circ \c1(\L)=\c1(f^*\L)\circ f^*$.
\item (Projection Formula) $\pi_*\circ \c1(\pi^*\M)=\c1(L)\circ \pi_*$
\item (Commutativity of the 1st Chern Classes) $\c1(\L)\circ \c1(\L')= \c1(\L')\circ \c1(\L)$
\end{enumerate}
\end{Prop}
\begin{proof}
The first point is \ref{PFK}, the second point is obvious, so is the third using the naturality of the Chern character, and the last one is just as obvious. Let us prove the projection formula, we have $$\pi_*(\Td(\ovl T_\pi)\wedge g\wedge (1-\ch(\pi^*\L^\vee)))=\pi_*(\Td(\ovl T_\pi)\wedge g)\wedge (1-\ch(\L^\vee)))$$
To prove the result on classes $[\ovl E, 0]$ we prove it first for a closed immersion, we have \begin{eqnarray*}
i_*[\c1(i^*\L)[\E]]&=&i_*([\E]-[\E\otimes i^*\L^\vee])\\
&=&[i_*\E]-[i_*(\E\otimes i^*\L^\vee)]+\bc(\ovl{T_{X}}, \ovl{T_{Y}}, \E, i_*\E)-\bc(\ovl{T_{X}}, \ovl{T_{Y}}, \E\otimes i^*\L^\vee, i_*\E\otimes \L^\vee)\\
&=&[i_*\E]-[i_*\E\otimes\L^\vee]+\bc(\ovl{T_{X}}, \ovl{T_{Y}}, \E, i_*\E)-\bc(\ovl{T_{X}}, \ovl{T_{Y}}, \E , i_*\E)\ch(\L^\vee)\\
&=&\c1(\L)\left[[i_*\E]+\bc(\ovl{T_{X}}, \ovl{T_{Y}}, \E, i_*\E)\right]\\
&=&\c1(\L)i_*[\E]
\end{eqnarray*}
where we have used the definition of the first Chern class, the definition of the direct image, the isometry of the chosen resolutions $i_*\E\otimes\L^\vee\simeq i_*(\E\otimes i^*\L^\vee)$ and the compatibility of the Bott-Chern singular current with the projection formula.\par
To prove the result for K\"ahler fibrations, we may assume that $\E$ is $\pi_*$-acyclic, and the same proof applies.
\end{proof}
As before, for arithmetic weak Chow groups, we have more, let us give now the properties that encode the arithmetic nature of this functor
\begin{Prop}(Arithmetic Type of $\Kw$)\label{ATforK}\\
Let $\X$ be an arithmetic variety of dimension $d$, we have
\begin{enumerate}
\item For any hermitian line bundles over $X$, $\L_1,...,\L_{d+2}$, we have $$\c1(\L_1)\circ...\circ \c1(\L_{d+2})=0$$ as an endomorphism of $\Kw(X)$.
\item Let $\L$ be a hermitian line bundle over $X$, with $s$ a global section of $L$ that is transverse to the zero section. Let $Z$ be the zero scheme of such a section, and $i:Z\to X$ the corresponding immersion. We have $$i_*(1_Z)=\c1(\L)(1_X)+a(\log\|s\|^2\Td(\L)^{-1})+i_*[\wt\Td^{-1}(\mcl E)\Td(\ovl T_Z)]$$
where $\mcl E$ is the exact sequence $0\to \ovl{T_Z}\to i^*\ovl{T_X}\to i^*\L\to 0$ associated to the immersion.
\item Given two hermitian bundles $\L$ and $\M$ over $X$ we have $$\c1(\L\otimes \M)=\c1(L)+\c1(M)-\c1(\L)\c1(\M)$$
\end{enumerate}

\end{Prop}
\begin{proof}
Let us keep the notation of the proposition
\begin{enumerate}
\item Firstly, we see that $\c1(\L_1)\circ...\circ \c1(\L_{d+2}).a(g)=0$ as the action of the first Chern class increases the type by $(1,1)$.\par 
It is now a good time to notice that we can see $\Kw_0(X)$ as a module over $\Kh_0(X)$ and that the action of the $\c1$ is given by multiplication by $[\O_X]-[\L^\vee]$, so the identity we want to prove is in fact an identity in $\Kh_0(X)$, where a product is defined in a way that the composition of the actions of the first Chern classes is just the multiplication of the corresponding classes.\par
Now since $X$ is regular we have $\Kh_0(X)\simeq \Gh_0(X)$ (\cite[Lem. 13]{GSinv}) and in $\Gh_0(X)$ we have $[\O_X]-[\L^\vee]=[\ovl{i_*\O_Z}]$ as soon as $L$ is effective, where $Z$ is the zero scheme of any global section of $L$ (where the hermitian structure on $\ovl{i_*\O_Z}$ is of course given by the obvious exact sequence).\par
Let us assume for a moment, that all the $L_i$'s are very ample, we want to prove $[i_{1*}\mcl O_{Z_1}]...[i_{(d+2)*}\mcl O_{Z_{d+2}}]=0$, using Bertini's theorem \cite[Theo 2.3]{Bert}, we can choose global sections of each $L_i$ such that $Z_i=\Div(s_i)$ is generically transverse to $\bigcap_{j<i} Z_j$, but this is tantamount to saying that $\bigcap_{i=1...d+1} Z_i$ is empty, therefore $[i_{1*}\mcl O_{Z_1}]...[i_{(d+2)*}\mcl O_{Z_{d+1}}]=0$.\par
Let us now assume that all the bundles $L_i$ are very ample, except one, which is anti-very-ample i.e its dual is very ample, and assume for simplicity that it is $\L_1=\L$. As for any line bundles $\M$ and $\M'$ we have (in $\Kh_0(X)$ or $\Gh_0(X)$ where the product is defined) $$[\O_X]-[\M^\vee\otimes \M']=([\O_X]-[\M^\vee])+([\O_X]-[\M'])-([\O_X]-[\M^\vee]).([\O_X]-[\M'])$$ (which incidentally proves the third point), we see, plugging $\L=\M_i=\M'_i$, that $$0=[\O_X]-[\L^\vee\otimes \L]=([\O_X]-[\L])+([\O_X]-[\L^\vee])-([\O_X]-[\L])([\O_X]-[\L^\vee])$$
therefore we can replace $([\O_X]-[\L])$ by $([\O_X]-[\L])([\O_X]-[\L^\vee])-([\O_X]-[\L^\vee])$, and as $\L^\vee$ is very ample, the results follow from the previous case.\par
Now let us consider the general case recall that each $L_i$ can be written as $M_i\otimes {M'_i}^\vee$ with $M_i$ and $M'_i$ very ample, let's endow these two line bundles with any metric rendering the previous isomorphism, isometric. As $$[\O_X]-[\M_i^\vee\otimes \M'_i]=([\O_X]-[\M_i^\vee])+([\O_X]-[\M'_i])-([\O_X]-[\M_i^\vee]).([\O_X]-[\M_i'])$$ we're reduced to the case of ample and anti-ample line bundles which yields the result.
\item In view of the exact sequence $$0\to L^\vee \to \mcl O_X\to i_*\mcl O_Z\to 0$$ We have \begin{eqnarray*}i_*(1_Z)&=&[\O_X]-[\L^\vee]+\bc(\ovl T_Z, \ovl T_X, \O_Z, \ovl{i_*O_Z})\\
&=&[\O_X]-[\L^\vee]+\Td(\L)^{-1}\wedge \log\|s\|^2 +i_*[\wt\Td^{-1}(\mcl E)\Td(\ovl T_Z)]\end{eqnarray*}
and on the other hand $\c1(\L)(1_X)=[\O_X]-[\L^\vee]$, the result follows.
\item It has been proven in the course of the demonstration of the first point.
\end{enumerate}
\end{proof}

%% file: WAC.tex
We now proceed to the construction of a weak arithmetic cobordism group, we take the common properties of $\CHw(X)$ and $\Kw(X)$, namely \ref{BMforCH}, \ref{ATforCH}, \ref{BMforK}, \ref{ATforK} as a guideline for our construction. To generalize both the formulae $(2)$ in \ref{ATforCH} and \ref{ATforK} we need to construct a universal Todd class that would give back the traditionnal Todd class in $\Kw$-theory and $1$ in $\CHw$-theory by appplication of a forgetful functor. The first task we need to tackle is to find a good ring of coefficients over which our universal Todd class will be defined.
\subsection{Arithmetic Lazard Ring, Universal Todd class and secondary forms associated to it}
I will define here a modified version of the Lazard ring.
\begin{Def}
We set the \emph{arithmetic Lazard ring} to be the ring $$\Lh=\mbb Z[a^{ij}, t_k, (i,j)\in \mbb N\times \mbb N, k\in \mbb N ]$$ divided by the ideal $I$ such that the following relations hold in $\Lh[[u,v,w]]$, \begin{itemize}\item $\displaystyle \sum_{i\geq 0} t_i(u+v)^{i+1}=\sum_{i\geq 0, j\geq 0} a^{i,j} u^iv^j \left(\sum_k t_k u^k\right)^i\left(\sum_r t_r u^r\right)^j$
\item $\mbb F(u, \mbb F(v,w))=\mbb F(\mbb F(u,v),w)$
\item $\mbb F(u,v)=u+v \text{  mod } (u,v)^2$
\item $\mbb F(u,v)=\mbb F(v,u)$
\item $\mbb F(0,u)=u$
\item $t_0=a^{1,0}=a^{0,1}=1$
\end{itemize}
where $\mbb F$ is the universal law group $\mbb F(u,v)=\sum a^{i,j}u^iv^j$.
\end{Def}
We can readily check that this ring is not zero. To do so, we can build a quotient of that ring that is not zero. Let us consider the map $\{a^{i,j}, t_i\}\to \mbb Q$ defined by $a^{1,1}\mapsto -1$ and $a^{i,j}\mapsto 0$ for $(i,j)\notin\{(1,1), (1,0), (0,1)\}$ and $t_i \mapsto (-1)^{i}/(i+1)!$ for $i>0$. This map induces a map form $\Lh$ to $\mbb Q$ that is not $0$ ensuring that $\Lh$ is not trivial.\par
We will let $\g(u)$ denote the universal power series over $\mbb Z[\t]$, $$\g(u)=\sum_{r}t_r u^r$$ we can re-write the first axiom as $$\mbb F(u\g(u),v\g(v))=(u+v)\g(u+v)$$
Let us denote the unique power series $\mfr h$ over $\mbb Q[\t]$, defined by $\mfr h(\g(u)u)=u$, we see that $$\mfr h(u)+\mfr h(v)=\mfr h(\mbb F(u, v))$$
in other words, $\h$ is a morphism of formal group laws, from the universal group law to the additive group law; as it turns out $\h$ is in fact an isomorphism after tensorization by $\mbb Q$.\par
We can now prove
\begin{Prop} As rings the arithmetic Lazard ring and the Lazard ring are isomorphic after tensorization by $\mbb Q$, $$\mbb L_{\mbb Q} \simeq \Lh_{\mbb Q}$$
\end{Prop}
\begin{proof}
Let $\mfr h$ be the power series over $\mbb Q[\t]$ previously defined, this power series define a formal law group on $\mbb Q[\t]$ given by \begin{equation}\mfr h^{-1}(\mfr h(u)+\mfr h(v))=\mbb F(u,v)\tag{$\star$}\label{eq: FGL}\end{equation} and this defines a morphism $$\mbb L_{\mbb Q}\to \mbb Q[\t]$$ which is an isomorphism, because there is a natural bijection between \begin{enumerate}
	\item $\Lambda \mapsto \Hom_{\mbb Q-\text{algebras}}(\mbb L_{\mbb Q}, \Lambda)$.
	\item $\Lambda \mapsto \FGL(\Lambda)$ (the set of formal group laws over the $\mbb Q$-algebra $\Lambda$).
	\item $\Lambda \mapsto \text{Genera}(\Lambda)$ (the set of genera over $\Lambda$ i.e satisfying $g(u)=u \mod u^2$)
\end{enumerate}
where the bijection between $2$ and $3$ is given by \eqref{eq: FGL} (see \cite{Ochanine}), let us denote $\psi$ the isomorphism $\mbb L_{\mbb Q}\ds\sim \mbb Q[\t]$ defined in this way. As $\mbb L_{\mbb Q}$ is evidently a $\mbb Q[\t]$-algebra, we have a twisted action of $\mbb L_{\mbb Q}$ in itself given by $\psi(a).b$, we will also denote by $\psi(a)$ the element $\psi(a).1$ in $\mbb L_{\mbb Q}$\par
With this in mind, we see that $\Lh_{\mbb Q}$ is isomorphic (as a left $\mbb L_{\mbb Q}$-module) to $\mbb L_{\mbb Q}\otimes \mbb L_{\mbb Q}/I$ where $I$ is the ideal generated by $a\otimes 1-1\otimes \psi(a)$. Let us consider the arrow $m:\mbb L_{\mbb Q}\otimes \mbb L_{\mbb Q}\to \mbb L_{\mbb Q}$ given by multiplication $a\otimes b\mapsto \psi(a)b$, this map certainly factors through $\Lh_{\mbb Q}$.\par
To see that it is an isomorphism we have to prove that the kernel of the multiplication map is exactly $I$, but this is easy, as we have a section $s:\mbb L_{\mbb Q}\to \mbb L_{\mbb Q}\otimes\mbb L_{\mbb Q}$ given by $a\mapsto \psi^{-1}(a)\otimes 1$, and the kernel of the multiplication is generated as a left $\mbb L_{\mbb Q}$-module by elements of the form $\sum 1\otimes x_i$ with $\sum m(1\otimes x_i)=0$ thus $$\sum 1\otimes x_i=\sum 1\otimes x_i-0=\sum 1\otimes x_i-s\circ m(1\otimes x_i)=\sum 1\otimes x_i-\psi^{-1}(x_i)\otimes 1$$ and the proof is complete.
\end{proof}
We see that $\Lh_\mbb Q$ doesn't have a richer structure than $\mbb L_\mbb Q$, because it is equipped with a formal group law and a genus that corresponds to it, this is essentially the fact that in characteristic zero, there is only one formal group law. In a way $\Lh_\mbb Q$ is just a different way of looking at $\mbb L_\mbb Q$.
\begin{Cor}(Mishenko\footnote{There's a typo in the first appearance of the formula in the paper, the correct formula is in its appendix}, \cite[Appendix 1, p. 72]{Mishenko})\label{Mishenko}\par
We have through the identification $\mbb L_{\mbb Q}\simeq \Omega(k)_{\mbb Q}\simeq \mbb Q[\P^1, \P^2,...]$,$$\mfr h(u)=\sum_{i\geq 0} \frac{[\P^i]}{i+1} u^{i+1}$$
\end{Cor}
\begin{Rque}
We have a natural grading for the Lazard ring given by $\deg(a^{i,j})=i+j-1$, if we set $\deg(t_i)=i$ then we have a natural grading on $\Lh$ given by the grading on the tensor product, namely $\deg(a^{i,j}t_k)=k+i+j-1$
\end{Rque}
\begin{Rque} Before proceeding to the study of arithmetic cobordism let us introduce a couple notations defined below $$ \begin{matrix} \mbb Z[t_0,t_1,..] &=&\mbb Z[\t] & & &\\
\mbb Z[\t]\otimes \At(X)&=&\Att(X)&  \mbb Z[\t]\otimes A^{\bullet,\bullet}_{\mbb R}(X)&=&A^{\bullet,\bullet}_{[\t]}(X)\\
\mbb Z[\t]\otimes \Dt(X)&=&\Dtt(X)&\mbb Z[\t]\otimes D^{\bullet,\bullet}_{\mbb R}(X)&=&D^{\bullet,\bullet}_{[\t]}(X)\\
\mbb Z[\t]\otimes Z^{\bullet,\bullet}_{\mbb R}(X)&=&Z^{\bullet,\bullet}_{[\t]}(X) & & &\\
\Lh\otimes \Dt(X)&=&\DL(X) & & &\end{matrix}$$
where we've extended the usual operations defined on $D^{\bullet,\bullet}$, such as $\d, \db$, the pull back and push forward operations for suited maps etc..., by $\mbb Z[\t]$-linearity. Notice that we still have a product $$Z^{\bullet,\bullet}_{[\t]}(X)\otimes \Dtt(X) \to \Dtt(X)$$
that preserves $\Att(X)$
\end{Rque}
We now wish to construct both multiplicative characteristic forms associated to $\g$ with value in $A^{\bullet,\bullet}_{[\t]}(X)$, and secondary Bott-Chern forms with value in $\Att(X)$. This can be done in a straightforward manner, let's quickly review the way to do so.\par
If we have $X$ a complex manifold, with $\E$ a hermitian vector bundle over it, then $\E$ comes equipped with a natural Chern connection charracterized by  
\begin{enumerate}
\item preservation of the metric $d\bra s, t\ket=\bra \nabla s, t\ket +\bra s, \nabla t\ket$
\item compatibility with the $\db$ operator, $\nabla^{0,1}=\db$.
\end{enumerate} 
Let us consider the power series $\fii(T_1,...,T_n)\in Z[\t][[T_1,...,T_n]]$ defined by $$\fii(T_1,...,T_n)=\prod_{i=1}^n \g(T_i)$$
we can write $\fii$ as a sum of $\fii^{(\ell)}$ with each $\fii^{(\ell)}$ homogenous of degree $\ell$ (in $T_1,...,T_n$). There exists a unique map, still denoted $\fii^{(\ell)}$ defined on matrix with coefficients in $A^{1,1}(X)$ and invariant by conjugation such that $$\fii^{(\ell)}\left(\begin{pmatrix}\omega_1 & & \\
& \ddots & \\
& &\omega_n
\end{pmatrix} \right)=\fii^{(\ell)}(\omega_1,...,\omega_n)$$
By identifying locally $\End(E)$ with the space of matrix with complex coefficients, we can define $$\g(\E)=\sum_k \fii^{(k)}(\frac{i}{2\pi}\nabla^2)\in A^{\bullet,\bullet}_{[\t]}(X)$$
We get a closed form whose cohomology class (with coefficients in $\mbb Z[\t]$) does not depend on the metric chosen on $\E$. Let's sum up the properties of this characteristic class.
\begin{Prop}\label{CharG}
The characteristic form $\g(\E)\in A^{\bullet,\bullet}_{[\t]}(X)$ associated to a hermitian bundle on a projective complex manifold $X$, satisfies the following properties \begin{enumerate}
	\item (Naturality) For any holomorphic map of complex manifold $f: Y\to X$ we have $\g(f^*\E)=f^*(\g(\E))$.
	\item (Definition for a line bundle) For a hermitian line bundle $\L$, we have $\g(\L)=\sum_r t_r c_1(\L)^r$.
	\item (Mulitplicativity) If $0\to \E'' \to \E \to \E'\to 0$ is an ortho-split exact sequence of hermitian bundles on $X$ we have $$\g(\E)=\g(\E')\g(\E'')$$
	\item (Closedness) The form $\g(\E)$ satisfies $d\g(\E)=0$
\end{enumerate}
\end{Prop}
\begin{Rque}
The closedness property can easily be deduced from Bianchi's second identity using the fact that the Chern connection is torsion free on a K\"ahler manifold.
\end{Rque}
\begin{Rque} As $\g(\E)^{(0)}=1$, the class $\g(\E)$ is invertible in $A^{\bullet,\bullet}_{[\t]}(X)$.
\end{Rque}
\begin{Rque}
If $f:\X \to \Y$ is a morphism between arithmetic varieties we will denote $\g(\ovl{T_f})$ for $\g(\ovl T_X)\g(f^*\ovl{T_Y})^{-1}$
\end{Rque}
We can now construct secondary forms with value in $\Att(X)$ to measure the defect of multiplicativity in the case of an arbitrary exact sequence of hermitian bundles.
\begin{Prop}
To each exact sequence $$\mcl E:0\to \E'\to \E\to \E''\to 0$$ of hermitian vector bundles on $X$ we can associate a form in $\Att(X)$, denoted $\wt{\g}(\mcl E)$ uniquely determined by the following properties \begin{enumerate}
	\item (Naturality) For any holomorphic map of complex manifold $f: Y\to X$ we have $\wt\g(f^*\mcl E)=f^*(\wt\g(\mcl E))$.
	\item (Differential equation) We have $$\g(\E)=\g(\E')\g(\E'')+dd^c\wt\g(\mcl E)$$
	\item (Vanishing) When $\mcl E$ is orhto-split, $\wt\g(\mcl E)=0$
	\end{enumerate}
\end{Prop}
\begin{proof}
This is the same proof as \cite[Theo 1.2.9]{BGS2}, almost verbatim.
\end{proof}
\begin{Rque}
We can give an explicit expression of $\wt\g(\mcl E)$, if $$0\to\E'\to \E\to \E''\to 0$$ is a short exact sequence of hermitian bundles, we set $\wt E$ to be $(p_2^*E\oplus p_2^*E'(1))/p_2^*E'$ over $\P^1_X$ where $\O(1)$ is equipped with its Fubini-Study metric, and we endow $\wt E$ with any metric rendering isometric the isomorphisms over the fibers of $\wt E$ at 0 and $\infty$ with $\E$ and $\E'\oplus \E''$ respectively then we have $$\wt\g(\mcl E)=-\int_{\P^1_X/X}\log|z|^2\g(\wt E)$$
See \cite[1.2]{GS1} for details.
\end{Rque}

\begin{Prop}(Naturality with respect to $\t$)\label{NatuT}\par
Let $R$ be a ring equipped with a morphism $\fii:\mbb Z[\t]\to R$ and let $\g_R$ be the formal power series over $R$ given by $\sum \fii(t_i)u^i$ then we have \begin{enumerate}
	\item In $A_R^{\bullet,\bullet}(X)$, $\g_R(\ovl{E})=\fii(\g(\ovl E))$ for every hermitian bundle $\E$ over a manifold $X$.
	\item In $\wt{A}_R^{\bullet,\bullet}(X)_R$, $\wt\g_R(\mcl{E})=\fii(\wt\g(\mcl E))$ for every exact sequence of hermitian bundles over $X$.
	where $\g(\E)$ (resp. $\wt\g(\mcl E)$) is obtained by the same process as above.
\end{enumerate}
\end{Prop}
\begin{proof}
The second point results from the first one as we have $$\wt\g_R(\E)=-\int_{\P^1_X/X}\log|z|^2\g_R(\wt E)$$ and the integration over the fiber commutes with $\fii$ by definition.\par
Let us prove the first one, in fact the properties mentioned in \ref{CharG}, characterize the form $\g_R$ over $R$, but using the fact that $\fii$ commutes with $d$ and pull-back ensures us that $\fii(\g)=\g_R$ and the proposition follows.

\end{proof}
In every case that we will consider $R$ will be a subring of $\mbb R$.
If $$\mcl E: 0\to\E'\to \E\to \E''\to 0$$ is a short exact sequence of hermitian vector bundles, we set $$\wt{\g^{-1}}(\mcl E)=-\wt \g(\mcl E)\g^{-1}(\E'')\g^{-1}(\E')\g^{-1}(\E)$$ 
using the fact that $\g^{-1}$ is closed and the differential equation satisfied by $\wt\g$ we see that $$dd^c\wt{\g^{-1}}(\mcl E)=\g^{-1}(\E)-\g^{-1}(\E')\g^{-1}(\E'')$$
moreover $\wt{\g^{-1}}(\mcl E)=0$ as soon as $(\mcl E)$ is ortho-split, and natural with respect to pull back, thus it is the secondary Bott-Chern form (with coefficient in $\mbb Z[\t]$) associated to $\g^{-1}$.\par
\begin{Rque}
If $\alpha$ is a closed $(p,p)$-form (or current), we see that $$\alpha \g(\E)-\alpha \g(\E')\g(\E'')=\alpha dd^c\wt\g(\mcl E)=dd^c(\alpha\wt\g(\mcl E))$$
therefore if we work in $\DL(X)$, we have $\alpha \g(\E)=\alpha \g(\E')\g(\E'')$ for any closed form $\alpha$, a fact that we will use in the following.
\end{Rque}
Let us finish by some basic comments about degrees in $\DL(X)$, in the context of the weak cobordism group we will set (see \ref{Degree}) $$\deg(\wt D_{\mbb R}^{p,p}(X))=d_X-(p+1)$$
that gives us a graded group structure on $\DL(X)$, of course this grading is not compatible with the product of currents even when it is defined because our theory will be \emph{homological} in nature.\par
However we see that $\deg(\g(\ovl E)\fii(t)g)=\deg(\fii(t).g)$ (and the same thing for $\g^{-1}$), that $\deg(dd^c(\fii(t).g))=\deg(\fii(t)g)-1$ and therefore $\deg(\wt\g(\mcl E))=d_X$, these observations will ensure that the class
$$a[i_*[\wt{\g}(\mcl E)\g^{-1}(\ovl T_Z)]]+a(\g(\L)\log\|s\|^2)$$
that will appear later, is homogenous of degree $d_X-1$.
\subsection{Construction of the Borel Moore Functor}
We first construct a Borel-Moore functor on arithmetic varieties
\begin{Rque}
In the following sections, many of our definitions could still make sense for arithmetic varieties over a Dedekind domain, or even a Dedekind scheme, however since we do not know how to prove even the geometric version of the properties of the arithmetic cobordism group I've chosen to remain in the context of varieties over a field, where the geometric theory is known to be well behaved.
\end{Rque}
\begin{Def}Let $\ovl X$ be an arithmetic variety over $k$.
We set $\mcl Z(\X)$ as the group $$\mcl Z'(\X)/\mcl R'(\X)\times \left[\mbb Z[\t]\otimes \Dt(X) \right]$$ 
where $\mcl Z'(\X)$ denotes the free abelian group built on symbols $$[\ovl Z\ds f \ovl X, \ovl L_1,...,\ovl L_r]$$
with
\begin{itemize}
	\item The morphism $f$ is a projective morphism between arithmetic varieties.
	\item The variety $\ovl Z$ is integral (connected).
	\item The line bundle $\L_i$ is a hermitian line bundle over $Z$.
\end{itemize}
The group $\mcl R'(\X)$ denotes the subgroup of $\mcl Z'(\X)$ generated by the classes $$[\ovl Z\ds f \ovl X, \ovl L_1,...,\ovl L_r]-[\ovl Z'\ds {f'} \ovl X, \ovl L'_1,...,\ovl L'_r]$$
such that there exists $h$ an $X$-isometry of $\ovl Z$ on $\ovl Z'$, that is to say an isomorphism $$\xymatrix{
   & Z \ar[rr]^h \ar[rd]_f& &Z'\ar[ld]^{f'}\\
     & &X \\}$$
inducing an isometry from $\Z(\mbb C)$ to $\Z'(\mbb C)$; and such that there exists a permutation $\sigma\in \mfr S_r$ and isomorphisms of hermitian line bundles $\ovl L_i\simeq \ovl L'_{\sigma(i)}$, in other words, we allow re-indexing of the (classes of) hermitian line bundles.
\end{Def}
\begin{Rque}In other, simpler, terms, we make no difference between two arithmetic varieties as long as they are isometric, ibidem for line bundles and we allow to permute the line bundles.
\end{Rque}
We naturally have a map $$a: \left\{\begin{matrix}\Dtt(X)&\to &\mcl Z(\X)\\ \fii(t).g &\mapsto & (0,\fii(t).g)\end{matrix}\right.$$

We will sometimes write $[\ovl Z\ds f \ovl X, \ovl L_1,...,\ovl L_r,\fii(t).g]$ for the element $[\ovl Z\ds f \ovl X, \ovl L_1,...,\ovl L_r]+a(\fii(t).g)$.
\newline The group $\mcl Z(\X)$ is equipped with a natural grading, defined in the following way.

\begin{Def}\label{Degree}
We set $\deg([\ovl Z\to \ovl X, \ovl L_1,...,\ovl L_r])=d_Z-r$, $\deg(\wt D_{\mbb R}^{p,p}(X))=d_X-(p+1)$, and $\deg(t_i)=i$.
On set $\mcl Z_d(\X)$ the subgroup of $\mcl Z(\X)$ of $d$-degree, and we shall note $\mcl Z_\bullet(\X)$ the graded group.
\end{Def}
\begin{Rque}
A word should be said about our conventions on coproducts and the extension of the definition to finite disjoint union of arithmetic varieties. Notice that the degree of $\Dtp p(X)$ depends on $p$ and also on the dimension of $X$, thus is $\Y=\Y_1\amalg \Y_2$, with $\dim_k(Y_1)\neq \dim_k(\Y_2)$ then $\Dtp p(Y)$ is not homogenous.
\end{Rque}
\begin{Rque}
If $\Z\ds i \X$ is the closed immersion of a smooth divisor, then $[\Z\to \X]$ has degree $d-1$, where $d$ is the dimension of $X$, and a Green current for $Z$ is given by a current of $\wt D_{\mbb R}^{0,0}(X)$, which has degree $d-1$, hence for such a current and for any $(1,1)$ closed smooth form $\omega$, the class $[\Z\to \X]-a(\g(\omega)\wedge g)$ is homogenous of degree $d_X-1$, which should explain the different choices in the grading, that differ slightly from the usual ones used in Arakelov geometry where we tend to grade by the codimension, which is not possible here.
\end{Rque}

\begin{Rque}
We will call a class of the form $[\ovl Z\to \ovl X, \ovl L_1,...,\ovl L_r]+a(g)$ a standard class, and we will refer to the term $[\ovl Z\to \ovl X, \ovl L_1,...,\ovl L_r]$ as the geometric part of the class, and to the term $a(g)$ as the analytic part. A class $[\ovl Z\to \ovl X]$ will be called a purely geometric class.
\end{Rque}

\subsection{Dynamics of the group $\mcl Z(X)$}
Let's have a closer look on the functoriality properties of the group $\mcl Z(\X)$.
\begin{Def}(\label{PF}Push-forward)\\
Let $\pi: \ovl X\to \ovl Y$ be a projective morphism between arithmetic varieties, we define $$\pi_*[\ovl Z\ds f \ovl X, \ovl L_1,...,\ovl L_r,g]=[\ovl Z\ds{\pi \circ f}  \Y, \ovl L_1,...,\ovl L_r,\pi_*(g\wedge \g^{-1}(\ovl{T_\pi}))]$$
we extend this morphism by linearity and we get a morphism $$\pi_*:\mcl Z(\ovl X)\to\mcl Z(\ovl Y)$$ whose functoriality is easy to verify.
\end{Def}
\begin{Rque}
Let us note that if $\pi$ is a projective morphism between smooth equidimensional varieties, and if $d$ designs the relative codimension of $\pi$, then $\pi_*$ induces a morphism from $D_{\mbb R}^{p,p}(X)$ to $D_{\mbb R}^{p-d,p-d}(Y)$, as $\dim(Y)-\dim(X)=-d$, as $\deg(\g(\ovl{T\pi})g)=\deg(g)$, we have $\deg(\pi_*a(g)))=\dim(Y)-p+d=\dim(X)-p=\deg(a(g))$, thus $\pi_*$ is a graded morphism.
\end{Rque} 
\begin{Rque}
Notice here, that we have been a bit sloppy and used the same notation for two different things: the natural push forward of currents and the "twisted" push-forward of currents are both denoted $\pi_*$. 
\end{Rque}

It is also possible to define the pull back of any element in $\mcl Z(\X)$ along a smooth morphism.

\begin{Def}(\label{PB}Pull-back)\\
Let $f: \ovl S\to \ovl X$  be a smooth equidimensional morphism between arithmetic varieties, we define $$f^*[\ovl Z\ds f \ovl X, \ovl L_1,...,\ovl L_r,g]=[\ovl{Z\times_X  S}\ds p_2 \ovl S, \ovl{p_1^*L_1},...,\ovl{p_1^*L_r},f^*(g)]$$
The metric on $Z\times_X  S$ is defined in the following way, as $X/k$ is separated, we have a closed immersion  $Z\times_X  S\to Z\times_k  S$, which gives an embedding $T_{Z\times_X  S/k}\to T_{Z\times_k  S/k}\simeq p_1^*T_{Z/k}\oplus p_2^*T_{S/k}$, this former bundle being equipped with a natural metric, we can induce this metric on $T_{Z\times_X  S/k}$.\par 
We extend this morphism by linearity and we get a morphism $$f^*:\mcl Z(\ovl X)\to\mcl Z(\ovl Y)$$ whose functoriality is easy to verify.
\end{Def}
\begin{Rque}
Here again, for equidimensional varieties (e.g connected), this morphism is a graded morphism with degree the relative codimension $\delta=\dim(S)-\dim(X)$.
\end{Rque}
At last, it is also possible to define a first Chern class operator.
\begin{Def}(\label{C1}First Chern Class)\\
Let $\ovl L\in \wh{\Pic}(X)$ be a hermitian line bundle over $\ovl X$, we define $$\c1(\ovl L)[\ovl Z\ds f \ovl X, \ovl L_1,...,\ovl L_r,g]=[\ovl Z\ds f \ovl X, \ovl L_1,...,\ovl L_r,\ovl{f^*L},c_1(\L)\wedge\g(\L)\wedge g]$$
We extend this morphism by linearity and we get a morphism $$\c1(\L):\mcl Z_\bullet(\ovl X)\to\mcl Z_{\bullet-1}(\ovl X)$$
\end{Def}
\begin{Rque}
It will be useful to keep in mind the "different parts" of the action of $\c1(L)$, on geometric classes we have $$\c1(\ovl L)[\ovl Z\ds f \ovl X, \ovl L_1,...,\ovl L_r]=[\ovl Z\ds f \ovl X, \ovl L_1,...,\ovl L_r,\ovl{f^*L}]$$
whereas on analytic classes $\c1(\L)$ acts by multiplication by $c_1(\L)\g(\L)$, which we will sometimes denote $\h^{-1}(\L)$ because it is the (composition) inverse of the $\h$ class we've defined earlier\footnote{A word of warning, $\g^{-1}$ denotes the multiplicative inverse of $\g$ whereas $\h^{-1}$ denotes the composition inverse of $\h$, it maybe unfortunate to use the same notation for two different things, but it shouldn't confuse the reader as $\h$ does not have any multiplicative inverse, and $\g$ doesn't have any composition one.}.
\end{Rque}
We list in the next proposition, the different compatibility properties between these morphisms.
\begin{Prop}
\label{Compatibility}
Let $\X, \Y$ and $\S$ be arithmetic varieties, and let $\pi: \Y\to \X$ be a projective morphism, $f:\X\to \S$ a smooth equidimensional morphism, and $\ovl L$ a hermitian line bundle over $X$.
\begin{enumerate}
	\item Over $\mcl Z(\X)$, $\pi_*\circ\c1(\pi^*\L)=\c1(\L)\circ \pi_*$.
	\item Over $\mcl Z(\X)$, $f^*\circ \c1(\L)=\c1(f^*\L)\circ f^*$.
	\item Over $\mcl Z(\X)$, $\c1(\L)\circ \c1(\M)=\c1(\M)\circ \c1(\L)$.
\end{enumerate}
Finally, if we have a fiber diagram $$\xymatrix{
    X'\ar[r]_{t'} \ar[d]_{\pi'} &X\ar[d]^{\pi}\\
     S'\ar[r]^t& S \\}$$
		with $\pi$ projective, and $t$ smooth equidimensional, and $X'=X\times_S S'$ equipped with its natural metric, then $$\pi'_*t'^*=t^*\pi_*$$
\end{Prop}
\begin{proof} It suffices to check all assertions on standard classes, as any standard class $[\Z\to \ovl S, \L_1,...,\L_r,g ]$ can be written as $[\Z\to \ovl S, \L_1,..,\L_r]+a(g)$, it is enough to check the identities on both summands.\par
At the level of the analytic term, the identity $f^*\circ \c1(\L)(a(g))=\c1(f^*\L)\circ f^*(a(g))$ is a consequence of the naturality of the $\g$-class, the naturality of the action of Chern forms on differential smooth forms and the density of smooth forms in the space of currents.\par 
This identity remains true for $f$ projective if we replace the current $g$ by a smooth form, and this, in turns, implies the first one by duality. The first three identities are evident enough for the geometric term.\par 
Let's prove the last one, here again only the $a(g)$ term is not a priori clear, we need to prove that $\pi'_*t'^*a(g)=t^*\pi_*a(g)$, so in other words $$\pi'_*(\g^{-1}(\ovl{T_{X'/S'}})t'^*g)=t^*(\pi_*(\g^{-1}(\ovl{T_{X/S}})g))$$
Now using the naturality of the $\g$-class and the fact that for a Cartesian diagram such as the one in the proposition we have $\g^{-1}(\ovl T_{X'/S'})=t'^*\g^{-1}(\ovl T_{X/S})$  we only need to prove that for any current $\eta$, we have $$\pi'_*t'^*(\eta)=t^*\pi_*(\eta)$$
By duality, it is sufficient to prove that for any smooth compactly supported form $\omega $ on $S'$ we have $\pi^*t_*\omega=t'_*\pi'^*\omega$.  but this is tantamount to proving that $$\int_{X'/X}\pi'^*\omega=\pi^*\int_{S'/S}\omega$$.\par
Notice that $S'$ being proper over $k$, and $S$ being separated over $k$, $t$ is proper, and thus closed, but it is also open because it is flat, we can thus assume that $t(S')$ is a connected component of $S$, and even surjective by making the base changing to the connected component in question. By Ehresmann theorem \cite[Thm 2.4, p. 64]{Ehresmann}, we can thus assume that $S'\to S$ is a proper fibration of typical fiber $F$.\par
Let $(U_i)$ be an open cover of $S$, trivializing the fibration $t$, and let $\mu_i$ be a partition of unity associated with $U_i\times F$, which is an open cover of $S'$. We can choose $U_i$ small enough so that it is isomorphic to an open subset of $\mbb C^n$ As by its very definition, for any smooth form $\omega$, $t_*(\omega)=\sum_i t_*(\mu_i\omega)$, and using the linearity of $\pi^*$, we may assume that $\omega$ is compactly supported in a open subset of the form $U_i\times F$, and can thus be written as a sum of $\alpha \wedge \beta $, where $\alpha$ (resp. $\beta$) is the pull-back of a smooth form on $U_i$ (resp. $F$)\par
But then both sides of the identity we want to prove are equal to $\pi^*(\alpha)\wedge \int_F \beta$
\end{proof}

\subsection{Saturation of a subset of $\mcl Z(\X)$}
Assume we've been given, for every arithmetic variety, $\ovl Y$, an assignment $\Y\mapsto \mu(\Y)\subset\mcl Z(\Y) $.
\begin{Def}(Saturation of $\mu$)\\
We call the saturation of $\mu$ (if it exists) and we denote $\bra \mu \ket$, the map $X\mapsto \bra \mu \ket(X)$, where $\bra \mu \ket(X)$ is the smallest class of subgroups of $\mcl Z(\X)$ satisfying, for every projective morphism $\pi:Y\to X$, for every smooth equidimensional morphism $f:X\to S $, and for every hermitian line bundle $\ovl L\in \wh{\Pic}(X)$,
$$\pi_*(\bra \mu \ket(\Y))\subset \bra \mu \ket(\X); f^*(\bra \mu \ket(\ovl S))\subset \bra \mu \ket(\X); \c1(\ovl L)(\bra \mu \ket(\X))\subset \bra \mu \ket(\X)$$
\end{Def}
\begin{Prop}
If the mapping $\mu$ is such that for every $X$, $\mu(X)$ consists of homogenous elements, then the saturation $\bra \mu \ket$ exists, and the quotient $\mcl Z_{\mu}(X)$ inherits a natural grading from $\mcl Z(\X)$.
\end{Prop}
\begin{proof}Let us notice that every standard class in $\mcl Z(\X)$ verifies 
\begin{eqnarray*}[\Z\ds f \X, \L_1,...,\L_r,g]&=&[\Z\ds f \X, \L_1,...,\L_r]+a(g)\\
&=&f_*[\Z\to \Z, \L_1,...,\L_r]+a(g)\\
&=&f_*\circ \c1(\L_r)\circ...\circ \c1(\L_r)[\Z\to \Z]+a(g)\\
&=&f_*\circ \c1(\L_r)\circ...\circ \c1(\L_r)\pi_{\Z}^*(1_k)+a(g)\end{eqnarray*} 
For every arithmetic variety $\ovl Y$, set $\bra \mu\ket(\Y)$ the subgroup of $\mcl Z(\Y)$ generated by the set $$A(\Y)=\{f_*\circ \c1(\L_r)\circ...\circ \c1(\L_r)\pi^*(\alpha)| \alpha \in \mu(\Z), \pi:\ovl T\to \ovl Z \mbox{ projective}, f:\ovl T\to \ovl Y \mbox{smooth}, \L_i \in \wh{\Pic}(T) \}$$
we're left to check that the set $A(\Y)$ is mapped to $A(\Z)$ (resp. $A(\ovl S)$) under the action of a projective (resp. smooth equidimensional) morphism from $Y \to Z$ (resp. from $S$ to $Y$). But this results simply from \ref{Compatibility}.\par
The fact that the quotient is naturally graded if $\mu$ takes only subset of homogenous elements in $\mcl Z$ as value, and the fact that pull-backs, push-forwards and first Chern class operators preserve the grading is immediate.
\end{proof}
If the saturation of $\mu$ exists, we shall denote $\mcl{Z}_{\mu}(X)$ the (possibly graded) quotient $\mcl Z(\X)/\bra\mu\ket(X)$.

\subsection{The final construction}
We will now impose the relations that'll turn our basic object $\mcl Z(\X)$ into an object with a real geometric and arithmetic significance, for that we need to impose the following three relations \begin{eqnarray*}(\DIM)& [\Y\to \X, \L_1,...,\L_{d+2}]=0 
\end{eqnarray*}
 for $d=\dim(Y)$.
\begin{eqnarray*}(\SECT)& [\X\to \X, \L]+a[i_*[\wt{\g}(\mcl E)\g^{-1}(\ovl T_Z)]]=[\Z\to \X] -a(\g(\L)\log\|s\|^2) \end{eqnarray*}
with $s$ a section of $\L$ with smooth zero scheme, and $\|\cdot\|$ the norm induced by the norm on $\L$, where $\mcl E$ is the exact sequence $$\mcl E: 0\to \ovl{T_Z}\to i^*\ovl{T_X}\to i^*\L\to 0$$
and
\begin{eqnarray*}(\FGL)& \c1(\L\otimes \M)=\mbb F(\c1(\L),\c1(\M)) \end{eqnarray*}
where $\mbb F$ is the universal formal law group.\\

Notice that in the axiom $\SECT$ the scheme $Z$ is smooth but not necessarily irreducible (e.g when $X$ is a curve), if $Z$ is not connected, then the notation $[Z\to X]$ is simply meaning the sum of the $[Z_i\to X]$ where the $Z_i$'s are the connected components of $Z$. This is consistent with our conventions on coproducts of arithmetic varieties and in any case $[\Z\to \X]=i_*(1_{\Z})=i_{1*}(1_{\Z_1})+...+i_{p*}(1_{\Z_p})$.\\

In order to do this construction, we first need to impose the $(\DIM)$ condition, and to tensor over $\mbb Z$ by $\mbb L$ the Lazard ring, for the last relation to make any sense.\par
As the set $\SECT+\DIM(X)=\{[Y\to X, \L_1,...,\L_r]| r>\dim(Y)\}\cup \{[X\to X, \L]-[Z\to X] +a(\log\|s\|^2)| s \mbox{ smooth section of } \L \}$ is made up of homogenous elements, we can consider the graded group $\mcl Z_{\DIM, \SECT,\bullet}(X)$.\par

Let's now finish the construction of arithmetic cobordism, we set $\widecheck{ \mcl Z}(X)=\mbb L\otimes_{\mbb Z}\mcl Z(\X)_{\DIM, \SECT} $, we can grade this group via the natural grading on both factors. It is naturally a $\mbb L$-module, and we can extend all operations defined in the previous section, by linearity and we can prove the analog of \ref{Compatibility} for $\mbb L$-modules.

\begin{Def}(Arithmetic weak Cobordism)\\
We set $$ \cob(X)=\widecheck{\mcl Z}_{\FGL}(X)$$ where $$\FGL(X)=\{\c1(\L\otimes \M)(1_X)=\mbb F(\c1(\L),\c1(\M))(1_X)\}\cup \{\c1(\L\otimes \M)(a(g))=\mbb F(\c1(\L),\c1(\M))(a(g))\}$$ It is a graded $\mbb L$-module that we will call the \emph{arithmetic weak cobordism group of $X$}.
\end{Def}
\begin{Rque}Notice that the operator $\c1(\L)$ being locally nilpotent (i.e for every $a$, there exists $n>0$, such that $\c1(\L)^n(a)=0$), the term $\mbb F(\c1(\L),\c1(\M))$ does make sense.\end{Rque}

\begin{Prop}\label{RelationsLh}
Let $\X$ be an arithmetic variety, the map $a: \mbb L\otimes \mbb Z[\t]\otimes \Dt(X) \to \cob(\X)$ factors through $\DL(X)$, we will still denote by $a$ this map $\DL(X)\to \cob(\X)$
\end{Prop}
\begin{proof}
The proof hinges on the following key remark $$\int_{\P^r\times \P^\ell}c_1(p_1^*\O(1))^ic_1(p_2^*\O(1))^j=\delta_{ir}\delta_{jl}$$
To exploit this we will compute $$I_{rl}=\int_{\P^r\times \P^\ell}\c1(p_1^*\O(1)\otimes p_2^*\O(1))a(1)$$
in two different ways. To ease notations we will simply write $u$ for $c_1(p_1^*\O(1))$ and $v$ for $c_1(p_2^*\O(1))$.\par
On the one hand, using $\FGL$ and the key remark we see that $$I_{rl}=\sum_{i,j}a^{i,j}\int_{\P^r\times \P^\ell}(\h^{-1}(u))^i(\h^{-1}(v))^j=\sum_{i,j}a^{i,j}\left[(\h^{-1}(u))^i(\h^{-1}(v))^j\right]^{\{(r,l)\}}$$
where $\{(r,l)\}$ denotes the coefficient in front of $u^{r}v^{l}$.\par
On the other hand using the explicit expression of the action of the first Chern operator on $a(1)$ we see that $$I_{rl}=(\h^{-1}(u+v))^{\{(r,l)\}}$$
we thus have $$(\h^{-1}(u+v))^{\{(r,l)\}}a(1)=\sum_{i,j}a^{i,j}\left[(\h^{-1}(u))^i(\h^{-1}(v))^j\right]^{\{(r,l)\}}a(1)$$
in $\cob(k)$, and this pulls back to the same relation in $\cob(\X)$ but those are exactly the relations between the $t_i$'s and the $a^{i,j}$'s in $\Lh$, so the proof is complete.
\end{proof}
We also have an obvious forgetful map $\zeta:\cob(\X)\to \Omega(X)$ sending analytic classes to $0$ and a geometric class $[\Y\to \X, \L_1,...,\L_p]$ to $\zeta([\Y\to \X, \L_1,...,\L_p])=[Y\to X, L_1,...,L_p]$. The fact that this map is well defined and surjective follows immediately from the construction.

\subsection{A remark on Borel-Moore Functors}
We need to restrict the notion of Borel-Moore functor introduced in \cite{LevineMorel}. The reason for this is that the smallest class Levine and Morel consider to define a Borel-Moore functor is the class of quasi-projective smooth varieties over a field $k$ whereas we are solely interested in the class of projective smooth varieties.  We refer the reader to \cite{LevineMorel} for notations and vocabulary that we may not define.
\begin{Def}(Projective Borel-Moore functor, compare with \cite[p.13]{LevineMorel})\par
Let $R$ be a graded ring, we call a (graded) projective $R$-Borel-Moore functor an assignment $X\to H_\bullet(X)$ for each $X$ projective and smooth over $k$, such that 
  \begin{enumerate}
	  \item $H_\bullet(X)$ is a (graded) $R$-module
    \item (direct image homomorphisms) a homomorphism $f_* : H_\bullet(X) \to H_\bullet(Y)$ of degree zero for each projective morphism $f : X \to Y$,
    \item (inverse image homomorphisms) a homomorphism $f^* : H_\bullet(Y) \to H_\bullet(X)$ of degree $d$ for each smooth morphism $f : X \to Y$ of relative dimension $d$,
    \item (first Chern class homomorphisms) a homomorphism $c_1(L) : H_\bullet(X) \to H_\bullet(X)$ of degree -1 for each line bundle $L$ on $X$,
  \end{enumerate}
satisfying the axioms
  \begin{enumerate}
    \item the map $f \mapsto f_*$ is functorial;
    \item the map $f \mapsto f^*$ is functorial;
    \item if $f : X \to Z$ is a projective morphism, $g : Y \to Z$ a smooth equidimensional morphism, and the square
      $$\xymatrix{ W \ar[r]^{g'}\ar[d]_{f'}& X \ar[d]^{f} \\
        Y \ar[r]^{g}& Z\\}$$
    is Cartesian, then one has 
      \[ g^* \circ f_* = f'_* \circ g'^*  \]
    \item if $f : Y \to X$ is projective and $L$ is a line bundle on $X$, then one has
      \[ f_* \circ c_1(f^*(L)) = c_1(L) \circ f_*  \]
    \item if $f : Y \to X$ is a smooth equidimensional morphism and $L$ is a line bundle on $X$, then one has
      \[ c_1(f^*(L)) \circ f^* = f^* \circ c_1(L)  \]
    \item if $X$ is a projective smooth variety and $L$ and $M$ are line bundles on $X$, then one has 
      \[ c_1(L) \circ c_1(M) = c_1(M) \circ c_1(L)  \]
      \end{enumerate}
\end{Def}
We will only be interested in projective Borel-Moore $\mbb L$-functor of a particular type.
\begin{Def}(Geometric type)\par
A \emph{projective oriented Borel-Moore functor with product} is the data of a projective oriented Borel-Moore functor together with the data of
    \begin{enumerate}
      \item (external product) a bilinear graded multiplication map
        \[ \times : H_\bullet(X) \times H_\bullet(Y) \to H_\bullet(X \times Y) \]
      which is associative, commutative, and admits a unit $1_K \in H_0(\Spec k)$,
    \end{enumerate}
satisfying the axioms
    \begin{enumerate}
      \item  if $f : X \to Y$ and $g : X' \to Y'$ are projective morphisms, one has the equality
        \[ \times \circ (f_* \times g_*) = (f \times g)_* \circ \times : H_\bullet(X) \times H_\bullet(X') \to H_\bullet(Y \times Y'); \]
      \item  if $f : X \to Y$ and $g : X' \to Y'$ are smooth equidimensional morphisms, one has the equality
        \[ \times \circ (f^* \times g^*) = (f \times g)^* \circ \times : H_\bullet(Y) \times H_\bullet(Y') \to H_\bullet(X \times X'); \]
      \item  if $L$ is a line bundle on $X$, and $\alpha \in H_\bullet(X)$, $\beta \in H_\bullet(Y)$, then one has the equality
        \[ c_1(L)(\alpha) \times \beta = c_1(p^*(L))(\alpha \times \beta) \]
      in $H_\bullet(X \times Y)$.
    \end{enumerate}
We will say that an $\mbb L$ projective Borel-Moore functor with product, $H_\bullet$ is of geometric type if the following additional properties are satisfied
  \begin{enumerate}
    \item(Dim)  For $X$ a smooth projective variety and $(L_1, \ldots, L_n)$ a family of line bundles on $X$ with $n > \dim(X)$, one has
      \[ c_1(L_1) \circ \cdots \circ c_1(L_n)(1_X) = 0 \]
    in $H_\bullet(X)$.
    \item (Sect)  For $X$ a smooth projective variety, $L$ a line bundle on $X$, and $s$ a section of $L$ which is transverse to the zero section, one has the equality
      \[ c_1(L)(1_X) = i_*(1_Z) \]
    where $i : Z \to X$ is the closed immersion defined by the section $s$.
		\item  (FGL)  There exists a formal law group $F_H$ on $\mbb L$ such that, for $X$ a smooth projective variety and $L,M$ line bundles on $X$, one has the equality
      \[ F_H(c_1(L), c_1(M))(1_Y) = c_1(L \otimes M)(1_Y) \]
    where $F_H$ acts on $H(X)$ via its $\mbb L$-module structure. Moreover we require the different pull-backs and push-forward maps to preserve $F_H$.
  \end{enumerate}
\end{Def}
\begin{Rque}
Two classical examples of (projective) Borel-Moore functor of geometric type are given by $\CH$ and $K_0$ (the latter being non graded\footnote{It is possible to render it graded by considering $K_0(X)\otimes \mbb Z[\beta,\beta^{-1}]$ where $\beta$ is an indeterminate of degree 1.}). In fact one can show (\cite[Thm 1.2.2 and Thm 1.2.3]{LevineMorel}) that $\CH$ is the universal additive\footnote{that means that the formal law group is given by the ordinary addition} Borel-Moore functor of geometric type, while $K_{0}$ is the universal multiplicative unitary\footnote{that means that the formal law group is given by $F(u,v)=u+v-uv$} Borel-Moore functor of geometric type, at least over fields of characteristic zero.
\end{Rque}
\begin{Rque}
To illustrate the depth of the fact that $K_0$ is a universal Borel-Moore functor, let us just note that it readily implies Grothendieck-Riemann-Roch theorem, after noting that we can obtain a multiplicative Borel-Moore functor out of Chow groups, by twisting the direct image by the Todd class.  \cite{LevineMorel}.
\end{Rque}
Getting back to our problem, we can easily construct a universal projective Borel-Moore functor of geometric type by following the exact same procedure as in \cite{LevineMorel}. But what's a bit less obvious it that such a functor should coincide with the restriction of $\Omega$ to the category of smooth projective varieties.\par
The reason for this is that we may have some relationships in $\Omega(X)$ that may "come from quasi-projective varieties", that is to say that in $\Omega(X)$ we may observe the vanishing of classes of the form $\pi_*(a)$ (resp. $f^*(a)$) for $a$ a vanishing class in $\Omega(Y)$ with $Y$ quasi-projective.\par
The first case is totally innocent of course, because the composition of projective morphisms is projective, so no relation in $\Omega(X)$ with $X$ projective can come from the cobordism ring of a quasi-projective variety. Let us take care of the second case.
\begin{Prop}
The cobordism functor restricted to the category of projective smooth varieties is the universal projective Borel-Moore functor of geometric type.
\end{Prop}
\begin{proof}
It is easy to give explicit generators for the saturation with respect to the relations we want to impose (see \cite[Lemma 2.4.2; 2.4.7 and Remark 2.4.11]{LevineMorel}, whose notations we will use).\par More precisely $\Omega(X)$ can be constructed as the quotient of $\mbb L\otimes \underline{\Omega}$ by the sub $\mbb L$-module generated by the relations $$f_*\circ c_1(L_1)\circ...\circ c_1(L_p)\left([L\otimes M]-[\mbb F(L,M)]\right)$$  where  $f$ is projective between smooth projective varieties.\par
Moreover $\underline{\Omega}$ is the quotient of $\underline{\mcl Z}(X)$ by the subgroup generated by relations $$[Z\to X, L_1,..,L_r]=[Z'\to X, i^*L_1,...,i^*L_{r-1}]$$ where $Z$ and $Z'$ are (of course) projective and smooth.\par
And that $\underline{\mcl Z}(X)$ is the quotient of $\mcl Z(X)$ by the subgroup generated by relation of the form \[ [Y\to X, \pi^*L_1,..\pi^*L_r, M_1,...,M_d]\tag{$\star$}\label{eq: Dimrel}\]
for every smooth equidimensional morphism $\pi: Y\to Z$ where where $Z$ is a smooth quasi-projective variety of dimension strictly lower than $r$.\par
It will be sufficient to prove that we can replace the relations of the form \eqref{eq: Dimrel} by the same relations but where $Z$ is a smooth \emph{projective} variety of dimension strictly lower than $r$.\par
To see this let $f:Y\to Z$ be a smooth quasi-projective morphism from a projective smooth variety to a quasi-projective smooth variety, then $f$ itself is projective and thus closed, but also open as it is flat, it is thus surjective and $Z$ itself is projective which proves the claim and the proposition. 
\end{proof}
From now on, we will only use the term Borel-Moore functor to mean a projective Borel-Moore functor.

\subsection{An exact sequence}
We will begin by a basic observation, notice that is $X$ if a smooth projective variety, the choice of the metric on $T_X$ doesn't change the structure of $\cob(\X)$. To be precise 
\begin{Prop}
The natural map $\ovl X \to \ovl X'$ gives an isomorphism of $\mbb L$-module $$\cob(\ovl X)\to \cob(\ovl X')$$
\end{Prop}
\begin{proof}
This is simply the functoriality of the push forward.
\end{proof}
\begin{Def}
Let $\ovl X$ be an arithmetic variety, we denote by $\cobna(\ovl X)$ the $\mbb L$-module $\cob(\ovl X)/a(\DL(X))$.
\end{Def}
\begin{Rque}
Notice that, obviously $\Omega(X)$ does not depend on the metric structure chosen on $X$, as a standard class $[\Z\to \X, \L_1,...,\L_r]$ is mapped to $[Z\to X, L_1,...,L_r]$, moreover $\zeta \circ a$ is of course trivial, we will still denote $\zeta$ the induced map from $\cobna(\X)$ to $\Omega(X)$.\par
We will usually denote $[\Z\to \X, \L_1,...,\L_r]\na$ for the image of $[\Z\to \X, \L_1,...,\L_r, g]$ in $\cobna(\X)$
\end{Rque}
With this definition it is not clear how $\cobna(\ovl X)$ depends on the choice of metric over $X$, even though its purely abstract $\mbb L$-module structure $\cobna(\ovl X)$ does not depend on it. In fact as a Borel-Moore functor $\cobna(\X)$ doesn't depend on the metric chosen on $X$ at all. Let us prove that essential fact.\par
Firstly, let us notice that we have commutative diagrams (when they're defined) $$\xymatrix{
    \DL(X)\ar[r]^{a} \ar[d]^{\pi_*} &\cob(Y) \ar[d]^{\pi_*}\\
     \DL(Y)\ar[r]_a&\cob(Y)  \\}, \xymatrix{
    \DL(X)\ar[r]^{a} \ar[d]_{f^*} &\cob(S) \ar[d]^{f^*}\\
     \DL(S)\ar[r]_a&\cob(S)  \\}, \xymatrix{
    \DL(X)\ar[r]^{a} \ar[d]_{\c1(\L)} &\cob(X) \ar[d]^{\c1(L)}\\
     \DL(X)\ar[r]_a&\cob(X)  \\}$$
		that ensure that the maps are well defined on the level of $\cobna$.
\begin{Lem}\label{Inv1}
Let $\X$ and $\X'$ be two arithmetic variety structures on the same underlying algebraic variety, and let $\pi:\X'\to \X$ be the identity morphism. Then $$\pi_*[\X'\ds{\Id} \X']=[\X \ds{\Id} \X] \text{ modulo } a(\DL(X))$$
\end{Lem}
\begin{proof}
Let $Y=\P^1_X$ be the projective line over $X$, and let $s_{[a:b]}$ be the section of $\O(1)$ over $\P^1$ defined $by-ax$, if we see $\P^1$ as $\Proj(k[X,Y])$, for any $(a,b)\in k^2$, $s_{[a:b]}$ is transverse to the zero section, and gives rise to an isomorphism$$ j_{[a,b]}^* T_{\P^1}\ds\sim j_{[a:b]}^*\O(-1)$$ if we look at the fiber square $$\xymatrix{X\ar[d]_{p_1}\ar[r]^{i_{[a,b]}} &Y\ar[d]^{p_1}\\
\Spec k\ar[r]^{j_{[a,b]}}&\P^1}$$ we have an exact sequence over $X$, $$0\to i_{[a,b]}^*p_2^* T_X\to i_{[a,b]}^*p_2^*T_X \oplus i_{[a,b]}^*p_1^*T_{\P^1}\to i_{[a,b]}^*p_1^*\O(1)\to 0  $$
the term in the middle being isomorphic to $i_{[a,b]}^*T_{Y}$. From now on let us assume that both $T_{\mbb P^1}$ and $\O(-1)$ are equipped with metric rendering isometric the isomorphisms given by $s_0$ and $s_\infty$. Now as $T_Y\simeq p_2^*T_X \oplus p_1^*T_{\P^1}$, we can choose on $T_Y$ a metric, say $h$ (resp. $h'$) such that we have an isometry $(T_Y,h) \simeq p_2^*\ovl{T_X} \oplus p_1^*\ovl{T_{\P^1}}$ (resp. $(T_Y, h') \simeq p_2^*\ovl{T_{X'}} \oplus p_1^*\ovl{T_{\P^1}}$). If $\fii$ is now any smooth function over $\P^1(\mbb C)$, such that $\fii(0)=1$ and $\fii(\infty)=1$, let us consider the metric $h''=\fii(t)h+\fii(1/t)h'$ over $Y$, the two following exact sequences are meager (because they're ortho-split)
$$0\to \ovl{T_X}\to i_{0}^*p_2^*\ovl{T_X} \oplus i_{0}^*p_1^*\ovl{T_{\P^1}}\to i_{0}^*p_1^*\ovl{\O(1)}\to 0  $$
$$0\to \ovl{T_{X'}}\to i_{\infty}^*p_2^*\ovl{T_{X'}} \oplus i_{\infty}^*p_1^*\ovl{T_{\P^1}}\to i_{\infty}^*p_1^*\ovl{\O(1)}\to 0  $$
Let us now consider the class $[\ovl Y\to \ovl Y, p_1^*\ovl{\O(1)}]$, where $\Y$ is equipped with the metric $h''$, by $\SECT$ we have, up to terms in $a(\DL(
\Y))$ \begin{eqnarray*}[\ovl Y\to \ovl Y, p_1^*\ovl{\O(1)}]&=&[\X\to \ovl Y]\\
&=&[\ovl X'\to \ovl Y]
\end{eqnarray*}
the result follows from pushing-forward along $p_2: \Y \to \X$.
\end{proof}

Let us further investigate the independence on the metrics in $\cobna(\X)$.
\begin{Lem}\label{Inv2}
Let $\X$ and $\Z$ be two arithmetic varieties, and $f$ any projective morphism between them, let us consider $L\in \Pic(Z)$ and let $h$ and $h'$ be two metrics on $L$, we have, $$[\Z\ds f \X, (L,h)]=[\Z \ds f \X, (L,h')]\text{ modulo } a(\DL(X))$$
\end{Lem}
\begin{proof}
This is the same idea as the previous lemma. Let us consider $\Y=\ovl{\P^1_X}$ equipped with its "horizontal" metric that induces the metric on $\X$ over the fiber at $0$ and $\infty$. Let us equip $p_1^* L$ with the metric $h''=\fii(t)h+\fii(1/t)h'$, where as before, $\fii$ is any smooth function over $\P^1(\mbb C)$, such that $\fii(0)=1$ and $\fii(\infty)=1$. Using $\SECT$, we see that the class $[\Y\to \Y, \ovl{p_1^*L}, \ovl{p_2^*\O(1)}]$ equals both $[\X\to \Y, i_0^*\ovl{p_1^*L}]$ and $[\X\to \Y, i_\infty^*\ovl{p_1^*L}]$, up to an analytic class, this yields $$[\X\to \Y, (L,h)]=[\X \to \Y, (L,h')]$$ and pushing forward along $p_1$ we get $$[\X\to \X, (L,h)]=[\X \to \X, (L,h')]$$
but this is enough to prove the proposition as $$[\Z\to \X, (L,h)]=f_*[\Z \to \Z, (L,h)]=f_*[\Z \to \Z, (L,h')]=[\Z \to \X, (L,h)]$$ and we are done.
\end{proof}
\begin{Prop}
Let $\X$ (resp. $\Z$) and $\X'$, (resp. $\Z'$) be two arithmetic variety structures on the same underlying algebraic variety, and let $\pi$ be the identity morphism. Assume that we've been given $L_1,..., L_r$, $r$ line bundles over $Z$, which we will equip with two sets of hermitian metrics, $\L_i$ and $\L'_i$ for each $i$. We have $$\pi_*[\Z'\to \X', \L'_1,...,\L'_r]\na=[\Z \to \X, \L_1,...,\L_r]\na$$
\end{Prop}
\begin{proof}
Let us denote, for precision's sake, the different morphisms as in the following commutative diagram $$\xymatrix{\Z\ar[d]_{f}\ar[r]^{i} &\Z'\ar[d]^{f'}\\
\X\ar[r]^{\pi}&\X'}$$ where $i$ the identity morphism from $\Z$ to  $\Z'$. Of course we have $i^*\L_i=\L_i$ because the morphism $i$ is the identity on the underlying variety. Using \ref{Inv1} and \ref{Inv2} we see that \begin{eqnarray*}[\Z\to \Z', \L_1]\na&=&\c1(\L_1)[\Z\to \Z']\na \\
&=&\c1(\L_1)[\Z'\to \Z']\na\\
&=&\c1(\L_1)i_*[\Z'\to \Z]\na\\
&=&i_*\c1(\L_1)[\Z'\to \Z]\na\\
&=&i_*\c1(\L'_1)[\Z'\to \Z]\na\\
&=&i_*[\Z'\to \Z, \L'_1]\na\\
&=&[\Z'\to \Z', \L'_1]\na\end{eqnarray*}By iterating, we see that $$[\Z\to \Z', \L_1,...,\L_r]\na=[\Z'\to \Z', \L'_1,...,\L'_r]\na$$ now pushing forward along $f'$ yields $$[\Z'\to \X', \L'_1,...,\L'_r]\na=[\Z\to \X', \L_1,...,\L_r]\na=\pi_*[\Z\to \X, \L_1,...,\L_r]\na$$ and the proof is complete.
\end{proof}
\begin{Cor}
Let us fix a choice of metric on every (isomorphism class of) algebraic smooth projective variety. We have a Borel-Moore functor associated to this choice given by $X \to \cobna(\X)$ for the specific choice of metric over $X$.\par
If we take two of these Borel-Moore functors associated to two different choices of metrics, they're naturally isomorphic.
\end{Cor}
From now on we will denote $\cobna(X)$ instead of $\cobna(\X)$ for this group, and we will omit the metrics when writing the elements of $\cobna(X)$. We shall now prove that we have in fact an isomorphism of Borel-Moore functor $$\cobna(\bullet)\ds \sim \Omega(\bullet)$$
In order to do this, we next show that in $\cobna(X)$ we have a stronger version of $\DIM$
\begin{Lem}\label{DimNA}
In $\cobna(X)$ , we have $$[Y\to X, L_1,...,L_r]\na=0$$ as soon as $r>\dim(Y)$. 
\end{Lem}
\begin{proof}
It is clear that we only need to prove $[Y\to Y, L_1,...,L_r]\na=0$, because pushing this formula will give the formula above.\par
Let us first show that it suffices to prove $[Y\to Y, L_1,...,L_r]\na=0$ where $L_1,...,L_r$ are very ample line bundles to prove the general case. Indeed, every line bundle on a projective variety may be written as $M\otimes M'$ where $M$ (resp. $M'$) is very ample (resp. anti very ample), therefore if one of the bundles, say $L_1$ is not very ample we have\footnote{Here $\chi$ denotes the formal inverse characterized by $\mbb F(u,\chi(u))=0$} $$[Y\to Y, L_1,...,L_r]\na=\sum a^{i,j}c_1(M)^i\chi (c_1(M'^\vee))^j[Y\to Y, L_2,...,L_r]\na$$
so it suffices to prove that $[Y\to Y, L_1,...,L_r]\na=0$ as soon as the $L_i$'s are very ample to ensure that this class vanishes.\par
Now using $\SECT$ we see that $$[Y\to Y, L_1,...,L_r]\na=[\emptyset \to Y]\na=0$$
and the result follows.
\end{proof}
\begin{Rque}
We have a natural external product structure of $\cobna$ given by $$[Y\to X, L_1,...,L_r]\na\otimes [Z\to X', M_1,...,M_k]\na\mapsto [Y\times Z\to X\times X', p_1^*L_1,...,p_2^*M_k]\na$$
this endows $\cobna$ with the structure of a Borel-Moore functor with (external) products. The proof that this product is well defined is immediate.
\end{Rque}

In happy concord with what happens for other weak arithmetic theories we have
\begin{Prop}\label{ExSe}
We have an exact sequence $$\DL(X)\ds a \cob(\X) \ds \zeta \Omega(X)\to 0$$
\end{Prop}
\begin{proof}
We will prove that $\zeta$ induces an isomorphism of Borel-Moore Functor of geometric type between $\cobna$ and $\Omega$. We've already proven that $\cobna$ is a Borel-Moore functor with products and the fact that $\zeta$ is a morphism of Borel-Moore functor is obvious.
The fact that $\cobna$ is of geometric type is easy as axioms $(\SECT)$ and $(\FGL)$ when $\zeta$ is applied to them, give the usual $(\SECT)$ and $(\FGL)$ axioms of a Borel-Moore functor of geometric type, as for $(\DIM)$ this is \ref{DimNA}\par
So we get a map from $\Omega$ to $\cobna$, and a commutative diagram $$\xymatrix{
    \cob(\X) \ar[rr]^\zeta \ar[rd]_\zeta& &\cobna(X)\\
     & \Omega(X)\ar[ru]&\\}$$
		
		This map is an inverse of $\zeta$, to check this we need to check that the standard classes $[X\to Y, L_1,...,L_r]$ are left invariant by the application of $\Omega\to \cobna \to \Omega$, but that is obvious by construction.
	\par
	The proof is then complete.
\end{proof}

\begin{Rque}
It can be seen in the preceding proof that the exact sequence obtained is in fact an exact sequence of graded $\mbb L$-modules, and it splits into exact sequences $$\DLp{p}(X)\ds a \cob_p(\X) \ds \zeta \Omega_p(X)\to 0$$
where $\DLp{p}(X)$ denotes the degree $p$ part of $\DL(X)$ which is made up of $$\Dtp {d_X-p-1}(X)\oplus\Dtp {d_X-p}(X)\otimes \Lh_{1} \oplus...\oplus \Dtp {d_X}(X)\otimes \Lh_{p+1}$$

After tensorization by $\mbb Q$ we can give a more explicit decomposition as  $$\Dtp {d_X-p-1}(X)\oplus\mbb Q[\P^1]\Dtp {d_X-p}(X)\oplus...\oplus \mbb Q [\P^{p+1}]\Dtp{d_X}(X)$$
\end{Rque}
Anticipating a little we get the result that served as a guideline during the construction of $\cob$
\begin{Cor}
We have the following exact sequences (which we will refine in the following sections).
 $$\xymatrix{\DL(X)_p\ar[d]\ar[r]^{a} &\cob(\X)_{p} \ar[r]^\zeta &\Omega(X)_{p}\ar[d]\ar[r]& 0\\
\wt D^{d_X-p+1,d_X-p+1}_{\mbb R}(X)\ar[r]^a &\CHw_{p}(\X) \ar[r]^\zeta &\CH_p(X)\ar[r]& 0}$$
and 
$$\xymatrix{\DL(X)\ar[r]^a\ar[d] &\cob(\X) \ar[r]^\zeta &\Omega(X)\ar[d]\ar[r]& 0\\
\wt D^{\bullet,\bullet}_{\mbb R}(X)\ar[r]^a &\Kw_{0}(\X) \ar[r]^\zeta &K_0(X)\ar[r]& 0}$$
\end{Cor}

\subsection{Some computations}
In this section we will investigate more closely the different dependencies on the metric, by proving an anomaly formula.\par
According to \ref{ExSe} we see that the difference $$[\X\to \X]-[\X'\to \X]$$
where $\X'$ and $\X$ are two different arithmetic structures on the same underlying variety should lie in the image  of $a$, so it is a natural investigation to try and find an expression for that class. The answer is fairly simple and given by the
\begin{Prop}(Anomaly Formula)\label{Anomaly1}\par
Let $X$ be an algebraic projective smooth variety and let $\X, \X'$ and $\X''$ be three arithmetic structures on it, we have $$[\X'\to \X]-[\X''\to \X]=a(\g(\ovl T_X)\wt{\g^{-1}}(T_X, h', h''))$$
and 
$$[\X'\to \X]-[\X''\to \X]=a(-\g^{-1}(\ovl T_X)\wt{\g}(T_X, h', h''))$$

\end{Prop} 
\begin{proof}
Let's use our usual trick consisting of endowing $Y=X\times \P^1$ with a metric such that the fiber at 0 (resp. $\infty$) of $T_{Y}$ is isometric to the orthogonal sum of $\ovl{T_X'}$ (resp. $\ovl{T_X''}$)  and $\ovl{T_{\P^1}}$ where $\P^1$ is equipped with its Fubini-Study metric. Let us compute $$[\Y\to \Y, p_1^*\ovl{\O(1)}]=[\X'\to \Y]-a(\g(p_1^*\ovl{\O(1)})\log\|x\|^2)=[\X''\to \Y]-a(\g(p_1^*\ovl{\O(1)})\log\|y\|^2)$$
Therefore \begin{eqnarray*}[\X'\to \X]-[\X''\to \X]&=&\int_{\P^1_X/X}\log|z|^2\g(p_1^*\ovl{\O(1)})\g^{-1}(\ovl{T_{\P^1_X}})\g(p_2^*\ovl{T_X})\end{eqnarray*}
We obviously have $[\X'\to \X]-[\X''\to \X]=0$ as soon as $X'$ and $X''$ are isometric.

Let us define the following assignment.

Take $Y$ any projective smooth variety over $\mbb C$ and 0$$\mcl E: 0\to \E'\to \E\to \E''\to 0$$ an exact sequence of hermitian vector bundle over $Y$, let's map $(Y,\mcl E)$ to the following form (defined up to $\im \d+\im \db$) with coefficient in $\mbb L$, $$I(Y,\mcl E)=\int_{\P^1_Y/Y}\log|z|^2\g(p_1^*\ovl{\O(1)})\g^{-1}(\wt{\E})$$
where $\wt{\E}$ is the hermitian vector bundle on $\P_Y^1$ defined as $p_2^*E\oplus p_2^*E(1)/p_2^*E'$ equipped with a metric rendering it fiber at $0$ isomorphic to $\E$ and its fiber at $\infty$ isometric to $\E'\oplus \E''$.\\
We readily see that 

\begin{eqnarray*}dd^cI(Y,\mcl E)&=&\int_{\P^1_Y/Y}d_zd_z^c(\log|z|^2)\g(p_1^*\ovl{\O(1)})\g^{-1}(\wt{\E})\\
&=&i_0^*\g(p_1^*\ovl{\O(1)})\g^{-1}(\wt{\E})-i_\infty^*\g(p_1^*\ovl{\O(1)})\g^{-1}(\wt{\E})\\
&=&\g(i_0^*p_1^*\ovl{\O(1)})\g^{-1}(\E)-\g(i_\infty^*p_1^*\ovl{\O(1)})\g^{-1}(\E'\oplus \E'')\\
&=&\g^{-1}(\E)-\g^{-1}(\E'\oplus \E'')\end{eqnarray*}

It is clear that $I(Y, \mcl E)$ will vanish as soon as $\mcl E$ is orthosplit as in that case we can chose the metric on $\wt{E}$ to be constant along $\mbb P^1$, and that if $h:Y'\to Y$ is a holomorphic map, then we will have a cartesian diagram
$$\xymatrix{\P^1_{Y'}\ar[r]\ar[d]& \P^1_Y\ar[d]\\Y'\ar[r]& Y}$$ ensuring by the projection formula that $$h^*I(Y, \mcl E)=I(Y', h^*\mcl E)$$

Now looking at our construction we see that $[\X'\to \X]-[\X''\to \X]=a(I(X, \mcl F))$ for $\mcl F: 0\to 0\to (T_X,h)\to (T_X, h')\to 0$, thus we get that $$[\X'\to \X]-[\X''\to \X]=a(\g(\ovl T_X)\wt{\g^{-1}}(T_X, h', h''))$$
which completes the proof.

\end{proof}
In \cite{GSA}, Gillet and Soul\'e define a \emph{star-product} operator on Green currents, to be able to define an intersection pairing for arithmetic cycles.\par
To this end, given two irreducible closed subsets of the ambient variety, say $Z$ and $Y$, they need to select a specific green current for $Z$ in the family of all admissible green current, that satisfies a "logarithmic-growth" condition. They prove that such a Green current always exists, and that we can multiply it with $g_Y$ to get a green current for $[Z].[Y]$.\par
But if we look more closely at their construction, we see that we only need for a current to be of log-type singularities along the singular locus of another current to define a star-product between those two currents. This makes it possible to define an intersection pairing for divisors (or more precisely for classes of arithmetic divisors associated to hermitian line bundles), as in that case, we have a (family of) favored log-type current, namely $\log\|s\|^2$. It turns out that this construction is already embedded in the group $\cob(X)$.

\begin{Lem}\label{Transverse}
Let $\L_1,\L_2$ be two very ample hermitian line bundles over an arithmetic variety $\ovl X$. Then \begin{eqnarray*}[\X\to \X, \L_1,\L_2]&=&[\Z'\to \X]-a(\log\|s_2\|^2\g(\L_2)\g(\L_1)\delta_Z)+a(\h^{-1}(\L_2)\log\|s_1\|^2\g(\L_1))\\
& &-a(\g(\L_2)j_*(\wt{\g}( Z/X)\g(\ovl{T_{Z}})))+a(j_*(i_*\wt{\g^{-1}}(Z'/Z)\g(\ovl N_{Z'/X})\g(j^*\ovl T_X))))\end{eqnarray*}
where $\Z'$ (resp $\Z$) is the smooth locus $ \Div(s_1)\cap\Div(s_2)$  (resp. $ \Div(s_1)$) endowed with any metric, for $s_1$ and $s_2$, two sections of $L_1$ and $L_2$ whose zero locus are transverse to each other.\par
\end{Lem} 
Here the notation $\wt{\g}( A/B)$ means $\wt{\g}( \mcl N)$ where $\mcl N: 0\to T_A\to j^*T_B\to N_{A/B}\to 0 $ is the usual exact sequence with the metrics that we chose, the metrics on the normal bundles being induced by the metrics on the line bundles as usual.

\begin{proof}
Let us compute $$\c1(\L)[\Z\ds j \X, g]$$ where $j:Z\to X$ is a regular immersion of smooth integral varieties, $\L$ is a very ample hermitian line bundle on $X$ and $g$ is any current on $X$, and where $Z'$ is the smooth zero locus of a global section of $L$ over $X$ transverse to $Z$.\par
First, let us note that, by Bertini's theorem, such a section always exists. We have \begin{eqnarray*}\c1(\L)[\Z\to \X, g]&=&\c1(\L)[\Z\to \X]+a(\h^{-1}(\L)\wedge g)\\
&=&\c1(\L)j_*[\Z\to \Z]+a(\h^{-1}(\L)\wedge g)\\
&=&j_*[[\Z'\ds i \Z]-a(\g(j^*\L)\log\|j^*s\|^2)-a(i_*[\wt{\g}(Z'/Z)\g^{-1}(\ovl T_{Z'})])]+a(\h^{-1}(\L)\wedge g)\\
&=&[\Z'\to \X]-\g(\L)a(\log\|s\|^2j_*(\g^{-1}(\ovl{T_Z})\g(j^*\ovl{T_X}))))\\
& &-a\left[j_*(i_*\wt\g(Z'/Z)\g^{-1}(\ovl T_{Z'}))\g^{-1}(\ovl T_Z)\g(j^*\ovl{T_X})\right]+a(\h^{-1}(\L)\wedge g)\\
&=&[\Z'\to \X]-\g(\L)a(\log\|s\|^2j_*(\g(\ovl{N_{Z/X}})))-\g(\L)a(\log\|s\|^2j_*(dd^c\wt{\g}( Z/X)\g^{-1}(\ovl{T_{Z}})))\\
& &+a(j_*(i_*\wt{\g^{-1}}(Z'/Z)\g(\ovl N_{Z'/X})\g(j^*\ovl T_X))))+a(\h^{-1}(\L)\wedge g)\\
&=&[\Z'\to \X]-\g(\L)a(\log\|s\|^2j_*(\g(\ovl{N_{Z/X}})))\\
& &-a(\g(\L)j_*(\wt{\g}( Z/X)\g^{-1}(\ovl{T_{Z}})))-a(\h^{-1}(L)j_*(\wt{\g}( Z/X)\g^{-1}(\ovl{T_{Z}})))\\
& &+a(j_*(i_*\wt{\g^{-1}}(Z'/Z)\g(\ovl N_{Z'/X})\g(j^*\ovl T_X))))+a(\h^{-1}(\L)\wedge g)
\end{eqnarray*}
Where we've used the fact that for the composition of regular immersion we have an exact sequence $$0\to N_{Z'/Z}\to N_{Z'/X}\to i^*N_{Z/X}\to 0$$
replacing $g$ by the expression given by $\SECT$ and using that $N_{Z/X}=j^*L_1$ yields the desired formula.
\end{proof}
\begin{Rque}
In the preceding formula we can see that there is a variable part depending on the metrics chosen on the different strata of $Z_1\cap Z_2\subset Z_1 \subset X$ and a fixed part depending only on the metrics chosen on the bundles. For two hermitian line bundles over $X$, we set $$(\L_1,\L_2)_X=-a(\log\|s_2\|^2\g(\L_2)\g(\L_1)\delta_{Z_1})+a(\h^{-1}(\L_2)\log\|s_1\|^2\g(\L_1))\in \cob(X)$$
inductively we can define $$(\L_1,..., \L_p)_X=\h^{-1}(\L_p)(\L_1,..., \L_{p-1})_X+\log\|s_p\|^2\g(\L_p)\delta_{Z_1\cap...\cap Z_{p-1}}$$
for a family of ample line bundles over $X$, and $s_p$ of $L_p$ over $X$ transverse to $Z_1\cap...\cap Z_{p-1}$.\par
This bracket is easily seen to be symmetric in the $L_i$'s, moreover, if we can find metrics on the $Z_i$'s such that the different strata of the intersection of divisors $\bigcap Z_i$ can be endowed with metrics rendering all the associated exact sequences meager then $$[\X\to \X, \L_1,...,\L_p]=[\ovl{\bigcap Z_i}\to \X]+(\L_1,..., \L_p)_X$$
If $f$ is projective morphism from $X$ to $Y$ we will denote $(\L_1,..., \L_p)_Y$ for $f_*(\L_1,..., \L_p)_X$, and $(\L_1,..., \L_p)$ for $(\L_1,..., \L_p)_{\Spec k}$
\end{Rque}
We thus have a formula for the class of the inclusion of a smooth subscheme given as a local complete intersection of smooth divisors. We can easily reduce the case of arbitrary intersection of divisors to this case, of course the formula obtained is more complicated.
\begin{Prop}
Let $\L'_1,...,\L'_r$ be arbitrary hermitian line bundles over an arithmetic variety $\ovl X$. For each $0\leq i \leq r$ there exists very ample hermitian line bundles $\L_i, \M_i$ such that $$[\X\to \X, \L'_1,...,\L'_r]=\sum_{i_1,j_1,...,i_r,j_r}a_{i_1,j_1}...a_{i_r,j_r}[\X\to \X, \L_{1},...,\L_1,...,\M_r^\vee,...,\M_r^\vee]$$
where the bundle $\L_k$ (resp. $\M_k$) is repeated $i_k$ (resp $j_k$) times.
\end{Prop}
\begin{proof}
As $X$ is projective over $k$, every bundle $L'_i$ can be written as $L_i\otimes M_i^\vee $ where $L_i$ and $M_i$ are very ample line bundles. Let's equip either one of them, say $L_i$ with an arbitrary metric and the other one, $M_i$ with the metric that turns the isomorphism $M_i\simeq L_i\otimes {L'}_i^{\vee}$ into an isometry.\par
The proposition readily follows. 
\end{proof}

\begin{Rque}
In the previous formula, the class $$[\X\to \X, \L_1,...,\L_1, \M_1^\vee,...,\M_1^\vee,...,\M_r^\vee,...,\M_r^\vee]$$ can be computed via \ref{Transverse} and the observation that $\M^\vee$ is the inverse of $\M$ for the formal law $\mbb F$, and thus $[\X\to \X, \L_1,...,\L_r, \M^\vee]=\chi(\c1(M))[\X\to \X, \L_1,...,\L_r]$ and this enables us to define the bracket $(\L_1,..., \L_p)_X$ for any family of hermitian line bundles.

\end{Rque}
\begin{Cor}
As an $\mbb L$-module, the group $\cob(X)$ is generated by purely geometric classes and analytic classes.
\end{Cor}

\subsection{Structure of $\cob(k)$}
Let's start by a basic observation 
\begin{Prop}
Let $u$ be any element in $k^*$, in $\cob(k)$ we have $a(-\log|u|)=0$
\end{Prop}
\begin{proof}
It suffices to consider the trivial line bundle over $\Spec k$, with metric $|u|^2$ (i.e, if $x$ and $y$ are lying in the line $\mcl O_k(\mbb C); x.y=|u|^2x\ovl{y}$), which will be denoted by $\ovl{\mcl O_k^u}$. Multiplication by $u$ induces an (algebraic) isometry for the trivial bundle with trivial metric to $\ovl{\mcl O_k^u}$, thus the classes $[\Spec k\to \Spec k, \mcl O_k]$ and $[\Spec k\to \Spec k, \ovl{\mcl O_k^u}]$ are equal. The first one is equal to $[\emptyset \to \Spec k]-a(\log|1|^2)=0$, and the second one is equal to $[\emptyset \to \Spec k]-a(\log|u|^2)$.
\end{proof}
\begin{Rque}
Let us draw the attention of the reader on the fact that over a point it is always possible to render an exact sequence of vector space $0\to V\to V'\to V''\to 0$ meager by an appropriate choice of metrics because it is obviously holomorphically split! Moreover if two of the three vector spaces appearing in this exact sequence are already equipped with metrics, it is possible to endow the last one with a metric rendering the short exact sequence meager. 
\end{Rque}
\begin{Cor}
We have a surjective arrow of groups $$\sideset{}{'}\prod^{}_{\tau: k \hookrightarrow \mbb C}\mbb R/(\sum_{f\in k^*} \mbb Q \log |\tau f|)\to \cob(k)_{-1,\mbb Q}$$
as well as a global exact sequence $$\sideset{}{'}\prod^{}_{\tau: k \hookrightarrow \mbb C}\mbb L_{\mbb R}/(\sum_{f\in k^*}\mbb L \log |\tau f|)\to \cob(k)\to \mbb L\to 0$$
\end{Cor}

A word of warning: what we denoted somehow sloppily $\sideset{}{'}\prod_{\tau: k \hookrightarrow \mbb C}$ designs in fact the product over all real embeddings of $k$ in $\mbb C$ as well as pair of complex conjugate ones for the non real ones. Every time this notation appears that's how it should be understood.
\begin{Rque}In the geometric case we have $\Omega(k)\simeq \mbb L$, for any field $k$ that admits a resolution of singularities, that is for instance any field of characteristic zero. This is a fundamental difference with the arithmetic theory developed here. The fact that $k$ is a number field is used in a crucial manner to obtain the preceding corollary.\par
Notice also that $\cob(k)$ already depends on the number field $k$ in a manner that is common in Arakelov theory, in fact the presence of the $\log(f)$ is some kind of artifact due to the fact that we work over fields instead of ring of integers, if we were able to define $\cob(\mbb Z)$ then we would expect that the $\log f$'s in the preceding formula should disappear.
\end{Rque}
\begin{Rque}
We will later see that $$\sideset{}{'}\prod_{\tau: k \hookrightarrow \mbb C}\mbb R/(\sum_{f\in k^*} \mbb Q \log |\tau f|)\to \cob(k)_{-1,\mbb Q}$$ is in fact an isomorphism. To compute higher degree terms we need a better understanding of the relationship with higher $\text{MGL}^{2n-1,n}(k)$.
\end{Rque}

Due to the fact that we chose to work with weak groups we will not be able to define a ring structure on $\cob(X)$ in full generality, however it is possible to define a ring structure on $\cob(k)$ and to deduce a module structure over $\cob(k)$, on $\cob(X)$. To understand that ring structure we need the following proposition that is closely related to the Riemann-Roch theorem of Hirzebruch
\begin{Prop}(Hirzebruch-Riemann-Roch)\label{HRR}\par
Let $X$ be a projective smooth variety over a field $k$, and assume that we have equipped $X$ with an arithmetic variety structure, we have in $\cob(k)$, $$\int_X\g^{-1}(\ovl T_X)a(1)=\ell(X)a(1)$$
here $\ell$ denotes the composition $$\cob(k)\ds \zeta \Omega(k) \simeq \mbb L$$ (recall that the isomorphism $\Omega(k) \to \mbb L$ is canonical).
\end{Prop}
\begin{proof}
This is essentially a combinatorics proof. Notice that it will be sufficient to prove this result in $\DL(k)$ and to push this identity forward in $\cob(k)$ via the map $a$, notice also that the left hand side does not depend on the metric on $T_X$ because of Stokes formula. So we'll prove the identity $$\int_X\g^{-1}(T_X)=\ell(X)$$ in $\DL(k)$. Notice that we can replace $\g^{-1}(\ovl T_X)$ by $\g^{-1}(\ovl T_X)^{\{d_X\}}$ where $\{n\}$ denotes the degree $n$ part.\par
Let us first prove the result for projective spaces, if $X=\P^r$ we have the following Euler exact sequence $$0\to \O_X\to \O(1)^{r+1}\to T_{X}\to 0$$ and thus $\g^{-1}(T_X)=\g^{-1}(\O(1))^{r+1}$.\par
Now in view of Mishenko's formula \ref{Mishenko}, it will be sufficient to prove that $$\left(\g^{-1}(\O(1))^{r+1}\right)^{\{r\}}=(r+1)\mfr h(\O(1))^{\{r+1\}}$$ and this results from the following
\begin{Lem}(Lagrange Inversion formula)\par

We have $$\left(\frac{1}{\g(u)^{r+1}}\right)^{\{r\}}=(r+1)\h(u)^{\{r+1\}}$$
\end{Lem}
\begin{proof} This is \cite[Thm 5.4.2, p. 38]{Combinatorics}
\end{proof}
Let us turn to the case of the product of projective spaces $X=\P^{r_1}\times...\times \P^{r_k}$. We have $$\g^{-1}(T_X)=\g^{-1}\left(\bigoplus_{i=1}^k p_i^* T_{\P^{r_i}}\right)=\prod_{i=1}^k\g^{-1}(p_i^*T_{\P^{r_i}})=\prod_{i=1}^kp_i^*\g^{-1}(T_{\P^{r_i}})$$
therefore $$\int_X\g^{-1}( T_X)=\int_X\prod_{i=1}^k p_i^*\g^{-1}(T_{\P^{r_i}})=\prod_{i=1}^k\int_{\P^{r_i}}\g^{-1}(T_{\P^{r_i}})=[\P^{r_1}]...[\P^{r_k}]=[\P^{r_1}\times...\times \P^{r_k}]$$
thus the results holds for a product of projective spaces.\par
Now, $\mbb L_{\mbb Q}$ is a polynomial ring over $\mbb Q$ generated by the projective spaces, as the right hand side of the formula we want to prove is $\mbb Q$-linear, it will be sufficient to prove that $$\int_X\g^{-1}( T_X)=\sum \alpha_{(r_1,..,r_k)} \int_{\P^{r_1}\times...\times \P^{r_k}}\g^{-1}( T_{\P^{r_1}\times...\times \P^{r_k}})$$
as soon as we have $$[X]=\sum \alpha_{(r_1,..,r_k)} [\P^{r_1}]...[\P^{r_k}]$$
in $\Omega(k)_{\mbb Q}$.\par
Recall, \cite{Milnor}, that two complex manifolds are in the same cobording class if and only if they share the same Chern numbers, $C_I(X)=\int_X c_{i_1}(T_X)...c_{i_p}(T_X)=\int_X c_I(X)$ for $I=(i_1,...,i_p)$ any partition of $d_X=i_1+...+i_p$.\par
Moreover we know that a complex embedding $\sigma:k\to \mbb C$ induces an isomorphism $\Omega_\bullet(k)\simeq \text{MU}_{2\bullet}$, \cite[Cor 1.2.11]{LevineMorel}, we see that two algebraic varieties over $k$ are rationally cobording iff their complex points share the same Chern numbers after any complex embedding. Let $a_1,...,a_d$ be the Chern roots of $T_X$, we see that $$\g^{-1}(T_X)^{\{d\}}=(\prod_i \g^{-1}(a_i))^{\{d\}}=(\prod_i(\sum_k t_k a_i^k)^{-1})^{\{d\}}=\sum_I v_{d,I}(t)c_I(X)$$
where $v_{d,I}(t)$ is a universal polynomial in (a finite number of) the $t_i$'s depending only on the dimension of $X$. Therefore $$\int \g^{-1}(T_X)=\int_X \sum_I v_{d,I}(t)c_I(X)=\sum_I v_{d,I}(t)C_I(X)$$
Now, as $C_I(X)=\sum_J \alpha_J C_I(\P^J)$ for $J$ some multi-indices of length $d$, we get \begin{eqnarray*}\sum_I v_{d,I}(t)C_I(X)&=&\sum_I v_{d,I}(t)\sum_J \alpha_J C_I(\P^J)\\
&=&\sum_J \alpha_J \sum_I v_{d,I}(t)C_I(\P^J)\\
&=&\sum_J \alpha_J\int_{\P^J}\g^{-1}(T_{\P^J})\\
&=&\sum_J \alpha_J\ell(\P^J)=\ell(X)\end{eqnarray*}
and the results follows.
\end{proof}
This result is the key ingredient that will enable us to show that we have a $\cob(k)$-module structure on $\cob(\X)$. In fact the previous result admits the following generalization
\begin{Cor}\label{HRRC}
We have in $\cob(k)$
$$\int_X\h^{-1}(\L_1)...\h^{-1}(\L_p)\g^{-1}(\ovl T_X)a(1)=\ell(\X, \L_1,...\L_p)a(1)$$
\end{Cor}
\begin{proof}
Notice that this equality is in fact an equality in $\Omega(k)_{\mbb Q}$ which we view inside $\DL(k)$ pushed inside $\cob(k)$. Both sides are independant on all the metrics. Let us first consider the case where $L$ is very ample, then we have an exact sequence $$0\to T_Z\to i^*T_X\to i^*L\to 0$$ thus if we compute 
\begin{eqnarray*}
\int_Z\g^{-1}(T_Z)&=&\int_X i_*\g^{-1}(T_Z)\\
 &=&\int_X i_*[\g^{-1}(i^*T_X)\g(i^*L)]\\
 &=&\int_X \g^{-1}(T_X)\g(L)\delta_Z\\
&=&\int_X \g^{-1}(T_X)\g(L)c_1(L)\\
&=&\int_X \g^{-1}(T_X)\h^{-1}(L)
\end{eqnarray*}
where we used the multiplicativity of the $\g$ class, the projection formula and Poincare-Lelong formula and Stokes formula. But the previous proposition the ensures us that $$\int_Z\g^{-1}(T_Z)a(1)=\ell(Z)a(1)=\ell(X,L)a(1)$$ The general case follows by induction when all the line bundles $L_i$ are very ample. Then again, by noting that any line bundle can be written as the difference of two very ample line bundles and using that $\mfr h^{-1}(u+v)=\mbb F(\h^{-1}(u),\h^{-1}(v))$ we may conclude.
\end{proof}
\begin{Prop}(Ring and Module Structures)\par
We have a commutative $\mbb L$-algebra structure on $\cob(k)$ given by $$[\X\to \Spec k, \fii(t)\alpha]\otimes [\Y\to \Spec k, \psi(t)\beta] \mapsto [\X\times \Y\to \Spec k]+\ell(X)\psi(t)a(\beta)+\ell(Y)\fii(t)a(\alpha)$$
\par
We have a natural $\cob(k)$-module structure on $\cob(\X)$ given by \begin{eqnarray*}[\X\to \Spec k, \fii(t)\alpha]\otimes[\Z\ds f \Y, \L_1,...,\L_r,\psi(t)g]&\mapsto& [\X\times \Z\to \Y, p_2^*\L_1,...,p_2^*\L_r]\\
& &+\ell(X)\psi(t)a(g)\\
& &+\fii(t)f_*[\c1(\L_1)\circ ...\c1(\L_r)\pi_Z^*(\alpha)]\end{eqnarray*}
where $\pi_Z$ is the structural morphism of $\Z$ ($\pi_Z^*(\alpha)$ is simply the locally function function $\tau\alpha$ over each connected component $Z_\tau(\mbb C)$ of $Z(\mbb C)$).
\end{Prop}
\begin{proof}
We need to show that all these operations are well defined. We will simply denote $[\Y, \L_1,...,\L_r]$ the class $[\Y\to \Spec k, \L_1,...,\L_r]$\par
We may define an $\mbb L$-bilinear map $$\left(\mbb L\otimes \Dtt(k)\oplus (\oplus \mbb L[\Y, \L_1,...,\L_r])\right)\otimes \left(\mbb L\otimes \Dtt(k)\oplus (\oplus \mbb L[\Z, \L_1,...,\L_r])\right) \to \cob(k)$$ by the following rules
\begin{enumerate}
	\item $a(\alpha).a(\beta)=0$
	\item $a(\alpha).[\Z, \L_1,...,\L_r]=\ell([\Z, \L_1,...,\L_r])a(\alpha)$
	\item $[\Y, \L_1,...,\L_r].a(\beta)=\ell([\Y, \L_1,...,\L_r])a(\beta)$
	\item $[\Y, \L_1,...,\L_r].[\Z, \M_1,...,\M_s]=[\Y\times \Z, p_1^*\L_1,...,p_2^*\M_s ]$
\end{enumerate}
We need to prove that this multiplication structure is compatible with $\SECT(k), \DIM(k)$ and $\FGL(k)$.\par
Examining the saturation procedure we described earlier, we can give explicit generators\footnote{Recall that $\bra\SECT(k)\ket$, $\bra\DIM(k)\ket$ or $\bra\FGL(k)\ket$, are not simply the submodules generated by $\SECT(k), \DIM(k)$ and $\FGL(k)$, but the submodules obtained by the saturation process.} for the sub-$\mbb L$-module of $M=\left(\mbb L\otimes \Dtt(k)\oplus (\oplus \mbb L[\Y, \L_1,...,\L_r])\right)$ that we kill to obtain $\cob(k)$.\par
Let us do $\bra\DIM(k)\ket$ first, it is generated by classes of the form $$c=[\Z, \pi^*\L_1...,\pi^*\L_{d+2}, \M_1,...,\M_s]$$ for $\pi:\Z\to \Y$ projective and smooth with $\dim_k(Y)<d$, as $\ell(c)=0$, and as $[\X, \ovl{N}_1,...,\ovl{N}_t ].c=\pi_{X *}(\c1(\ovl{N}_1)\circ...\circ\c1(\ovl{N}_t)\pi_X^*(c))$ it is clear that the result of the multiplication by such a class will vanish in $\cob(k)$.\par
Similarily $\bra\FGL(k)\ket$ is the submodule of $M/\bra\DIM(k)\ket$ generated by classes of the form $c=\mbb F(\c1(\L),\c1(\L'))[\Y, \M_1,..,\M_r]-\c1(\L\otimes\L'))[\Y, \M_1,..,\M_r]$ and $c'=\mbb \pi_{Y*}([F(\c1(\L),\c1(\L'))a(g)]-\c1(\L\otimes\L'))a(g))$. As the multiplication of two analytic classes together vanish, and as $\ell([\Z, \L_1,...,\L_r]).c'$ clearly lies in the submodule generated by the $c'$'s the multiplication by classes of the second form is well defined.\par
For the first case, it suffices to notice that $\ell(c)=0$ in $\Omega(k)$ and that $[\X, \ovl{N}_1,...,\ovl{N}_t ].c$ equals as before $\pi_{X *}(\c1(\ovl{N}_1)\circ...\circ\c1(\ovl{N}_t)\pi_X^*(c))$ to make sure that the multiplication by such classes is well defined.\par
The generators for $\bra\SECT(k)\ket$ in $(M/\bra\DIM(k)\ket)/\bra\FGL(k)\ket$ are now given by 
\begin{eqnarray*}
c&=&[\Z, i^*\L_1,...,i^*\L_p]-[\X, \L, \L_1,...,\L_p]-\pi_{X*}[\c1(\L_1)...\c1(\L_p)a(\log\|s\|^2\g(\L))]\\
& &-\pi_{X*}[\c1(\L_1)...\c1(\L_p)a(i_*(\wt{\g(\mcl E)}\g^{-1}(\ovl{T_Z})))]
\end{eqnarray*}
An argument, very similar to the one given below for the module structure, shows that the multiplication by those classes vanishes, because of Poincare-Lelong formula and \ref{HRRC}.\\

Let's turn to the case of the module structure.\par
 We now need to check that the pairing vanishes as soon as the class on the right hand side is on of the form $\bra\SECT(\Y)\ket$, $\bra\DIM(Y)\ket$ or $\bra\FGL(Y)\ket$(the fact that is vanishes on the left hand side when the class is of the form $\bra\SECT(k)\ket$, $\bra\DIM(k)\ket$ or $\bra\FGL(k)\ket$ is just a repetition of the previous argument).\par
Let us also notice that, by the usual trick of writing a possibly non very ample line bundle as the difference of two very ample line bundles, every class in $\cob(k)$ can be written as a linear combination of classes of the form $[\Y]+a(g)$ with coefficients in $\mbb L$.\par
Finally let us notice that this module structure is compatible with all operations (push-forwards, pull-back, first Chern class) that we have defined, in an obvious sense, so we can check the vanishing on classes of the form $\SECT(\Y)$, $\DIM(\Y)$ or $\FGL(\Y)$ instead of explicit generators.\\

Let's start by $\DIM$, and let's do the multiplication by an element of the form $a(\alpha)$ first. We have $$a(\alpha)[\Z\ds f \Y, \L_1,...,\L_r]=f_*[\c1(\L_1)\circ ...\c1(\L_r)\pi_Z^*(\alpha)]$$
but this zero as soon as $r>d_Z+1$ because the action of the first Chern class increases the type of the forms by $(1,1)$. Now for the case of a product $$[\X][\Z\ds f \Y, \L_1,...,\L_r]=[\X\times \Z\to \Y, p_2^*\L_1,...,p_2^*\L_r]$$
which is also zero as soon as $r>d_Z+1$ because this is $f_*p_{2*}p_1^*[ \Z\to \Z, \L_1,...,\L_r]$ and $[ \Z\to \Z, \L_1,...,\L_r]$ is zero.\par
Concerning $\SECT$, for the multiplication by an analytic class, as the multiplication by an analytic class vanishes on analytic classes we're left with checking that \begin{eqnarray*}
0&=&\alpha\left([\Z\to \X]-[\X\to \X, \L]\right)\\
&=&\alpha i_*(1)-\c1(\L)\alpha\\
&=&\g(\L)\alpha(\delta_Z-c_1(\L))\\
\end{eqnarray*}
and Poincare-Lelong formula ensures that this vanishes up to an exact current.\par
On the other hand, let us examine $$\mu=[\Y]\left([\Z\ds i \X]-[X\to X, \L]-\log\|s\|^2\g(\L)-i_*(\wt{\g(\mcl E)}\g^{-1}(\ovl{T_Z}))\right)$$
we have
\begin{eqnarray*}
\mu&=&[\Y\times \Z \to \X]-[\Y\times \X \to \X, p_2^*\L]-\ell(Y)\g(\L)\log\|s\|^2-\ell(Y)i_*(\wt{\g(\mcl E)}\g^{-1}(\ovl{T_Z}))\\
&=&p_{2*}\left[[\Y\times \Z \ds j \Y\times \X]-[\Y\times \X \to \Y \times\X , p_2^*\L]\right]\\
& &-\ell(Y)i_*(\wt{\g}(\mcl E)\g^{-1}(\ovl{T_Z}))-\ell(Y)\g(\L)\log\|s\|^2\\
&=&p_{2*}[\log\|p_2^*s\|^2\g(p_2^*\L)+j_*(\wt{\g}(p_2^*\mcl E)\g^{-1}(p_2^*\ovl{T_Z})\g^{-1}(p_1^*\ovl{T_Y}))]
\\
& &-\ell(Y)i_*(\wt{\g}(\mcl E)\g^{-1}(\ovl{T_Z}))-\ell(Y)\g(\L)\log\|s\|^2\\
&=&\g(\L)\log\|s\|^2[p_{2*}p_2^*a(1)-\ell(Y)a(1)]-i_*(\ell(Y)\wt{\g}(\mcl E)\g^{-1}(\ovl{T_Z}))+i_*p_{2*}(\wt{\g}(p_2^*\mcl E)\g^{-1}(p_2^*\ovl{T_Z})\g^{-1}(p_1^*\ovl{T_Y}))\\
&=&\g(\L)\log\|s\|^2[p_{2*}p_2^*a(1)-\ell(Y)a(1)]-i_*[\ell(Y)\wt{\g}(\mcl E)\g^{-1}(\ovl{T_Z})-\wt{\g}(\mcl E)\g^{-1}(\ovl{T_Z})p_{2*}p_2^*a(1)]\\
&=&\left(\g(\L)\log\|s\|^2-\wt{\g}(\mcl E)\g^{-1}(\ovl{T_Z})\delta_Z\right)[p_{2*}p_2^*a(1)-\ell(Y)a(1)]
\end{eqnarray*}
and this is seen to be zero because of \ref{HRR}.\par
Let's now tackle the case of $\FGL$, an easy computation shows that $$[\X].\c1(\L)(1_{\Y})=p_{2*}\c1(p_1^*\L)[\X\times \Y],\ \ \  a(\alpha)\c1(\L)(1_{\Y})=\c1(\L)a(\pi_Y^*(\alpha))$$ and this readily implies that the multiplication by $[\X]+a(\alpha)$ of a class in $\FGL$ vanishes.\par
This shows that the two $\mbb L$-bilinear maps are well defined; the verification of the fact that these maps give a ring structure to $\cob(k)$ and a $\cob(k)$-module structure to $\cob(\X)$ is now straightforward.
\end{proof}
\begin{Rque}
It may appear surprising at first glance that classes of the form $a(\alpha).a(g)$ vanish, but this should not be so because in other (strong) arithmetic theories, the product of such classes is given by $a(\alpha\d\db g)$ which is $a((\d\db\alpha) g)$ up to something in $\im \d+\im \db$, and this is of course zero. 
\end{Rque}

\subsection{Arrows}
We will now construct arrows from the group $\cob(X)$ to the groups $\CHw(X)$ and $\Kw_0(X)$, these arrows will be compatible with the different maps we have defined between those groups. To define those we will introduce the notion of Borel-Moore functor of arithmetic type on the category of arithmetic varieties. 
\begin{Def}(Hermitian Borel Moore Functor)\label{HBM}\par

We will call a (graded) hermitian Borel-Moore functor an additive assignment $\X\to \H_\bullet(\X)$ for each arithmetic variety $\X$, such that we have
  \begin{enumerate}
	  \item $\H_\bullet(\X)$ is a (graded) $\mbb L$-module with a specified element denoted $1_{\X}$, and called the unit element,
		\item $\H_\bullet(\X)$ is equipped with an action of $\DL(X)$ denoted by $a$,
		\item (genus) a multiplicative genus $\fii\in \H(k)[[u]]$
    \item (direct image homomorphisms) a homomorphism $f_* : \H_\bullet(\X) \to \H_\bullet(\Y)$ of degree zero for each projective morphism $f : \X \to \Y$,
    \item (inverse image homomorphisms) a homomorphism $f^* : \H_\bullet(\Y) \to \H_\bullet(\X)$ of degree $d$ for each smooth equidimensional morphism $f : \X \to \Y$ of relative dimension $d$ that preserves the unit element,
    \item (first Chern class homomorphisms) a homomorphism $\c1(\L) : \H_\bullet(\X) \to \H_\bullet(\X)$ of degree -1 for each hermitian line bundle $\L$ on $\X$,
  \end{enumerate}
satisfying the axioms
  \begin{enumerate}
    \item the map $f \mapsto f_*$ is functorial;
    \item the map $f \mapsto f^*$ is functorial;
    \item if $f : \X \to \Z$ is a projective morphism, $g : \Y \to \Z$ a smooth equidimensional morphism, and the square
      $$\xymatrix{ \ovl W \ar[r]^{g'}\ar[d]_{f'}& \X \ar[d]^{f} \\
        \Y \ar[r]^{g}& \Z\\}$$
    is Cartesian, then one has 
      \[ g^* \circ f_* = f'_* \circ g'^*  \]
    \item if $f : \Y \to \X$ is projective and $\L$ is a hermitian line bundle on $\X$, then one has
      \[ f_* \circ \c1(f^*(\L)) = \c1(\L) \circ f_*  \]
    \item if $f : \Y \to \X$ is a smooth equidimensional morphism and $\L$ is a hermitian line bundle on $\X$, then one has
      \[ \c1(f^*\L) \circ f^* = f^* \circ \c1(L)  \]
    \item if $\L$ and $\M$ are hermitian line bundles on $\X$, then one has 
      \[ \c1(\L) \circ \c1(\M) = \c1(\M) \circ \c1(\L)  \]
		\item if $f : \Y \to \X$ is projective, then one has
      \[ f_* \circ a(g) = a(f_*(g \wedge \fii(\ovl{T_f})))  \]
		\item if $\L$ is a hermitian line bundle on $\X$, then one has
      \[ \c1(\L) \circ a(g)  = a(c_1(\L) \fii(\L) g) \]
      \end{enumerate}
\end{Def}
Just like for the geometric case we need to restrict the class of Borel Moore functors we'll be interested in, in order to give them an arithmetic significance.
\begin{Def}(Arithmetic Type)\par
A \emph{Hermitian Borel-Moore functor with weak product} is the data of a hermitian Borel-Moore functor together with the data of
    \begin{enumerate}
      \item a commutative $\mbb L$-algebra structure on $\H( k)$,
			\item a $\H( k)$-module structure on $\H(\X)$ compatible with its $\mbb L$-structure.
    \end{enumerate}
We will say that a hermitian Borel-Moore functor with weak product, $\H_\bullet$ is of arithmetic type if the following additional properties are satisfied
  \begin{enumerate}
    \item(Dim)  For $\X$ an arithmetic variety and $(\L_1, \ldots, \L_n)$ a family of hermitian line bundles on $X$ with $n > \dim(X)+1$, one has
      \[ \c1(\L_1) \circ \cdots \circ \c1(\L_n)(1_{\X}) = 0 \]
    in $\H_\bullet(X)$.
    \item (Sect)  For $\X$ an arithmetic variety, $\L$ a hermitian line bundle on $X$, and $s$ a section of $L$ which is transverse to the zero section, one has the equality
      \[ \c1(L)(1_{\X}) +a(i_*[\wt{\fii}(\mcl E)\fii^{-1}(\ovl T_Z)])+a(\fii(\L)\log\|s\|^2)= i_*(1_{\Z}) \]
    where $i : Z \to X$ is the closed immersion defined by the section $s$ and $\mcl E$ is the exact sequence $$0\to \ovl{T_Z}\to i^*\ovl{T_X}\to i^*\L\to 0$$
		\item  (FGL)  If $F_H$ is the formal group law defined by $\mbb L\to \H(\Spec k)$, then for $\X$ an arithmetic variety and $\L,\M$ hermitian line bundles on $\X$, one has the equality
      \[ F_H(\c1(\L), \c1(M)) = \c1(\L \otimes \M) \]
    where $F_H$ acts on $\H(X)$ via its $\mbb L$-module structure. Moreover we require the different pull-backs and push-forward maps to preserve $F_H$.
  \end{enumerate}
\end{Def}
The following theorem is a tautology
\begin{Theo}
The assignment $\X\mapsto \cob(\X)$ is the universal (weak) Borel-Moore functor of arithmetic type.
\end{Theo}
\begin{Rque}
In fact in view of (sect) the genus of an arithmetic Borel-Moore functor is completely determined by its formal group law, we can prove it from the axiom sect, but as it is already the case with $\cob$, it will be automatically the case in every such functor also. 
\end{Rque}
\begin{Rque}
We've proven in \ref{BMforCH} and \ref{ATforCH} that $\X \mapsto \CHw(\X)$ is a (weak) Borel-Moore functor of arithmetic type, its formal group law is additive, and its genus is given by $1$, which explains that $\CHw(\X)$ does not depend on the choice of hermitian structure on $X$.\par
We've also proven in \ref{BMforK} and \ref{ATforK} that $\X \mapsto \Kw(\X)$ is a (weak) Borel-Moore functor of arithmetic type with multiplicative unitary law, and the usual Todd-genus as genus.
\end{Rque}
\begin{Cor}
We have natural arrows $$\cob(\X)\to \CHw(X)\ \ \cob(X)\to \Kw_0(X)$$
that make the following diagrams commute$$\xymatrix{\DL(X)_p\ar[d]\ar[r]^{a} &\cob(\X)_{p} \ar[r]^\zeta\ar[d] &\Omega(X)_{p}\ar[d]\ar[r]& 0\\
\wt D^{d_X-p+1,d_X-p+1}_{\mbb R}(X)\ar[r]^a &\CHw_{p}(X) \ar[r]^\zeta &\CH_p(X)\ar[r]& 0}$$
and 
$$\xymatrix{\DL(X)\ar[r]^a\ar[d] &\cob(\X) \ar[r]^\zeta\ar[d] &\Omega(X)\ar[d]\ar[r]& 0\\
\wt D^{\bullet,\bullet}_{\mbb R}(X)\ar[r]^a &\Kw_{0}(\X) \ar[r]^\zeta &K_0(X)\ar[r]& 0}$$
\end{Cor}
\begin{proof}
Both arrows are uniquely determined by the choices of the formal group law, if we chose the additive one, we get a map from $\cob(\X)$ to $\CHw(X)$.\par
The arrow $\cob(\X)\to \Kw_0(\X)$ is given by the choice of the multiplicative unitary law.
\end{proof}
\begin{Cor}
We have an isomorphism $$\sideset{}{'}\prod_{\tau: k \hookrightarrow \mbb C}\mbb R/(\sum_{f\in k^*} \mbb Q \log |\tau f|)\to \cob(k)_{-1,\mbb Q}$$
\end{Cor}
\begin{proof}
We already know that the map is surjective it remains to show that it is injective, but $$\CHw_{-1}(X)\simeq \sideset{}{'}\prod_{\tau: k \hookrightarrow \mbb C}\mbb R/(\sum_{f\in k^*} \mbb Z \log |\tau f|)$$ therefore if the image of any element in $\sideset{}{'}\prod_{\tau: k \hookrightarrow \mbb C}\mbb R/(\sum_{f\in k^*} \mbb Q \log |\tau f|)$ would be zero in $\cob(k)_{-1,\mbb Q}$ then a multiple of it would be mapped to zero in $\CHw_{-1}(X)$ and $$\sideset{}{'}\prod_{\tau: k \hookrightarrow \mbb C}\mbb R/(\sum_{f\in k^*} \mbb Q \log |\tau f|)\to \CHw_{-1}(X)_{\mbb Q}$$ would not be injective, a contradiction.
\end{proof}
These results shed some light on different constructions in Arakelov theory.\par
It explains why the direct image in $K$-theory depend on a choice of metric on the varieties whereas it is possible to construct a push forward for arithmetic Chow groups without specifying any metric. This is because the Todd class of the Chow theory is $1$, and therefore the secondary forms associated to it are 0.\par
It also explains why the star product of \cite{GSA} is what it is, because the computation of the bracket $(\L,\M)_X$ reduces to the computation of the star-product $-\log\|s\|^2\star-\log\|t\|^2$
\begin{Rque}
In an upcoming paper we will show that we actualy have comparison isomorphisms $$\cob(\X)_{\mbb Z}\simeq \CHw(X), \ \cob(\X)_{\mbb Q}\simeq \Kw_0(X)_{\mbb Q}$$ refining the comparison theorems of \cite{LevineMorel}, and how we can deduce an arithmetic Riemann-Roch formula from these facts.
\end{Rque}

%% file: Weak_Arithmetic.bbl
\def\cprime{$'$}
\begin{thebibliography}{BGFiML14}
\expandafter\ifx\csname fonteauteurs\endcsname\relax
\def\fonteauteurs{\scshape}\fi

\bibitem[BGFiML12]{BFL2}
Jos{\'e}~Ignacio \bgroup\fonteauteurs\bgroup Burgos~Gil\egroup\egroup{}, Gerard
  Freixas~i \bgroup\fonteauteurs\bgroup Montplet\egroup\egroup{} et
  R{\u{a}}zvan \bgroup\fonteauteurs\bgroup Li{c{t}}canu\egroup\egroup{} :
\newblock Hermitian structures on the derived category of coherent sheaves.
\newblock {\em J. Math. Pures Appl. (9)}, 97(5)\string:\penalty500\relax
  424--459, 2012.

\bibitem[BGFiML14]{BFL}
Jos{\'e}~Ignacio \bgroup\fonteauteurs\bgroup Burgos~Gil\egroup\egroup{}, Gerard
  Freixas~i \bgroup\fonteauteurs\bgroup Montplet\egroup\egroup{} et
  R{\u{a}}zvan \bgroup\fonteauteurs\bgroup Li{c{t}}canu\egroup\egroup{} :
\newblock Generalized holomorphic analytic torsion.
\newblock {\em J. Eur. Math. Soc. (JEMS)}, 16(3)\string:\penalty500\relax
  463--535, 2014.

\bibitem[BGI]{SGA6}
P.~\bgroup\fonteauteurs\bgroup Berthelot\egroup\egroup{},
  A.~\bgroup\fonteauteurs\bgroup Grothendieck\egroup\egroup{} et
  L.~\bgroup\fonteauteurs\bgroup Illusie\egroup\egroup{} :
\newblock {\em S\'eminaire de g\'eom\'etrie alg\'ebrique du Bois-Marie,
  Th\'eorie des intersections et Th\'eor\`eme de Riemann-Roch}.

\bibitem[BGL10]{BL}
Jos{\'e}~I. \bgroup\fonteauteurs\bgroup Burgos~Gil\egroup\egroup{} et
  R{\u{a}}zvan \bgroup\fonteauteurs\bgroup Li{c{t}}canu\egroup\egroup{} :
\newblock Singular {B}ott-{C}hern classes and the arithmetic {G}rothendieck
  {R}iemann {R}och theorem for closed immersions.
\newblock {\em Doc. Math.}, 15\string:\penalty500\relax 73--176, 2010.

\bibitem[BGS]{BGS1}
J-M. \bgroup\fonteauteurs\bgroup Bismut\egroup\egroup{},
  H.~\bgroup\fonteauteurs\bgroup Gillet\egroup\egroup{} et
  C.~\bgroup\fonteauteurs\bgroup Soul\'{e}\egroup\egroup{} :
\newblock Complex immersions and arakelov geometry.
\newblock \emph{In} {\em The Grothendieck Festschrift, Vol. I}.

\bibitem[BGS88a]{BGS2}
J-M. \bgroup\fonteauteurs\bgroup Bismut\egroup\egroup{},
  H.~\bgroup\fonteauteurs\bgroup Gillet\egroup\egroup{} et
  C.~\bgroup\fonteauteurs\bgroup Soul\'e\egroup\egroup{} :
\newblock Analytic torsion and holomorphic determinant bundles i: Bott-chern
  forms and analytic torsion.
\newblock {\em Comm. Math. Phys.}, 115\string:\penalty500\relax 49--78, 1988.

\bibitem[BGS88b]{BGS22}
Jean-Michel \bgroup\fonteauteurs\bgroup Bismut\egroup\egroup{}, Henri
  \bgroup\fonteauteurs\bgroup Gillet\egroup\egroup{} et Christophe
  \bgroup\fonteauteurs\bgroup Soul{\'e}\egroup\egroup{} :
\newblock Analytic torsion and holomorphic determinant bundles. {II}. {D}irect
  images and {B}ott-{C}hern forms.
\newblock {\em Comm. Math. Phys.}, 115(1)\string:\penalty500\relax 79--126,
  1988.

\bibitem[BGV92]{BGV}
N.~\bgroup\fonteauteurs\bgroup Berline\egroup\egroup{},
  E.~\bgroup\fonteauteurs\bgroup Getzler\egroup\egroup{} et
  M.~\bgroup\fonteauteurs\bgroup Vergne\egroup\egroup{} :
\newblock {\em Heat Kernels and Dirac Operators}.
\newblock Springer, 1992.

\bibitem[Bis97]{Bismut1}
J-M. \bgroup\fonteauteurs\bgroup Bismut\egroup\egroup{} :
\newblock Holomorphic families of immersions and higher analytic torsion forms.
\newblock {\em Ast\'erisque}, 244, 1997.

\bibitem[BK92]{BismutKohler}
J-M. \bgroup\fonteauteurs\bgroup Bismut\egroup\egroup{} et
  K.~\bgroup\fonteauteurs\bgroup K\"ohler\egroup\egroup{} :
\newblock Higher analytic torsion forms for direct images and anomaly formulas.
\newblock {\em J. Alg. Geom.}, 1\string:\penalty500\relax 647--684, 1992.

\bibitem[Fal92]{FaltBook}
Gerd \bgroup\fonteauteurs\bgroup Faltings\egroup\egroup{} :
\newblock {\em Lectures on the arithmetic {R}iemann-{R}och theorem}, volume 127
  de {\em Annals of Mathematics Studies}.
\newblock Princeton University Press, Princeton, NJ, 1992.
\newblock Notes taken by Shouwu Zhang.

\bibitem[Ful86]{Fulton}
W.~\bgroup\fonteauteurs\bgroup Fulton\egroup\egroup{} :
\newblock {\em Intersection Theory}.
\newblock Springer, 1986.

\bibitem[GH94]{Griffith-Harris}
P.~\bgroup\fonteauteurs\bgroup Griffith\egroup\egroup{} et
  J.~\bgroup\fonteauteurs\bgroup Harris\egroup\egroup{} :
\newblock {\em Principles of Algebraic Geometry}.
\newblock Harvard University, 1994.

\bibitem[GiML12a]{BFL1}
Jos\'e Ignacio~Burgos \bgroup\fonteauteurs\bgroup Gil\egroup\egroup{},
  Gerard~Freixas i~\bgroup\fonteauteurs\bgroup Montplet\egroup\egroup{} et
  Razvan \bgroup\fonteauteurs\bgroup Litcanu\egroup\egroup{} :
\newblock The arithmetic grothendieck-riemann-roch theorem for general
  projective morphisms, 2012.

\bibitem[GiML12b]{BFLGGR}
Jos\'e Ignacio~Burgos \bgroup\fonteauteurs\bgroup Gil\egroup\egroup{},
  Gerard~Freixas i~\bgroup\fonteauteurs\bgroup Montplet\egroup\egroup{} et
  Razvan \bgroup\fonteauteurs\bgroup Lit{c}anu\egroup\egroup{} :
\newblock The arithmetic grothendieck-riemann-roch theorem for general
  projective morphisms, 2012.

\bibitem[GRS08]{RGS}
H.~\bgroup\fonteauteurs\bgroup Gillet\egroup\egroup{},
  D.~\bgroup\fonteauteurs\bgroup Rossler\egroup\egroup{} et
  C.~\bgroup\fonteauteurs\bgroup Soul\'e\egroup\egroup{} :
\newblock An arithmetic riemann-roch theorem in higher degree.
\newblock {\em Ann. Ins. Fourier}, 58\string:\penalty500\relax 2069--2089,
  2008.

\bibitem[GS]{GSA}
H.~\bgroup\fonteauteurs\bgroup Gillet\egroup\egroup{} et
  C.~\bgroup\fonteauteurs\bgroup Soul\'{e}\egroup\egroup{} :
\newblock Arithmetic intersection theory.
\newblock {\em Publ. Math. IHES}, 72\string:\penalty500\relax 94--174.

\bibitem[GS90]{GS1}
H.~\bgroup\fonteauteurs\bgroup Gillet\egroup\egroup{} et
  C.~\bgroup\fonteauteurs\bgroup Soul\'{e}\egroup\egroup{} :
\newblock Characteristic classes for algebraic vector bundles with hermitian
  metrics i, ii.
\newblock {\em Ann. of Math.}, 131\string:\penalty500\relax 163--303, 205--238,
  1990.

\bibitem[GS91]{ZGS}
H.~\bgroup\fonteauteurs\bgroup Gillet\egroup\egroup{} et
  C.~\bgroup\fonteauteurs\bgroup Soul\'{e}\egroup\egroup{} :
\newblock Analytic torsion and the arithmetic todd genus.
\newblock {\em Topology}, 30\string:\penalty500\relax 21--54, 1991.
\newblock with an appendix by D. Zagier.

\bibitem[GS92]{GSinv}
H.~\bgroup\fonteauteurs\bgroup Gillet\egroup\egroup{} et
  C.~\bgroup\fonteauteurs\bgroup Soul\'{e}\egroup\egroup{} :
\newblock An arithmetic riemann roch theorem.
\newblock {\em Invent. Math.}, 110\string:\penalty500\relax 473--543, 1992.

\bibitem[Kod86]{Ehresmann}
K.~\bgroup\fonteauteurs\bgroup Kodaira\egroup\egroup{} :
\newblock {\em Complex manifold and deformation of complex structures}.
\newblock Springer-Verlag, 1986.

\bibitem[LM07]{LevineMorel}
M.~\bgroup\fonteauteurs\bgroup Levine\egroup\egroup{} et
  F.~\bgroup\fonteauteurs\bgroup Morel\egroup\egroup{} :
\newblock {\em Algebraic Cobordism}.
\newblock Springer, 2007.

\bibitem[Mil60]{Milnor}
J.~\bgroup\fonteauteurs\bgroup Milnor\egroup\egroup{} :
\newblock On the cobordism ring {$\Omega ^{\ast} $} and a complex analogue.
  {I}.
\newblock {\em Amer. J. Math.}, 82\string:\penalty500\relax 505--521, 1960.

\bibitem[Mur94]{Bert}
P.~\bgroup\fonteauteurs\bgroup Murthy\egroup\egroup{} :
\newblock Zero cycles and projective modules.
\newblock {\em Ann. of Math.}, 140\string:\penalty500\relax 405--434, 1994.

\bibitem[Nov67]{Mishenko}
S.~P. \bgroup\fonteauteurs\bgroup Novikov\egroup\egroup{} :
\newblock Methods of algebraic topology from the point of view of cobordism
  theory.
\newblock {\em Izv. Akad. Nauk SSSR Ser. Mat.}, 31\string:\penalty500\relax
  855--951, 1967.

\bibitem[Och87]{Ochanine}
S.~\bgroup\fonteauteurs\bgroup Ochanine\egroup\egroup{} :
\newblock Sur les genres multiplicatifs d\'efinis par des int\'egrales
  elliptiques.
\newblock {\em Topology}, 26\string:\penalty500\relax 143--151, 1987.

\bibitem[Qui73]{Quillen}
D.~\bgroup\fonteauteurs\bgroup Quillen\egroup\egroup{} :
\newblock Higher $k$-theory i.
\newblock \emph{In} {\em Lecture Notes 341, H.Bass Ed.}, pages 85--147, 1973.

\bibitem[Sta12]{Combinatorics}
Richard~P. \bgroup\fonteauteurs\bgroup Stanley\egroup\egroup{} :
\newblock {\em Enumerative combinatorics. {V}olume 1}, volume~49 de {\em
  Cambridge Studies in Advanced Mathematics}.
\newblock Cambridge University Press, Cambridge, second \'edition, 2012.

\bibitem[Zha99]{Zha}
Yuhan \bgroup\fonteauteurs\bgroup Zha\egroup\egroup{} :
\newblock {\em A General Arithmetic Riemann-Roch Theorem}.
\newblock Th\`ese de doctorat, Harvard University, 1999.

\end{thebibliography}
